\documentclass[12pt, twoside, reqno]{amsart}
\calclayout
\usepackage[utf8]{inputenc}
\usepackage{amsmath,amssymb,amscd,amsxtra,latexsym,bbm,graphicx,epsfig,epic,
eepic,oldgerm,bm,psfrag,amsthm,mathrsfs,bm,bbm, verbatim}
\usepackage{caption}
\usepackage{subcaption}
\usepackage{booktabs}
\usepackage{tikz-cd}
\usepackage{enumitem}
\usepackage[pdftex]{hyperref}
\usepackage[noabbrev,capitalize]{cleveref}
\usetikzlibrary{shapes}
\usetikzlibrary{decorations.markings}

\usepackage{xcolor}
\hypersetup{
	colorlinks,
	linkcolor={red!50!black},
	citecolor={green!50!black},
	urlcolor={blue!80!black}
}
\usepackage{soul}

\tikzcdset{arrow style=math font}
\tikzset{cross/.style={
    cross out, draw, solid, thin, 
    minimum size=2*(#1-\pgflinewidth), 
    inner sep=0pt, outer sep=0pt
  },
  cross/.default={3}
}

\input xy
\xyoption{all}
\CompileMatrices

\usepackage{listings} 
\usepackage{xspace}  

\newcommand{\nocontentsline}[3]{}
\newcommand{\tocless}[2]{\bgroup\let\addcontentsline=\nocontentsline#1{#2}\egroup}

\renewcommand{\contentsname}{Content and organization of the paper}

\newcommand{\commentintoc}[1]{
\addtocontents{toc}{%
  \protect\hspace{6mm}
  \begin{minipage}[t]{.85\linewidth}{#1}
    \end{minipage}
    \protect\vspace{2mm}
}}

\lstdefinelanguage{Sage}[]{Python}
{morekeywords={True,False,sage,singular},
sensitive=true}
\definecolor{dblackcolor}{rgb}{0.0,0.0,0.0}
\definecolor{dbluecolor}{rgb}{.01,.02,0.7}
\definecolor{dredcolor}{rgb}{0.8,0,0}
\definecolor{dgraycolor}{rgb}{0.30,0.3,0.30}

\theoremstyle{plain}
\newtheorem{theorem}{Theorem}[section]
\newtheorem*{theorem*}{Theorem}
\newtheorem{lemma}[theorem]{Lemma}
\newtheorem{proposition}[theorem]{Proposition}
\newtheorem{corollary}[theorem]{Corollary}
\newtheorem{definition}[theorem]{Definition}
\newtheorem*{definition*}{Definition}
\newtheorem{observation}[theorem]{Observation}

\newtheorem{question}[theorem]{Question}

\theoremstyle{definition}
\newtheorem{claim}[theorem]{Claim}

\newtheorem{rmk}[theorem]{Remark}
\newenvironment{remark}
{\pushQED{\qed}\rmk} %
  {\popQED\endrmk}
\newtheorem{rmks}[theorem]{Remarks}
\newenvironment{remarks}
  {\pushQED{\qed}\rmk} %
  {\popQED\endrmks}

\newtheorem*{remark*}{Remark}
\newtheorem*{remarks*}{Remarks}

\newtheorem{exple}[theorem]{Example}
\newenvironment{example}
  {\pushQED{\qed}\exple} %
  {\popQED\endexple}

\newtheorem*{example*}{Example}
\newtheorem*{examples*}{Examples}

\theoremstyle{plain}

\def\Ker{\operatorname{ker}}
\def\im{\operatorname {im}}
\def\Symp{\operatorname{Symp}}
\def\Sympc0{\operatorname{Symp^c_0}}

\def\Ham{\operatorname{Ham}}

\def\supp{\operatorname{supp}}

\def\Flux{\operatorname{Flux}}
\def\GL{\operatorname{GL}}
\def\SL{\operatorname{SL}}
\def\PSL{\operatorname{PSL}}

\def\Aut{\operatorname{Aut}}

\def\FM{\operatorname{FM}}

\def\pp{\partial}

\def\id{\operatorname{id}}

\def\sth{\,\vert\,}
\def\del{\partial}
\def\MCG{\operatorname{MCG}}
\def\tcl{T_{\rm Cl}}
\def\tch{T_{\rm Ch}}
\def\fl{\mathrm F}
\def\fml{\widehat{\cm}}
\def\fmh{\widehat{\cm}_{\rm Ham}}
\def\Sh{\operatorname{Sh}}

\usepackage{scalerel}[2016/12/29]
\newcommand\Colon{\scaleobj{.8}{\colon}}
\def\defi{\raisebox{.3mm}{$\Colon$}\hspace{-2.3mm}=}   

\def\LSymp{\m L{\mathrm{Symp}}}  \def\LHam{\m L{\mathrm{Ham}}}   
 \def\m#1{\mathcal{#1}}  
\def\wh#1{\widehat{#1}}  \def\wt#1{\widetilde{#1}}

\def\wtsymp{\mathrm{S}\wt{\mathrm{ymp}}}

\def\cl{{\mathcal L}}
\def\pcl{\widetilde{{\mathcal L}}}
\def\cm{{\mathcal M}}

\def\cp{{\mathcal P}}

\def\C{\mathbb{C}}

\def\N{\mathbb{N}}

\def\Q{\mathbb{Q}}
\def\R{\mathbb{R}}
\def\T{\mathbb{T}}
\def\Z{\mathbb{Z}}
\def\1{\mathbbm{1}}

\def\RP{\R \mathrm P}
\def\CP{\C \mathrm P}

\newcommand{\mb}[1]{\mathbb{#1}}\def\symp{\mathrm{Symp}}
\let\oldsharp\sharp
\renewcommand*{\sharp}{\,\oldsharp\,}

\definecolor{brickred}{rgb}{0.8, 0.25, 0.33}

\begin{document}

\title{The Lagrangian flux-monodromy morphism}

\author{Jo\'e Brendel}  

\address{
D-MATH,
ETH Zürich, 
Rämistrasse 101,
8092 Zürich,
Switzerland }
\email{joe.brendel@math.ethz.ch}

\author{Jean-Philippe Chass\'e}

\address{
CRM,
Université de Montréal, 
2900 boul. Édouard-Montpetit,
H3T 1J4,
Canada }
\email{jean-philippe.chasse@umontreal.ca}

\author{R\'emi Leclercq}

\address{
Laboratoire de Mathématiques Jean Leray,
2 rue de la Houssinière,
BP 92208,
F-44322 Nantes Cedex 3,
France }
\email{remi.leclercq@univ-nantes.fr}

\date{\today}

\maketitle

\begin{center}
	\textit{Happy birthday, Octav!}
\end{center}

\begin{abstract}  
We define a new morphism on the fundamental group of the space of Lagrangian submanifolds, which takes into account the Lagrangian flux and monodromy. We study the discreteness of its image and relate this discreteness to the ($C^\infty$) topology of the Hamiltonian orbit of the Lagrangian in question. We completely describe this new morphism for Lagrangian tori in $\R^4$ and $\CP^2$ and give explicit constructions. Similar constructions allow us to study the shape invariant of the Clifford torus in $\CP^n$. Finally, we upgrade our results on Lagrangian tori in $\R^4$ from the $C^\infty$ topology to the Hausdorff one.
\end{abstract}

\addtocontents{toc}{\protect\setcounter{tocdepth}{1}}

\tocless\section{Introduction}

The \textit{flux morphism} of a symplectic manifold $(X,\omega)$ can be interpreted as measuring how much a symplectic isotopy fails to be Hamiltonian.
It is defined as follows. Given a symplectic isotopy $\Psi$ starting at the identity, one can associate to any loop $\gamma$ in $X$ the symplectic area of the cylinder swept out by $\gamma$ under $\Psi$.
Because $\omega$ is closed, this quantity only depends on the homotopy class of $\gamma$ and the homotopy class of $\Psi$ with fixed endpoints. Thus, this induces a map from the universal cover of the identity component of the symplectomorphism group of $(X,\omega)$ to the first de Rham cohomology group of $M$:
$$\Flux : \wtsymp_{0}(X,\omega) \to H^{1}(X;\mb R).$$

This map is a morphism, since the two ends of a cylinder are homologous cycles in $X$.
The following facts make the flux morphism particularly meaningful:
\begin{enumerate}[label=(F\arabic*)]
	\item Banyaga \cite{Ban97} proved that its kernel is the universal cover of the Hamiltonian group of $(X,\omega)$. Equivalently, a symplectic isotopy has vanishing flux if and only if it is homotopic relative to its endpoints to a Hamiltonian isotopy. 
	\item Ono \cite{Ono06} proved the deep fact that \textit{the flux group} $\Gamma_{\!X} \defi \Flux\big(\pi_{1}(\symp_{0}(X,\omega))\big)$ is discrete for any symplectic manifold. 
\end{enumerate}
The latter is known as the \textit{$C^{1}$ flux conjecture} because of its important topological reformulation. Indeed, it is well known that $\Gamma_{X}$ being discrete is equivalent to $\Ham(X,\omega)$ being closed in $\symp_{0}(X,\omega)$ in the $C^{1}$-topology. It is also equivalent to $\Ham(X,\omega)$ being locally path connected in that topology. 

There is a natural variant of the flux for Lagrangian submanifolds: In $\R^{2n}$, it was already considered by Gromov \cite{Gro85}, Viterbo \cite{Vit90}, and Sikorav \cite{Sik91}, among other people. The Lagrangian version of the $C^1$ flux conjecture was proved in special cases by Ono \cite{Ono07}. The Lagrangian flux has garnered some attention~---~see, for example, \cite{Sol13, EntGanMem18, SheTonVia19,ChaLec24} and \S\ref{ssec:review} below for more details~---~but many of its fundamental properties are still somewhat ill-understood, especially when compared to how much is known about the symplectic flux. One of the striking differences is that \emph{the Lagrangian flux is not a morphism}. The starting point of the present work is a natural upgrade of the Lagrangian flux map to an actual morphism and a study of the resulting morphism from various points of view. This study also illustrates that, as is often the case, the Lagrangian setting can be more subtle than the absolute case.

\subsection{The flux-monodromy morphism}
\label{sec:definitions}

The Lagrangian variant of the flux is similar to the symplectic flux morphism. Fix a closed connected Lagrangian $L$ in $(X,\omega)$, and let $\Pi = \{L_t\}_{t \in [0,1]}$ be a Lagrangian isotopy starting at $L_0=L$. 
To such a path, one can associate an element $\Flux \Pi \in H^1(L;\R)$, called the \emph{Lagrangian flux} of $\Pi$, by associating to any 1-cycle $\gamma$ of $L$ the symplectic area swept out by $\gamma$ under the isotopy. 

The Lagrangian flux is invariant under endpoints-preserving homotopies of isotopies. Restricting our attention to loops, i.e.\ to the case $L_0 = L_1$, we obtain the \textit{Lagrangian flux map}
\begin{equation}
	\label{eq:flux_intro}
	\Flux \colon \pi_1(\cl, L) \rightarrow H^1(L;\R), 
\end{equation}
where $\cl$ is the space of Lagrangians in $X$ that are \emph{Lagrangian isotopic} to $L$. This space is equipped with the $C^1$-topology~---~see \S\ref{ssec:basic} for details. As for the symplectic flux, the Lagrangian flux of an isotopy can be interpreted as measuring how much it fails to be Hamiltonian: $\Pi$ is induced by an ambient Hamiltonian isotopy of $X$ if and only if the Lagrangian flux $\Flux \{L_t\}_{t \in [0,\tau]}$ vanishes for all $\tau \in [0,1]$. 

\begin{remark}
	In contrast with the case of symplectic flux (see Fact (F1) above), it is not true that a Lagrangian isotopy with vanishing \emph{total} flux can be deformed into a Hamiltonian isotopy. Indeed, there is a Lagrangian isotopy from the Clifford torus in $(\R^4,\omega_0)$ to the Chekanov torus with vanishing total flux. However, these tori are not Hamiltonian isotopic~\cite{Che96}. On the other hand, we prove the analogue of Fact (F1) for Lagrangian \emph{immersions} in Appendix~\ref{app:Lag_homotopies}.
\end{remark}

At the outset of this paper, let us highlight another substantial difference between the Lagrangian and symplectic versions of the flux:

\begin{observation}
	The Lagrangian flux map \eqref{eq:flux_intro} is not a morphism. 
\end{observation}

This comes from the fact that, \textit{seen as 1-cycles in $L$}, the end loops of a cylinder in $M$ need not be homologous, as illustrated by the following example.

\begin{example}
  \label{exple:Flux_not_morphism}
Let $X$ be the unit 2-sphere equipped with the standard area form. Denote by $S^1_h = \{z = h\}$ the circle of constant height $h \in (-1,1)$, and fix a Lagrangian $L = S^1_a$ for some $a \in (0,1)$. Let $\Lambda \in \pi_1(\cl)$ be the loop obtained by concatenation of the Lagrangian deformation $t \mapsto S^1_{(1-2t)a}$ for $t \in [0,1]$, which deforms $S^1_a$ to $S^1_{-a}$, with a rotation of the sphere mapping $S^1_{-a}$ back to $S^1_a$. For $\xi$ a generator of $H_1(L)$, we obtain $\Flux \Lambda \cdot \xi = 4\pi a$, the area of the cylinder bounded by $S_a^1$ and $S_{-a}^1$. This is because the rotation is realized by a Hamiltonian isotopy and thus has vanishing flux. However, $\Lambda$ maps $\xi$ to $-\xi$, and hence 
\begin{equation*}
	\Flux (\Lambda * \Lambda) \cdot \xi
	= \Flux \Lambda \cdot \xi + \Flux \Lambda \cdot (-\xi)
	= 4\pi a - 4\pi a =0. \hfill\qedhere
\end{equation*}

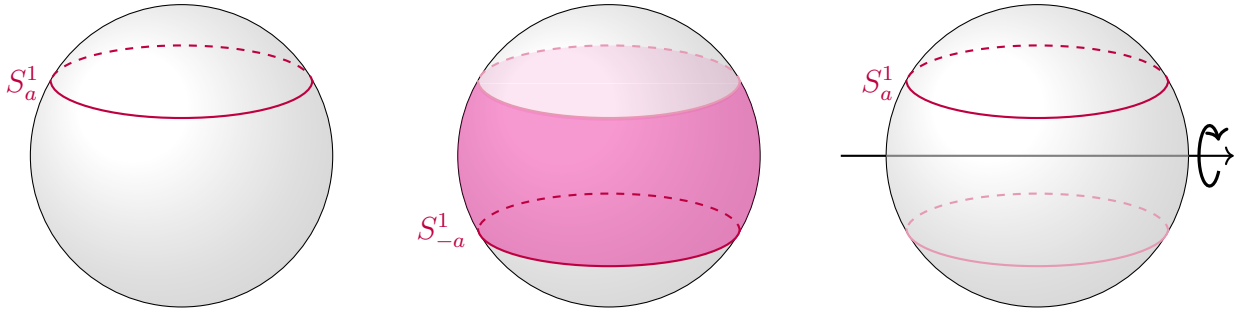
\begin{figure}[ht]
	\centering
	\begin{tikzpicture}[scale=2]
		\shade[ball color=white,opacity=0.25](0,0)circle (1);
		\draw(0,0)circle(1);
		\draw[thick,purple](0.866,0.5)arc(0:180:0.866cm and -0.25cm);
		\draw[thick,dashed,purple](0.866,0.48)arc(0:180:0.866cm and 0.25cm) node[anchor=east] {$S^1_a$};
	\end{tikzpicture}
	\qquad
	\begin{tikzpicture}[scale=2]
		\shade[ball color=white,opacity=0.25](0,0)circle (1);
		\draw(0,0)circle(1);
		
		\fill[magenta,opacity=0.4] (0.866,0.5)arc(0:180:0.866cm and -0.25cm) -- (-0.866,0.5)arc(150:210:1) -- (-0.866,-0.48)arc(180:0:0.866cm and -0.25cm) -- (0.866,-0.5)arc(330:390:1);
		\fill[magenta,opacity=0.1] (0.866,0.48)arc(0:180:0.866cm and -0.248cm) -- (0.866,0.48)arc(0:180:0.866cm and 0.25cm);
		
		\draw[thick,purple!40](0.866,0.5)arc(0:180:0.866cm and -0.25cm);
		\draw[thick,purple!40,dashed](0.866,0.48)arc(0:180:0.866cm and 0.25cm);
		\draw[thick,purple](0.866,-0.48)arc(0:180:0.866cm and -0.25cm);
		\draw[thick,purple,dashed](0.866,-0.5)arc(0:180:0.866cm and 0.25cm) node[anchor=east] {$S^1_{-a}$};
	\end{tikzpicture}
	\qquad
	\begin{tikzpicture}[scale=2]
		\shade[ball color=white,opacity=0.25](0,0)circle (1);
		\draw(0,0)circle(1);
		\draw[thick,purple](0.866,0.5)arc(0:180:0.866cm and -0.25cm);
		\draw[thick,purple,dashed](0.866,0.48)arc(0:180:0.866cm and 0.25cm) node[anchor=east] {$S^1_{a}$};
		\draw[thick,purple!40](0.866,-0.48)arc(0:180:0.866cm and -0.25cm);
		\draw[thick,purple!40,dashed](0.866,-0.5)arc(0:180:0.866cm and 0.25cm);
		\draw[->,thick] (1,0) -- (1.3,0);
		\draw[thick,gray] (-1,0) -- (1,0);
		\draw[thick] (-1.3,0) -- (-1,0);
		\draw[->,very thick] (1.2,-0.1) arc (330:30:0.07cm and 0.2cm);
	\end{tikzpicture}
\caption{The Lagrangian loop $\Lambda$ from \cref{exple:Flux_not_morphism}.}
\label{fig:sphere-ex}
\end{figure}
\end{example}

This example suggests that the defect of the flux map \eqref{eq:flux_intro} can be fixed by incorporating information about the action of $\Lambda$ on $L$. To that end, we define the \emph{monodromy} $\cm(\Lambda) \in \MCG(L)=\pi_0(\operatorname{Diff}(L))$ of $\Lambda \in \pi_1(\cl)$ as the class of the self-map of $L$ induced by $\Lambda$.

\begin{definition}
\label{def:flux_monodromy}
Let $L \subset X$ be a closed connected Lagrangian and $\cl$ its connected component in the space of Lagrangians. We define the \emph{Lagrangian flux-monodromy morphism} by 
\begin{align*}
	\FM = \FM_L \colon \pi_1(\cl) & \rightarrow H^1(L;\R) \rtimes \MCG(L)\\
	\Lambda & \mapsto (\Flux \Lambda, \cm(\Lambda)),
\end{align*}
where $\MCG(L)$ acts on $H^1(L;\R)$ by pullback. 
\end{definition}

\subsection{Discreteness of the flux-monodromy group}
\label{sec:discr-flux-group}

In contrast to the symplectic case, not only is the image of the Lagrangian flux map \eqref{eq:flux_intro}
	\begin{equation*}
		\fl(L) = \Flux\big(\pi_1(\cl, L)\big) \subset H^1(L;\R)
	\end{equation*} 
not a subgroup of $H^1(L;\R)$, but it also need not be discrete in general. Indeed, $\fl(L)$ is known to be non-discrete for many product tori in $\R^{2n}$ when $n \geqslant 3$ by \cite{Che96}.
Based on Chekanov's work, \cite[Theorem 1]{ChaLec24} provides Lagrangian tori whose associated flux map accumulates at $0 \in H^1(L;\mb R)$ in any symplectic manifold of dimension $2n \geqslant 6$. Examples in certain 4-manifolds can be constructed using, for example, \cite[Theorem 1.5]{CheSch16} or symmetric probes, as in \cite{Bre23,BreKim23}. 

\begin{definition}
	The \emph{flux-monodromy group} of $L$ is defined as
	\begin{equation}
		\Gamma(L) = \FM_L(\pi_1(\cl)) \subset H^1(L;\R) \rtimes \MCG(L).
	\end{equation} 
\end{definition}

As follows from \cref{thm:A} below, this group turns out to be discrete in all examples of the above paragraph. It is therefore natural to ask whether this is always true.

\begin{question}
	Is $\Gamma(L) \subset H^1(L;\R) \rtimes \MCG(L)$ discrete?
\end{question}

The answer to this question depends on the group of periods $\omega(H_2(X;\mb Z))$ and, more particularly, on two of its subgroups:
\begin{align}
	\label{eq:periodsrestricted}
	\begin{split}
		\mathcal{P}(X) &= \left\{ \langle \omega , S \rangle \in \R \sth S \in \ker (c_1 \colon H_2(X;\mb Z) \rightarrow \Z) \right\}, \mbox{ and}\\
		\mathcal{P}_S(X) &= \left\{ \langle \omega , S \rangle \in \R \sth S \in \ker (c_1\circ h_2 \colon \pi_2(X) \rightarrow \Z) \right\}
	\end{split}
\end{align}
with $h_2 \colon \pi_2(X) \to H_2(X;\mb Z)$ being the second Hurewicz map.
As usual, $c_1 \in H^2(X;\mb Z)$ denotes the first Chern class of $(TX,J)$ for any $\omega$-compatible almost complex structure $J$.

\begin{theorem}
	\label{thm:A}
	If the flux-monodromy $\Gamma(L)$ is discrete for all Lagrangians $L \subset X$, then the group $\mathcal{P}_S(X)$ is discrete. Conversely, if $\mathcal{P}(X)$ is discrete, then $\Gamma(L)$ is discrete for all Lagrangians $L \subset X$.
\end{theorem}

We suspect that the discreteness of $\Gamma(L)$ is actually equivalent to that of $\cp(X)$, but we were unable to prove it. Note that the first claim in \cref{thm:A} relies on work by Chekanov--Schlenk \cite{CheSch16}.

Combining information on both the flux set $\fl(L)$ and the flux-monodromy group $\Gamma(L)$ of a given Lagrangian $L$ yields interesting observations. Assume, for example, that 0 is not isolated in $\fl(L)$ but that $\Gamma(L)$ is discrete. This means that the smaller the flux of a loop in $\pi_1(\cl,L)$ is, the more complicated is their monodromy. This happens, for example, in the case of product tori in $\R^{2n}$ for $n \geqslant 3$, as proved by Chekanov~\cite{Che96}. Here is a 4-dimensional example where this occurs.

\begin{example}
	\label{ex:S2S2}
	Let $X_a$ be $S^2 \times S^2$ equipped with $(a\,\omega_{S^2}) \oplus \omega_{S^2}$, with $a \in [1,+\infty)$. Note that if $a \notin \Q$, the group of periods $\omega(H_2(X_a;\mb Z))$ of $X_a$ is not discrete while  $\m P(X_a)$ is, since $\ker c_1 \subset H_2(X_a)$ has rank one. Therefore, \cref{thm:A} ensures discreteness of $\Gamma(L)$ for any Lagrangian $L$ in $X_a$.
	When $a \neq 1$, it was shown in \cite{BreKim23} that there are product tori $L \subset X_a$ for which there is a sequence of loops $\{\Lambda_i\}\subset \pi_1(\cl,L)$ with $\Flux \Lambda_i \rightarrow 0$. Since $\Gamma(L)$ is discrete, the corresponding sequence of monodromies $\{\cm(\Lambda_i)\}$ is unbounded. 
\end{example}

\subsection{Topological consequences}
\label{sec:topol-cons}

The relevant analogue of the Hamiltonian group in the Lagrangian setting is the orbit $\LHam(L)$ of a Lagrangian $L$ under the action of $\Ham(X)$ on the space of Lagrangians $\cl$, that is
$$\LHam(L) \defi \{ \phi(L) \;|\; \phi\in \Ham(X) \}.$$
We also consider the orbit $\LHam(\iota)$ of a Lagrangian \emph{embedding} $\iota:L\hookrightarrow X$ in the space $\pcl$ of Lagrangian embeddings $L\hookrightarrow X$ that are homotopic through Lagrangian embeddings to $\iota$.

On the one hand, the flux set of $L$ is connected to topological properties of $\LHam(L)$ in $\cl$. On the other hand, the discreteness of the flux-monodromy group gives information on the parametrized objects, as shown by the following two propositions.

\begin{proposition} \label{prop:FM-top}
  If the flux-monodromy group $\Gamma(L)$ of $L$ is discrete, then the Hamiltonian orbit $\LHam(\iota)$ of $\iota$ is $C^1$-locally path connected and $C^1$-closed in $\pcl$.

  Conversely, suppose that the map $\iota_*:H_1(L;\R)\rightarrow H_1(X;\R)$ induced by $\iota$ is injective. If $\LHam(\iota)$ is either $C^1$-locally path connected or $C^1$-closed in $\pcl$, then $\Gamma(L)$ is discrete.
\end{proposition}

\begin{proposition} \label{prop:Flux-top}
	If 0 is isolated in the flux set $\fl(L)$, then $\LHam(L)$ is $C^1$-locally path connected and $C^1$-closed in $\cl$. 
	
	Conversely, suppose either that $\iota_*:H_1(L;\R)\rightarrow H_1(X;\R)$ is injective or that $L$ is a product torus in $\C^n$. If $\LHam(L)$ is either $C^1$-locally path connected or $C^1$-closed in $\cl$, then $\fl(L)$ is discrete. 
\end{proposition}

\begin{remark} \label{rem:0-isolated_no-monodromy}
  \begin{enumerate}
    \item The condition that $H_1(L;\R)\to H_1(X;\R)$ be injective ensures that there cannot be any monodromy at the homological level.
      Therefore, in that case, the flux map is actually a morphism. In particular, 0 is isolated in $\fl(L)$ if and only if $\fl(L)$ is discrete.
    \item There is no hope to prove the converse statement of~\cref{prop:Flux-top} without hypotheses on $\iota_*$: some Lagrangians have dense flux set $\fl(L)$, but $C^1$-locally path-connected Hamiltonian orbit; see \cref{rem:flux-set_loc-connected_CP^2}.\qedhere
  \end{enumerate}
\end{remark}

We note that we can actually prove discreteness of the flux set in a few important cases, under the same condition on $\iota$ as above:
\begin{proposition} \label{prop:few-cases_discreteness}
	Let $L$ be a Lagrangian such that the inclusion induces an injection $H_1(L;\R)\hookrightarrow H_1(X;\R)$. Moreover, suppose that one of the following holds:
	\begin{enumerate}[label*=(\alph*)]
		\item $L$ is weakly exact;
		\item $L$ is a non-displaceable torus;
		\item $X$ is $H$-monotone, i.e.\ monotone over $H_2(X;\Z)$.
	\end{enumerate}
	Then, $\fl(L)$ is discrete, and $\LHam(L)$ is both $C^1$-locally path connected and $C^1$-closed in $\cl$. Moreover, in the last case, $\fl(L)=0$.
\end{proposition}
In particular, this means that, when $X$ is $H$-monotone, its symplectic flux group $\Gamma_{X}$ is equal to the image of the Lagrangian flux group of the diagonal $\Delta\subseteq X\times X$ under the isomorphism $\iota^*:H^1(\Delta;\R)\xrightarrow{\sim} H^1(X;\R)$.

From Proposition~\ref{prop:Flux-top}, discreteness of $\fl(L)$ ensures that $\LHam(L)$ is locally path connected and closed in $\cl$ in the $C^1$ topology, for any Lagrangian torus $L$ in $\C^2$. Using techniques inspired by~\cite{ChaLec24}, we can prove the much stronger following statement.
\begin{theorem} \label{thm:B_Hausdorff}
	Let $L$ be a Lagrangian torus in $\C^2$. Its Hamiltonian orbit $\LHam(L)$ is locally path connected and closed in $\cl$ in the Hausdorff topology.
\end{theorem}

In fact, we prove a stronger form of local connectivity: \textit{The path ensuring local path-connectedness can be realized by a Hamiltonian isotopy}; see \cref{cor:local-connected_2-tori} for a precise statement.
	This extends \cite[Theorem~2]{ChaLec24} to not-necessarily $H$-rational tori in $\C^{2}$.

\subsection{From monodromy to flux-monodromy}
\label{ssec:mtofm}
To determine the image $\Gamma(L)$ of the flux-monodromy morphism for a given Lagrangian $L \subset X$, it is often useful to understand the group of all possible monodromies first. For this discussion, it is convenient to consider the \emph{homological Lagrangian monodromy} group
\begin{equation*}
	{\cm}^\Z(L) 
	=
	\{ (\phi^\Lambda)_{*} \in \Aut(H_1(L;\Z)) \sth \Lambda \in \pi_1(\cl,L) \},
\end{equation*}
and the \emph{extended homological Lagrangian monodromy} group
\begin{equation*}
	\widehat{\cm}^\Z(L) 
	=
	\{ \Phi^\Lambda \in \Aut(H_2(X,L;\Z)) \sth \Lambda \in \pi_1(\cl,L) \},
\end{equation*}
where the automorphisms $(\phi^{\Lambda})_{*}$ and $\Phi^\Lambda$ are the ones induced by the Lagrangian loop $\Lambda$~---~see \S\ref{ssec:basic} for details.

\begin{question}
	\label{q:realizability}
	Can any $\Psi \in \Aut (H_2(X,L;\Z))$ that
		\begin{enumerate}
			\item preserves the Maslov class, $\Psi^*m_L = m_L$, and
			\item acts by the identity on the subspace $\im (j \colon H_2(X) \rightarrow H_2(X,L))$ of ambient elements
		\end{enumerate}
	be realized as the extended monodromy $\Phi^{\Lambda} = \Psi$ of some $\Lambda \in \pi_1(\cl,L)$?
\end{question}

Those two conditions are obviously necessary. Similarly to \cite[Proposition 5.1 (a)]{Che96}, which answers positively for product tori in $\R^{2n}$, we prove the following.

\begin{theorem}
	\label{thm:realizability_intro}
	The answer to \cref{q:realizability} is positive for the Clifford torus in $\CP^n$, and thus, for every Lagrangian torus that is Lagrangian isotopic to it.
\end{theorem}

\begin{remarks}
  \begin{enumerate}
  \item If $\pp \colon H_2(X,L) \rightarrow H_1(L)$ is surjective, then $\widehat{\cm}^\Z(L)$ determines the group of homological monodromies $\cm^\Z(L)$. 
  \item Furthermore, if $X$ is (homologically) monotone, then it turns out that $\cm^\Z(L)$ determines the flux-monodromy group $\Gamma(L)$; see \cref{prop:flux-monodromy_monotone}. 
  \item Every Lagrangian torus in $\C^2$ and $\CP^{2}$ is Lagrangian isotopic to the monotone product torus. This deep fact from \cite{GooIvrRiz16} allows us to get the general description from the monotone case through a basepoint-change formula; see \cref{prop:fm_basic}. \qedhere
  \end{enumerate}
\end{remarks}

When $n=2$, we get a completely explicit description of the flux-monodromy group $\Gamma(L)$ of any Lagrangian torus $L$. In order to state these results, we fix an appropriate basis of $H_1(L)$ and its dual basis in cohomology.

\begin{theorem}
	\label{thm:B}
	For every Lagrangian torus $L$ in $\C^2$, there exists $x \in \R_{\geqslant 0}$ such that the flux-monodromy group associated to $L$ is given by
	\begin{align}
		\label{eq:imFM_C2}
		\Gamma(L) = \left\{\left.\left( 
		x \begin{pmatrix}
			(-1)^{\delta}-1 \\
			k
		\end{pmatrix},
		\begin{pmatrix}
			(-1)^{\delta} & k \\
			0 & 1
		\end{pmatrix}
		\right)
		\;
		\right\vert
		\;
		k \in \Z, \delta \in \{0,1\}	
		\right\} .
	\end{align}
	In particular, both the associated flux set $\fl(L)$ and flux-monodromy group $\Gamma(L)$ are discrete.
	
	Moreover, there is an isomorphism $\pi_1(\cl)\cong (\Z\rtimes\Z_{2})\ltimes N$ for some normal subgroup $N \triangleleft \,\ker\FM$.
\end{theorem}
By rescaling in $\C^2$, we can ensure that a Lagrangian loop is homotopic to one that is supported in any ball that contains the base Lagrangian $L$. Therefore, if $L$ is any Lagrangian torus in a symplectic 4-manifold that is contained in a Darboux ball, then $\pi_1(\cl,L)$ contains a subgroup isomorphic to $\Z \rtimes \Z_2$.

\begin{remarks} 
Equivalent descriptions of the monodromy group of Lagrangian tori in $\mb C^{2}$ already appeared in the works of Chekanov \cite[Proposition 5.1 (a)]{Che96} and of Yau~\cite{Yau12}. We give a more explicit description in terms of almost toric fibrations. This allows us to construct a section $s : \Gamma(L) \to \pi_{1}(\cl)$ to the flux-monodromy morphism from which the splitting of $\pi_1(\cl)$ follows, which is entirely new; see~\cref{thm:section_FM}.
\end{remarks}

In the case of \emph{product tori}, that is,
	\begin{equation*}
		T(a,b) 
		= S^1(a) \times S^1(b)
		= \{\pi\vert z_1 \vert^2 = a, \pi \vert z_2 \vert^2 = b\},
	\end{equation*}
	for $a,b \in \R_{>0}$, the number $x \in \R$ appearing in~\eqref{eq:imFM_C2} is simply $x = \vert a - b \vert$. Thus, loops based at the monotone product torus $L=T(a,a)$ have vanishing flux.\footnote{In \cref{cor:flux_monotone} below, we prove that this is a general phenomenon for monotone Lagrangians when $H_1(X;\R)=0$.} Moreover, the copy of $\Z_2$ in $\Gamma(L)$ is then generated by an ambient \emph{Hamiltonian} isotopy whose time one map coincides with the coordinate swap $(z_1,z_2) \mapsto (z_2,z_1)$ on a compact subset. On the other hand, we show that no element of the copy of $\Z$ can ever be realized by an ambient Hamiltonian isotopy. 

        \medskip
In order to state our result for $\CP^2$, we denote by $\rho : \GL(2;\Z) \to \GL(2;\Z/3\Z)$ the reduction mod 3 and by $\1$ the vector $(1,1)$.
\begin{theorem}
	\label{thm:Cliff_CP2}
	The monodromy and flux-monodromy groups associated with a Lagrangian torus $L$ in $\CP^2$ are isomorphic to the preimage of the stabilizer of $\1$ under $\rho$:
\begin{equation}
  \label{eq:stabilizer_maslov_class}
	H' \defi \rho^{-1}(\mathrm{Stab}(\1)) = \{h \in \GL(2;\Z) \sth \1\, h = \1 \mod 3\} .
      \end{equation}
It is a congruence subgroup of index 8 in $\GL(2;\Z)$. Moreover, it is generated by
        \begin{equation*}
		\begin{pmatrix}
			0 & 1 \\
			1 & 0
		\end{pmatrix},\quad
		\begin{pmatrix}
			2 & 1 \\
			-1 & 0
		\end{pmatrix},\quad \mbox{and }
		\begin{pmatrix}
			0 & -1 \\
			1 & -1
		\end{pmatrix}.
\end{equation*}
  \end{theorem}

Notice that the modulo $3$ condition comes from the fact that the Maslov class takes values in $6\Z$ when evaluated on $H_2(\CP^2;\Z)$. Therefore, it yields a well-defined class $2\cdot\1$ in the mod-6 reduction of $H^{1}(L)$.

As for the three generating matrices, the first two correspond to the obvious generators of $\Z_2\ltimes\Z$, the flux-monodromy group of a Lagrangian torus in $\C^2$. The last matrix comes from a rotation of the moment polytope of $\CP^2$.

\subsection{From flux-monodromy to flux} \label{ssec:f-m_to_flux} Knowing the homological flux-monodromy group $\Gamma^\Z(L) \subset H^1(L;\R) \rtimes \Aut H_1(L;\Z)$, we can understand the image $\fl(L) \subset H^1(L;\R)$ of the flux map on $\pi_1(\cl,L)$ by projecting to the $H^1$-component. By the discussion in \S\ref{ssec:mtofm}, we can thus understand $\fl(L)$ for any $L$ which is Lagrangian isotopic to a product torus in $\C^n$ or to the Clifford torus in $\CP^n$~---~though, again, these sets are difficult to write down explicitly. In $\CP^2$, we prove that the flux set $\fl(L)$ is dense in $H^1(L; \R^2) \cong \R^2$ whenever there is a Lagrangian isotopy from the Clifford torus $T_{\rm Cl}$ to $L$ which has flux $(x,y) \in \R^2$ with $\frac{x}{y} \notin \Q$; see \cref{prop:flux-set_CP^2}. This contrasts with the discrete flux set we got in \cref{thm:B} for all tori in $\C^2$. 


This begs the question of what happens for \emph{Lagrangian paths}, meaning Lagrangian isotopies which do not necessarily close up. The \emph{shape} of $L$ is defined as
\begin{equation*}
	\Sh (L) = \{\Flux \Pi \in H^1(L;\R) \sth \Pi_0 = L \}.
\end{equation*}
See \S\ref{ssec:review} for a review of known facts about this invariant.

Our understanding of Lagrangian monodromy in $\CP^n$ allows us an almost complete description of the set $\Sh(T_{\rm Cl})$ for the Clifford torus. We concentrate here on the case $n = 2$; see \S\ref{ssec:fluxCPn} for more details on the case $n\geqslant 3$.

\begin{figure}[ht]
	\centering
	\begin{tikzpicture}
		\fill[opacity=0.1] (-3,-3) rectangle (3,3);
		\fill (0,0) circle[radius=2pt] node[anchor=south east] {$\tcl$};
		\draw[very thick] (-1,-1) -- (2,-1) -- (-1,2) -- (-1,-1);
		
		\fill[violet] (1,1) circle[radius=2pt];
		\draw[very thick,violet] (0.5,0.5) -- (3,3);
		\fill[violet] (-2,1) circle[radius=2pt];
		\draw[very thick,violet] (-1,0.5) -- (-3,1.5);
		\fill[violet] (1,-2) circle[radius=2pt];
		\draw[very thick,violet] (0.5,-1) -- (1.5,-3);
		\draw[very thick,violet] (0.5,2) -- (0.75,3);
		\draw[very thick,violet] (2,0.5) -- (3,0.75);
		\draw[very thick,violet] (2,-2.5) -- (2.4,-3);
		\draw[very thick,violet] (0.5,-2.5) -- (0.6,-3);
		\draw[very thick,violet] (-1,-2.5) -- (-1.2,-3);
		\draw[very thick,violet] (-2.5,-1) -- (-3,-1.2);
		\draw[very thick,violet] (-2.5,0.5) -- (-3,0.6);
		\draw[very thick,violet] (-2.5,2) -- (-3,2.4);
		
		\fill[pink] (-1,-1) circle[radius=2pt];
		\draw[very thick,pink] (-1,-1) -- (-3,-3);
		\fill[pink] (2,-1) circle[radius=2pt];
		\draw[very thick,pink] (2,-1) -- (3,-1.5);
		\fill[pink] (-1,2) circle[radius=2pt];
		\draw[very thick,pink] (-1,2) -- (-1.5,3);
		
		\draw[very thick,white] (1,0) -- (3,0);
		\draw[very thick,white] (0,1) -- (0,3);
		\draw[very thick,white] (-1,0) -- (-3,0);
		\draw[very thick,white] (0,-1) -- (0,-3);
		\draw[very thick,white] (1,-1) -- (3,-3);
		\draw[very thick,white] (-1,1) -- (-3,3);
		\draw[very thick,white] (1,2) -- (1.5,3);
		\draw[very thick,white] (2,1) -- (3,1.5);
		\draw[very thick,white] (-1,-2) -- (-1.5,-3);
		\draw[very thick,white] (-2,-1) -- (-3,-1.5);

		\fill[teal] (1,0) circle[radius=2pt];
		\draw[very thick,teal,dashed] (1,0) -- (3,0);
		\fill[teal] (0,1) circle[radius=2pt];
		\draw[very thick,teal,dashed] (0,1) -- (0,3);
		\fill[teal] (-1,0) circle[radius=2pt];
		\draw[very thick,teal,dashed] (-1,0) -- (-3,0);
		\fill[teal] (0,-1) circle[radius=2pt];
		\draw[very thick,teal,dashed] (0,-1) -- (0,-3);
		\fill[teal] (1,-1) circle[radius=2pt];
		\draw[very thick,teal,dashed] (1,-1) -- (3,-3);
		\fill[teal] (-1,1) circle[radius=2pt];
		\draw[very thick,teal,dashed] (-1,1) -- (-3,3);
		\fill[teal] (1,2) circle[radius=2pt];
		\draw[very thick,teal,dashed] (1,2) -- (1.5,3);
		\fill[teal] (2,1) circle[radius=2pt];
		\draw[very thick,teal,dashed] (2,1) -- (3,1.5);
		\fill[teal] (-1,-2) circle[radius=2pt];
		\draw[very thick,teal,dashed] (-1,-2) -- (-1.5,-3);
		\fill[teal] (-2,-1) circle[radius=2pt];
		\draw[very thick,teal,dashed] (-2,-1) -- (-3,-1.5);
		\fill[teal] (3,1) circle[radius=2pt];
		\fill[teal] (3,2) circle[radius=2pt];
		\fill[teal] (2,3) circle[radius=2pt];
		\fill[teal] (1,3) circle[radius=2pt];
		\fill[teal] (-1,3) circle[radius=2pt];
		\fill[teal] (-2,3) circle[radius=2pt];
		\fill[teal] (-3,2) circle[radius=2pt];
		\fill[teal] (-3,1) circle[radius=2pt];
		\fill[teal] (-3,-1) circle[radius=2pt];
		\fill[teal] (-3,-2) circle[radius=2pt];
		\fill[teal] (-2,-3) circle[radius=2pt];
		\fill[teal] (-1,-3) circle[radius=2pt];
		\fill[teal] (1,-3) circle[radius=2pt];
		\fill[teal] (2,-3) circle[radius=2pt];
		\fill[teal] (3,-2) circle[radius=2pt];
		\fill[teal] (3,-1) circle[radius=2pt];
	\end{tikzpicture}
	\caption{Illustration of the statement of \cref{thm:possible-fluxes_CP2_intro}. The simplex is the moment polytope $\Delta_{\CP^2} \subset \R^2 = H^1(T_{\rm Cl};\R)$. The colours of the solid dots indicate their orbit type in the decomposition \eqref{eq:hprime_intro}. The purple and pink dots and half-lines emanating from them are obstructed. The dashed green lines are the cases that remain unknown. The area shaded in grey can be realized as the flux of a Lagrangian isotopy.}
	\label{fig:possible-fluxes_CP2}
\end{figure}
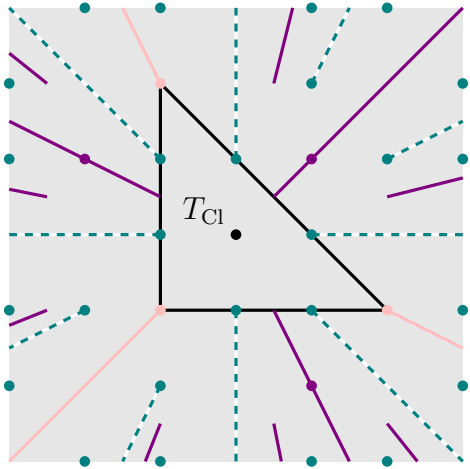

We show that the action of the subgroup $H'$, defined by \cref{eq:stabilizer_maslov_class} above, on the set $\cp \subset H_1(T_{\rm Cl};\Z) = \Z^2$ of primitive elements has three orbits:
\begin{equation}
	\label{eq:hprime_intro}
	\cp = ((1,1) \cdot H') \sqcup ((-1,-1) \cdot H') \sqcup ((0,1) \cdot H').
\end{equation}

Our result is the following. 

\begin{theorem} \label{thm:possible-fluxes_CP2_intro}
	Let $\eta \in H^1(T_{\rm Cl}; \R) = \R^2$. If $\eta \notin \R \cdot \Z^2$, then $\eta \in \Sh(T_{\rm Cl})$. If $\eta \in \R \cdot \Z^2$, then write $\eta = r \lambda$ for $r \geqslant 0$ and $\lambda \in \cp$. We have the following three cases:
	\begin{enumerate}
		\item if $\lambda \in (1,1) \cdot H'$, then $\eta = r\lambda \in \Sh(T_{\rm Cl})$ if and only if $r < \frac{1}{2}$;
		\item if $\lambda \in (-1,-1) \cdot H'$, then $\eta = r\lambda \in \Sh(T_{\rm Cl})$ if and only if $r < 1$;
		\item if $\lambda \in (0,1) \cdot H'$, then $\eta = r\lambda \in \Sh(T_{\rm Cl})$ if $r < 1$.
	\end{enumerate} 
\end{theorem} 

The only case remaining open is thus $\eta = r \lambda$ with $\lambda \in (0,1) \cdot H'$ and $r \geqslant 1$; see \cref{q:cp2flux}. The construction part of the argument, showing that $\R^n \setminus (\R \cdot \Z^n) \subset \Sh(T_{\rm Cl})$, is true for all $n$. The obstruction part is based on \cite[Theorem 1.1]{CieMoh18}, meaning that it works for general $n$, too. However, it is much trickier to understand the set of directions which is obstructed, mainly because we do not understand the orbit structure of the analogue of $H'$ in higher dimensions.

\section*{Acknowledgements}

This project was started during a visit of RL at the Forschungsinstitut für Mathematik at ETH Zürich and continued during visits of JB and JPC at the Laboratoire de Mathématiques d'Orsay. We thank both these institutions for the fantastic research environment they provide. We have benefited from discussions with Yves Benoist and Manfred Einsiedler about Theorems \ref{thm:Cliff_CP2} and \ref{thm:possible-fluxes_CP2_intro}. We are especially grateful to Simon Machado for his patience in explaining the methods used in \S\ref{ssec:mon_CPn} to us. Likewise, we thank Dominique Rathel-Fournier for sharing his notes on Lagrangian homotopies and their flux, which led to the newest version of Appendix~\ref{app:Lag_homotopies}. Furthermore, we thank Paul Biran, Emmy Murphy, and Felix Schlenk for various enlightening discussions about the content of this paper.  

JB acknowledges the support of the Swiss National Science Foundation Ambizione Grant PZ00P2-223460. JPC acknowledges the partial support from the Swiss National Science Foundation (grant 200021\_204107) and the Fondation Courtois.

\tableofcontents

\section{Background} \label{sec:background}
\commentintoc{We review related literature in \S\,\ref{ssec:review} and present basic facts about the flux-monodromy morphism in \S\,\ref{ssec:basic}.}

\subsection{Review of literature and relation to previous work}
\label{ssec:review}
\phantom{a}

\emph{Monodromy of Lagrangian submanifolds.} The first result about monodromy of Lagrangian submanifolds that we are aware of appears in Chekanov's work \cite{Che96}. More precisely, he answers \cref{q:realizability} positively for product tori in $\R^{2n}$. See also the work by Yau \cite{Yau09, Yau12} on the subject. We are not aware of any results about the possible monodromies of Lagrangian loops beyond these. 

However, when we replace loops of Lagrangian isotopies by \emph{Hamiltonian diffeomorphisms} mapping the Lagrangian to itself, there has been much more work. Chekanov \cite[Theorem 4.5]{Che96} gave a full description for product tori in $\R^{2n}$. Hu--Lalonde--Leclercq \cite{HuLalLec11} showed that Hamiltonian loops of closed weakly exact Lagrangians have trivial homological monodromy; this was recently extended to generalized homology theories by Porcelli \cite{Por24, Por25}. Another recent work by Augustynowicz--Smith--Wornbard \cite{AugSmiWor22} gave a full description of the (homological) Hamiltonian monodromy of many monotone Lagrangian tori. In \cite{Bre23} and \cite{BreKim23}, the use of symmetric probes and new obstructions allowed for a full picture of this group for many (not-necessarily monotone) toric fibres in the 4-dimensional setting. For a full literature review, we recommend the excellent introduction of \cite{AugSmiWor22}. 

\smallskip

\emph{Values of the Lagrangian flux.} The question discussed in \S\ref{ssec:f-m_to_flux} of which possible values in $H^1(L;\R)$ can be realized as the flux of a Lagrangian isotopy has a long history. For a Lagrangian isotopy $\iota_t$ in $(\R^{2n},\omega_0 = d\lambda)$, the flux is equal to $[\iota_1^* \lambda] - [\iota_0^*\lambda] \in H^1(L;\R)$. By Gromov \cite{Gro85}, there are no closed exact Lagrangians in $\R^{2n}$, meaning that $[\iota^* \lambda] \neq 0$ in that case. In fact, for Lagrangian product tori in $\R^{2n}$, it follows from work of Viterbo \cite[Theorem A]{Vit90} that $[\iota^*\lambda]$ is not a non-positive multiple of the Maslov class. Sikorav \cite[\S 1.5]{Sik91} conjectured that, in the standard basis, the class $[\iota^*\lambda]$ should be constrained to the positive orthant of $H^1(T^n;\R) = \R^n$. Chekanov proved that this is false; in fact, combining \cite[Proposition 5.1 (b)]{Che96} and Viterbo's result completely determines the \emph{shape} $\Sh(L)$ for any product torus $L = \iota(T^n)$, we get that
\begin{equation*}
	\Sh(L) - [\iota^*\lambda] = H^1(T^n;\R) \setminus \R_{\leqslant 0} \cdot \mu_L.
\end{equation*} 
Remarkably, it was found in more recent work by Entov--Ganor--Membrez~\cite{EntGanMem18} and Shelukhin--Tonkonog--Vianna~\cite{SheTonVia19} that the so-called \emph{star shape} is much smaller. The star shape is the subset of $H^1(L;\R)$ that can be reached by Lagrangian isotopies moving along a straight line in $H^1(L;\R)$. This salvages a version of Sikorav's above-mentioned conjecture: In \cite[Theorem 6.8]{SheTonVia19}, it was proved that the star shape of a product torus in $\R^{2n}$ is indeed just the positive orthant.\smallskip

\emph{Topology of the space of Lagrangians.} In general, very few things are known about the topology of $\cl(L)$, i.e.\ of the space of Lagrangians that are Lagrangian isotopic to a given Lagrangian $L$. The case of curves $L$ in (finite type) surfaces $\Sigma$ is, however, one notable exception to this. Indeed, one can compute the homotopy type of $\cl(L)$ and $\pcl(L)$ using a classical result of Gramain~\cite[Theorem 4]{Gramain1973}. Alternatively, one can also use the fact that length shortening flow converges to either a point or a geodesic~\cite{Grayson1989} to construct a deformation retraction of these spaces. The table below gives all possible homotopy type, together with a description of the flux-monodromy group in these cases. In particular, $\fl(L)$ is always discrete in these cases.

\begin{table}[ht]
	\label{tab:curves-in-surfaces}
\begin{tabular}{ccccc} 
	\toprule
	{$\Sigma$} & {$L$ contractible?} & {$\pcl(L)$} & {$\cl(L)$} & $\Gamma(L)$  \\ \midrule
	$\Sigma=S^2$ & Yes &  $\RP^3$ & $\RP^2$ & $\Z_2$ \\
	$\Sigma\neq S^2$ & Yes  & $S\Sigma$  & $\Sigma$ & 0 \\
	$\Sigma=\T^2$ & No  & $\T^2$  & $S^1$ & $\Z$ \\ 
	$\Sigma\neq \T^2$ & No  & $S^1$  & $\{*\}$ & 0 \\ 
	\bottomrule
\end{tabular}
\caption{Homotopy types of spaces of curves in finite type surfaces. Here, $S\Sigma$ is the  unit tangent bundle of the surface $\Sigma$.}
\end{table}

In higher dimensions, the picture is much less clear. Nonetheless, the homotopy type of $\pcl$ was computed by Coffey~\cite{Coffey2005} in two important 4-dimensional cases: the Lagrangian $\RP^2$ in $\CP^2$ and Lagrangian spheres $S$ in $S^2 \times S^2$. In both cases, it is given by the rigid transformations of the ambient spaces: $\pcl(\RP^2)$ is homotopy equivalent to $\mathrm{PU}(3)$ and $\pcl(S)$ is homotopy equivalent to $\mathrm{SO}(3)\times\mathrm{SO}(3)$. This yields that $\cl(\RP^2)$ is homotopy equivalent to $\mathrm{PU}(3)/\mathrm{PO}(3)=(\mathrm{SU}(3)/\Z_3)/\mathrm{SO}(3)$ and $\cl(S)$ is homotopy equivalent to $\mathrm{SO}(3)$. In particular, $\cl(\RP^2)$ and $\cl(S)$ have respective fundamental groups $\Z_3$ and $\Z_2$. However, since all noncontractible loops in $\cl(L)$ come from loops in $\Ham(X)$, these are not detected by $\FM$. \par

We end with one last well-studied case: when $L$ is the 0-section of $T^*L$. In this case, Arnol'd's nearby Lagrangian conjecture~\cite{Arn86} states that the Hamiltonian orbit of $L$ should be the set of all exact Lagrangians, i.e.\ those Lagrangians $L'$ such that restriction to $L'$ of the tautological 1-form $\lambda_0$ is exact. This is only known to hold in $S^1$ [folklore], $S^2$~\cite{Hin04}, $\RP^2$~\cite{HinPinWu16}, and $\T^2$~\cite{GooIvrRiz16}. Moreover, in the first three cases, the authors show that the \emph{strong} nearby Lagrangian conjecture holds: the space of exact Lagrangians is (weakly) contractible. This also determines the (weak) homotopy type of $\cl(L)$.

\begin{lemma} \label{lem:T^*L_from-gen-to-exact}
	Let $L$ be the 0-section in $T^*L$. The space $\cl(L)$ deformation retracts onto $\LHam(L)$. In particular, whenever the strong nearby Lagrangian conjecture holds, $\cl(L)$ is (weakly) contractible.
\end{lemma}

\begin{proof}
  Fix a Riemannian metric on $L$. Note that, for any $L'\in\cl(L)$, the canonical projection $\pi:T^*L\to L$ restricts to a homotopy equivalence  $\pi|_{L'}:L'\to L$. Therefore, there is a unique harmonic 1-form $\sigma_{L'}$ on $L$ such that $\pi|_{L'}^*[\sigma_{L'}]=[\lambda_0|_{L'}]$ in $H^1(L';\R)$. Let $\{\psi_{L'}^t\}$ be the flow on $T^*L$ generated by $\pi^*\sigma_{L'}$. A direct computation shows that $\psi_{L'}^1(L')$ is exact. \par
  
  Note that, by elliptic regularity, a $C^\infty$-continuous perturbation of $L'$ gives a $C^\infty$-continuous perturbation of $\sigma_{L'}$, and thus of $\psi^t_{{L'}}$. Moreover, $\sigma_{L'}=0$~---~and thus $\psi^t_{{L'}}=id$ for all $t$~---~whenever $L'$ is exact. Therefore, $[0,1]\times\cl(L)\to \cl(L);$ $(t,L')\mapsto \psi_{L'}^t(L')$ gives a deformation retraction onto the subspace $\cl^{ex}(L)$ of $\cl(L)$ formed by exact Lagrangians.
  
  But since $\cl(L)$ is path connected, so is $\cl^{ex}(L)$. However, a Lagrangian isotopy through exact Lagrangians is a Hamiltonian isotopy; this is well-known, but see \cref{lem:local-to-global_exact} for a self-contained proof. Therefore, $\cl^{ex}(L)=\LHam(L)$.
\end{proof}

\subsection{Basic facts on the flux-monodromy morphism}

\label{ssec:basic}

Let $L \subset X$ be a closed connected Lagrangian submanifold of a symplectic manifold $(X,\omega)$. By $\cl = \cl(L)$ we will denote the space of (unparametrized) Lagrangians in $X$ that are Lagrangian isotopic to $L$, which we see as a subspace of $C^\infty(L,X)/\operatorname{Diff}(L)$. More precisely, we mean by this that $L'$ is in $\cl$ if there exists a smooth map $I \colon [0,1] \times L \rightarrow X$ such that each $\iota_t = I(t,\cdot)$ is a Lagrangian embedding, and such that $\im \iota_0 = L$, and $\im \iota_1 = L'$. Such a map $I$ is called a \emph{Lagrangian isotopy}. Every Lagrangian isotopy can be extended to an ambient smooth isotopy $\{f_t\}_{t \in [0,1]}$ of $X$, in the sense that $f_t \circ \iota_0 = \iota_t$. We call $\{f_t\}$ an \emph{ambient Lagrangian isotopy}. For our purposes, it is often more convenient to forget about the parametrization and instead work with the corresponding path $\Pi = \{L_t\}_{t \in [0,1]}$ in $\cl$ of embedded Lagrangian submanifolds defined as the image $\im \iota_t = L_t$. By abuse of terminology, we call both $I$ and $\Pi$ a Lagrangian isotopy. Sometimes, we refer to $I$ as a \emph{parametrization} of $\Pi$.

Given $\Pi = \{L_t\}_{t \in [0,1]}$ with $L_0 = L$ and $L_1 = L'$, we can choose an ambient Lagrangian isotopy $\{f_t\}_{t \in [0,1]}$, i.e.\ one that satisfies $f_t(L) = L_t$ for all $t \in [0,1]$. This induces two maps, which we use throughout the paper:
\begin{align*}
	\Phi^{\Pi} = (f_1)_* \colon H_2(X,L) \rightarrow H_2(X,L'), \quad\mbox{and}\quad
	\phi^{\Pi} = f_1\vert_L \in \operatorname{Diff}(L,L').
\end{align*}
One can check that the first map is well-defined and that the second one is well defined up to post-composing with an element in $\operatorname{Diff}_0(L')$. In particular, the induced map $\phi^{\Pi}_* \colon H_1(L) \rightarrow H_1(L')$ is well-defined.

We equip the space $\cl$ with its natural $C^1$-topology unless otherwise stated; see \cite{Ono07} or \cite{ChaLec24} for more details. Let $\pi_1(\cl) = \pi_1(\cl, L)$ be the fundamental group of $\cl$ with basepoint $L$ (which will often be suppressed in the notation). Before turning to the \emph{flux-monodromy morphism}, we define and discuss the \emph{monodromy morphism} and the \emph{flux map}. Let $\MCG(L) = \operatorname{Diff}(L)/\operatorname{Diff}_0(L)=\pi_0(\operatorname{Diff}(L))$ be the smooth mapping class group.

\begin{definition}
\label{def:monodromy_morphism}
The \emph{monodromy morphism} is defined by 
	\begin{equation}
		\cm \colon \pi_1(\cl) \rightarrow \MCG(L), \quad
		\cm(\Lambda) = [\phi^{\Lambda}].
	\end{equation} 
\end{definition}

This is well-defined. Indeed, 
\begin{enumerate}
	\item the map $\phi^{\Lambda}$ is well-defined up to $\operatorname{Diff}_0(L)$;
	\item since the mapping class group is discrete, the element $[\phi^{\Lambda}]$ does not depend on the choice of loop representing the element $\Lambda \in \pi_1(\cl)$. 
\end{enumerate}
It is easy to check that $\cm$ is indeed a morphism with respect to the product $(\phi,\psi)\mapsto \psi\phi$ on $\MCG(L)$.

Let $\Pi = \{L_t\}_{t \in [0,1]}$ be a smooth path (which is not necessarily a loop) of Lagrangians starting at $L_0 = L$. Pick an ambient isotopy $\{f_t\}$ realizing $\Pi$. For any $\xi \in H_1(L)$, let $c:S^1\to L$ be a loop satisfying $[c] = \xi$. Let $C^{\Pi}_{\xi} \in H_2(X,L_0 \cup L_1)$ be the relative class defined by the map $(s,t)\mapsto f_t(c(s))$. Geometrically, this is the cylinder swept out by the loop $c$ under the Lagrangian isotopy $\Pi$. The map $H_1(L) \mapsto H_2(X,L_0 \cup L_1)$ defined by $\xi \mapsto C_{\xi}^{\Pi}$ is linear. Furthermore, it satisfies $\pp C_{\xi}^{\Pi} = \phi_*^{\Pi} \xi - \xi$. The equation 
\begin{equation}
	\label{eq:fluxmap}
	\Flux \Pi \cdot \xi = \langle [\omega] ,  C^{\Pi}_{\xi} \rangle
\end{equation}
defines the \emph{Lagrangian flux} $\Flux \Pi \in H^1(L;\R)$ of $\Pi$. This definition is independent of all choices we have made and is furthermore invariant under endpoint-preserving homotopies in the space of Lagrangians. Therefore, if we take $\Lambda = \{L_t\}_{t \in [0,1]}$ to be a loop in $\cl$, then the Flux map descends to a map on $\pi_1(\cl)$.

\begin{definition}
Let $L \subset X$ be a closed, connected Lagrangian submanifold. The \emph{flux map} $\Flux \colon \pi_1(\cl,L) \rightarrow H^1(L;\R)$ is defined by \eqref{eq:fluxmap}.
\end{definition}

This setup allows us to define the \emph{flux-monodromy morphism} as in \cref{def:flux_monodromy}.

Before moving on, we consider how the flux-monodromy morphism behaves under concatenation. By $\Pi * \Pi'$, we denote the concatenation of the Lagrangian path $\Pi$ followed by $\Pi'$. By $\Pi^-$, we denote the inverse path of $\Pi$, defined by $\Pi^-_t = \Pi_{1-t}$.

\begin{proposition} \label{prop:fm_basic} The flux-monodromy map $\FM_L \colon \pi_1(\cl) \rightarrow H^1(L;\R) \rtimes \MCG(L)$ has the following properties:
	\begin{enumerate}
		\item It is a morphism: For any $\Lambda, \Lambda' \in \pi_1(\cl)$ it satisfies
			\begin{equation*}
				\FM_L(\Lambda * \Lambda') = (\cm(\Lambda)^*\Flux(\Lambda') + \Flux(\Lambda), \cm(\Lambda') \cm(\Lambda));
			\end{equation*}
		\item It has a natural change-of-basepoint formula: Let $\Lambda \in \pi_1(\cl,L)$ and $\Pi$ be a path in $\cl$ with endpoints $\Pi(0) = L$ and $\Pi(1) = L'$. Then $\Pi^- * \Lambda * \Pi \in \pi_1(\cl,L')$ satisfies
			\begin{equation}
				\label{eq:fm_base_point_change}
				\FM_{L'}\left( \Pi^- * \Lambda * \Pi \right) 
				= 
				\left( (\phi^{\Pi})^{-*}\left( \cm(\Lambda)^* \Flux \Pi + \Flux \Lambda - \Flux \Pi \right), [\phi^{\Pi} \circ \phi^{\Lambda} \circ (\phi^{\Pi})^{-1}] \right).
			\end{equation}
	The element $[\phi^{\Pi} \circ \phi^{\Lambda} \circ (\phi^{\Pi})^{-1}] \in \MCG(L')$ is well-defined because $\phi^{\Pi}$ is well-defined up to post-composition with an element in $\operatorname{Diff}_0(L')$.
	\end{enumerate}
\end{proposition}

\begin{proof}
  Let $\Pi$, $\Pi'$ be paths in $\cl$ such that $\Pi(1) = \Pi'(0)$, meaning that the concatenation $\Pi * \Pi'$ makes sense. We claim that
  \begin{equation}
    \label{eq:concatenation}
    \Flux(\Pi * \Pi') = (\phi^{\Pi})^* \Flux \Pi' + \Flux \Pi.
  \end{equation}
  Both properties listed in the proposition follow from a straightforward computation using \eqref{eq:concatenation}. To prove the identity \eqref{eq:concatenation}, recall that, from \eqref{eq:fluxmap}, $\Flux(\Pi * \Pi') \cdot \xi$ is given by the symplectic area of the cylinder over $\xi$ swept out by $\Pi * \Pi'$. Decomposing the cylinder as $C^{\Pi * \Pi'}_{\xi} = C^{\Pi'}_{\phi_*^{\Pi}\xi} \cup C^{\Pi}_{\xi}$ and applying $[\omega]$ yields the claim.
\end{proof}

By $[\omega]_L \in H^2(X,L;\R)$ we denote the class in relative cohomology which assigns to any class in $H_2(X,L)$ its symplectic area. Note that any element $b \in H_2(X,L)$ and any path $\Pi = \{L_t\}_{t \in [0,1]}$ starting at $L = L_0$ induces a natural element $\Phi^{\Pi} b \in H_2(X,L_1)$ with boundary $\pp \Phi^{\Pi} b = \phi^{\Pi}_* \pp b$ when transported using $\Pi$. Furthermore, $\Phi^{\Pi}b$ is homologous to $b + C^{\Pi}_{\pp b}$. Hence, applying $[\omega]$ yields the \emph{area identity}:
\begin{equation}
	\label{eq:areaformula}
	(\Phi^{\Pi})^* [\omega]_{L_1} - [\omega]_{L_0} = \pp^* \Flux \Pi.
\end{equation}

We denote the Maslov class (on homology) by $m_L \in H^2(X,L;\Z)$ and call its evaluation on a surface with boundary on $L$ the Maslov index of the surface. Since the Maslov class is a discrete invariant, it is preserved under Lagrangian isotopies. This implies an important property.

\begin{proposition} \label{prop:Maslov-cyl}
	Let $\Lambda=\{L_t\}$ be a Lagrangian loop, i.e. $L_1=L_0=L$. For any loop $c$ on $L$, the associated cylinder $C^\Lambda_{\xi}$ has Maslov index 0.
\end{proposition}

This generalizes the fact that tori defined by loops of symplectic diffeomorphisms have vanishing first Chern number, following classical work of McDuff~\cite{McD84}. 

\begin{proof}
  Fix a parametrization $\{\iota_t:L\hookrightarrow X\}$ of $\Lambda$, and let $c:S^1\to L$ with $[c] = \xi$, and consider $f:S^1\times [0,1]\to X; f(s,t)=\iota_t(c(s))$. Note that, since complex bundles over surfaces are classified by the first Chern class and $H^2(S^1\times [0,1];\Z)=0$, $f^*TX$ is trivial. Moreover, given a trivialization of $f^*TX$, $m_L(C^{\Lambda}_{\xi})$ is obtained as the difference of the Maslov indices of $s\mapsto T_{\iota_1(c(s))}L_1$ and $s\mapsto T_{\iota_0(c(s))}L_0$, seen, via that trivialization, as loops into the Lagrangian Grassmannian of $\R^{\dim X}$. But $(s,t)\mapsto T_{\iota_t(c(s))}L_t$ gives a homotopy between these loops, so that these Maslov indices are the same.
\end{proof}

This result is particularly useful when studying \emph{H-monotone} Lagrangians, i.e. those $L$   such that $[\omega]_L = \frac{C}{2} m_L$ for some $C\in\R$. Indeed, we directly get the following.

\begin{corollary} \label{cor:flux_monotone}
Suppose that $L$ is an $H$-monotone Lagrangian in $X$. Then, $\Flux_L\Lambda=0$ for every loop $\Lambda\in\pi_1(\cl,L)$.
\end{corollary}

\section{Discreteness and topological properties}
\label{sec:topol-prop}
\commentintoc{We first prove \cref{thm:A}, which gives necessary and sufficient conditions to ensure discreteness of the flux-monodromy group in \S\,\ref{ssec:thmA}. We then turn to the topological consequences of said discreteness, namely by proving Propositions \ref{prop:FM-top} and \ref{prop:Flux-top}. These are divided into two parts: local path connectedness in \S\,\ref{sec:local-path-connectedness} and closedness in \S\,\ref{sec:closedness}. \cref{prop:few-cases_discreteness}, which ensures discreteness of the flux set in various settings, is proved at the end of \S\,\ref{sec:local-path-connectedness}.}

We first prove \cref{thm:A}, which gives necessary and sufficient conditions to ensure discreteness of the flux-monodromy group in \S\,\ref{ssec:thmA}. We then turn to the topological consequences of said discreteness, namely by proving Propositions \ref{prop:FM-top} and \ref{prop:Flux-top}. These are divided into two parts: local path connectedness in \S\,\ref{sec:local-path-connectedness} and closedness in \S\,\ref{sec:closedness}. \cref{prop:few-cases_discreteness}, which ensures discreteness of the flux set in various settings, is proved at the end of \S\,\ref{sec:local-path-connectedness}.

\subsection{Proof of \cref{thm:A} } \label{ssec:thmA}

We start with the proof of the second statement in~\cref{thm:A}, i.e.\ that discreteness of $\cp(X) = \omega(\ker c_{1})$ implies discreteness of $\Gamma(L)$ for all $L\subseteq X$. In fact, we prove a slightly stronger statement that will be useful for applications. \par

Consider the natural projection $p^{\Z}:H^1(L;\R)\rtimes\MCG(L)\to H^1(L;\R)\rtimes\Aut(H_1(L;\Z))$. One can easily check that this is a morphism. We call $\FM^{\Z}:=p^{\Z}\circ \FM$ the \emph{homological flux-monodromy morphism} and denote its image by $\Gamma^{\Z}(L)$.

\begin{proposition}
	\label{prop:discreteness2}
	Assume that $\cp(X)$ is discrete. Then the homological flux-monodromy group $\Gamma^{\Z}(L)$ is discrete for all Lagrangians $L \subset X$.
\end{proposition}

Considering the projection $\Gamma(L)\to\Gamma^{\Z}(L)$ and the discreteness of $\MCG(L)$, this obviously implies the second part of \cref{thm:A}.

\begin{proof}
  Let $\Lambda \in \pi_1(\cl,L)$ be a loop in the space of Lagrangians of some Lagrangian submanifold $L \subset X$. If its monodromy preserves some $\xi \in H_1(L)$, i.e.\ if $\cm(\Lambda)\xi = \xi$, then the cylinder $C^{\Lambda}_\xi$ swept out by $\xi$ under $\Lambda$ satisfies $\pp C^{\Lambda}_\xi = 0$ and thus defines\footnote{This class in unique up to a choice of element in $H_{2}(L)$, but the symplectic area \eqref{eq:fluxsigma} does not depend on this choice, since $L$ is Lagrangian.} a class $\Sigma_{\xi} = [C^{\Lambda}_{\xi}] \in H_2(X)$.
  This correspondance satisfies
  \begin{equation}
    \label{eq:fluxsigma}
    \Flux(\Lambda) \xi = \omega(\Sigma_{\xi}) \quad \emph{ for all } \xi \in H_1(L).
  \end{equation}
  Moreover, the first Chern number $c_1(\Sigma_{\xi})\in\Z$ of $\Sigma_{\xi}$ is half the Maslov index $m_L(C^{\Lambda}_{\xi})$ of $C^{\Lambda}_{\xi}$. By \cref{prop:Maslov-cyl}, we have that $c_1(\Sigma_{\xi})=0$, and $\Flux(\Lambda)\xi$ is in $\cp(X) \subset \R$ whenever $\Lambda$ has trivial homological monodromy. 

  Now let $\{(f_n,M_n)\}_{n \in \N} \in \Gamma^{\Z}(L) \subset H^1(L;\R) \rtimes \Aut(H_1(L;\Z))$ be a sequence converging to $(f,M)$, and let $\Lambda_n \in \pi_1(\cl,L)$ be a sequence of loops with $\FM^\Z(\Lambda_n) = (f_n,M_n)$. Note that $\Aut(H_1(L;\Z))$ is itself a discrete group. Therefore, the sequence $M_n$ is constant and equal to $M$ for all $n$ greater than or equal to some $n_0$. The sequence $\{\Lambda_n * \Lambda_{n_0}^- \}_{n \in \N}$ has flux $\Flux(\Lambda_n * \Lambda_{n_0}^-) = f_n -M^*f_{n_0}$ and trivial monodromy for all $n \geqslant n_0$ by Proposition~\ref{prop:fm_basic}. By the above paragraph, it follows that $f_n -M^*f_{n_0} \in H^1(L;\cp(X))$. Since the latter group is discrete, the sequence $f_n$ must stabilize, too.
\end{proof}

\begin{remarks}\label{rk:discreteness_equiv}
	A few things need to be noted about this proof.
	\begin{enumerate}
		\item An approach analogous to the latter part of the proof of Proposition~\ref{prop:discreteness2} also shows that $\Gamma(L)$ is discrete if and only if its subgroup
		\begin{equation*}
			\Gamma_0(L) = \FM
			\left(
			\ker \left( \cm \colon \pi_1(\cl,L) \rightarrow \MCG(L) \right)
			\right)
		\end{equation*}
		is. This will be used later on.
		
		\item The proof of \cref{prop:discreteness2} only uses classes in $H_2(X)$ realized by tori. Therefore, we only need the subset of $\cp(X)$ generated by Chern-zero tori to be discrete. For example, products of surfaces of genus at least 2 have this property, although the full $\cp(X)$ might not be discrete.
		
		\item When only $\cp_S(X)$, the subset of $\cp(X)$  formed by spherical classes, is known to be discrete, essentially the same proof as above implies that $\Gamma^{\Z}(L)$ discrete for all $L$ \emph{bounding enough disks}, i.e. such that the natural map $\pi_2(X,L)\to\pi_1(L)\to H_1(L)$ has finite cokernel.\qedhere
	\end{enumerate}
\end{remarks}

Before proving the other part of \cref{thm:A}, we note that the proof of \cref{prop:discreteness2} actually allows us to compute $\Gamma(L)$ and $\fl(L)$ in some important cases.

\begin{proposition} \label{prop:flux-monodromy_monotone}
	Let $X$ be $H$-monotone, meaning monotone on $H_2(X)$, and let $L\subseteq X$ be a Lagrangian. The natural projections $\Gamma(L)\to\cm(L)$ and $\Gamma^{\Z}(L)\to\cm^\Z(L)$ are isomorphisms. If moreover the inclusion $L\hookrightarrow X$ induces an injection $H_1(L;\R)\hookrightarrow H_1(X;\R)$, then $\fl(L)=0$.
\end{proposition}

As the notation suggests, $\cm^\Z(L)$ is the image of the monodromy group $\cm(L)$ of $L$ in $\Aut(H_1(L;\Z))$.

\begin{proof}
  The kernel $\Gamma_0(L)$ of $\Gamma(L)\to \cm(L)$ consists of the set of fluxes realized by loops with trivial monodromy. By the proof of \cref{prop:discreteness2}, these fluxes are all in $H^1(L;\cp(X))$. But by $H$-monotonicity, $\cp(X)=0$, so that $\Gamma_0(L)$ is trivial. The proof for $\Gamma^{\Z}\to\cm^\Z(L)$ is completely analogous. \par
  
  Suppose now that $H_1(L;\R)\hookrightarrow H_1(X;\R)$. Then, the real homological monodromy $\cm^{\R}(\Lambda)$ of any Lagrangian loop $\Lambda$ is trivial. By the universal coefficient theorem, this means that $\cm^{\Z}(\Lambda)$ acts trivially on the free part of $H_1(L)$. Since the torsion part of $H_1(L)$ is finitely generated, there can be no homomorphism from it to the free part, and its automorphism group is finite. Therefore, there is some $N$ such that $\cm^{\Z}(\Lambda)^N=\cm^{\Z}(\Lambda^N)$ is the identity. By the first paragraph, this means that $\Flux(\Lambda^N)=N\Flux(\Lambda)=0$, where the first equality follows from the fact that $\Flux$ is a homomorphism in this setting. Therefore, $\fl(L)=0$ as stated.
\end{proof}

Note that this proves Case~(c) of \cref{prop:few-cases_discreteness}. The rest will be proven in \cref{sec:local-path-connectedness}.

For the other direction in \cref{thm:A}, we prove its contrapositive, i.e.\ under the assumption that $\mathcal{P}_S(X)$ is not discrete, we will find a Lagrangian $L$ such that the image of $\FM_L$ is not discrete. The Lagrangian $L$ is obtained by embedding a small product torus
\begin{equation}
	T(a) = T(a_1,\ldots,a_n) = \{\pi\vert z_i \vert^2 = a_i\} \subset \C^{n}
\end{equation}
for $a = (a_1,\ldots,a_n) \in \R^{n}_{>0}$ by a symplectic embedding $\varphi \colon B^{2n}(b) \hookrightarrow X$. Note that any two such tori $\varphi(T(a))$ and $\varphi(T(a'))$ are Lagrangian isotopic by a deformation through the images of product tori. We call this the \emph{obvious} Lagrangian isotopy. Such tori where studied in great detail by Chekanov--Schlenk in \cite{CheSch16}; in fact, our argument is a straightforward adaptation of \cite[Theorem 1.5]{CheSch16} that takes into account the monodromy of the Hamiltonian diffeomorphisms appearing in it.

\begin{theorem}	
	\label{thm:CS_monodromy}
	Let $X$ be a symplectic manifold, and consider $\varphi \colon B^{2n}(b) \hookrightarrow X$ a symplectic ball embedding. For every real number $c>0$, there exists $A >0$ such that for all $a \in (0,A]$, the following holds. If $d_1,\ldots,d_k$ and $e_1,\ldots,e_k$ satisfy $d_j,e_j \geqslant c$ for all $j\in \{1,\ldots,k\}$ and 
	\begin{equation}
		\label{eq:difference_Gamma}
		d_j - e_j \in \mathcal{P}_S(X),
	\end{equation}
then the tori (provided the corresponding product tori are contained in the domain of $\varphi$),
	\begin{equation}
		L = \varphi(T(a,\ldots,a,a+d_1,\ldots,a+d_k)), \quad
		L' = \varphi(T(a,\ldots,a,a+e_1,\ldots,a+e_k))
	\end{equation}
are Hamiltonian isotopic through a Hamiltonian isotopy with trivial monodromy. Here, $\operatorname{Diff}(L,L')$ and $\operatorname{Diff}(L)$ are identified by the obvious Lagrangian isotopy.
\end{theorem}

The case where $X$ is \emph{special} (in the terminology of Chekanov--Schlenk) is absent from the statement because if $X$ is special, then, in particular, $\mathcal{P}_S(X)$ is discrete. The key difference with respect to \cite[Theorem 1.5]{CheSch16} is that we take the $d_j - e_j$ to lie in $\mathcal{P}_S(X)$ instead of the larger group $G_a$ generated by the quantities $\omega(S) - ac_1(S)$ for $S \in \pi_2(X)$. Note that $\mathcal{P}_S(X)$ is obtained from $G_a$ by restricting to elements $S \in \ker c_1$. As will become apparent in the proof, considering spheres with trivial first Chern number is precisely the restriction needed to obtain trivial monodromy of the Hamiltonian diffeomorphisms constructed by Chekanov--Schlenk. \smallskip

\begin{proof}[Proof of \cref{thm:CS_monodromy}]
  Following the notation of the proof of \cite[Theorem 1.5]{CheSch16}, we define, on a suitable domain in $(\C^{\times})^2$ and for some $m \in \Z$ and $s > 0$, the symplectomorphism 
  \begin{equation}
    \label{eq:psi_ms}
    \psi_{m,s}(\rho_1,\theta_1,\rho_2,\theta_2)
    =
    (\rho_1,\theta_1 + m\theta_2, \rho_2 + s -m\rho_1, \theta_2),
  \end{equation}
  which we have expressed in symplectic polar coordinates $(\rho_i,\theta_i) \in \R_{\geqslant 0} \times S^1$ defined by $z_i = \sqrt{\frac{\rho_i}{\pi}} e^{2\pi i \theta_i}$ for all $z_i \neq 0$. Now let $S \in \pi_2(X)$ and set $s = \omega(S)$ and $m = c_1(S)$. In the case $n=2$, \cite[Proposition 6.1]{CheSch16} shows that there is a neighbourhood $U$ of the isotropic circle $\{0\} \times S^1(y_0) \subset \C^2$ and a Hamiltonian isotopy $\chi$ of $X$ which restricts to $\varphi \circ \psi_{m,s} \circ \varphi^{-1}\vert_{\varphi(U)}$ on $\varphi(U)$. Thus using \eqref{eq:psi_ms}, we find that for every $T(x,y) \subset U$, the Hamiltonian diffeomorphism $\chi$ maps $\varphi(T(x,y))$ to $\varphi(T(x,y + s - mx))$ with monodromy, expressed in the standard basis of product tori, given by 
  \begin{equation}
    \label{eq:chi_mon}
    (\chi\vert_{T(x,y)})_*
    =
    \begin{pmatrix} 
      1 & m \\ 
      0 & 1
    \end{pmatrix}.
  \end{equation}
  Note that the fact that these tori are Hamiltonian isotopic relies non-trivially on the existence of the sphere $S$, since in $\C^2$ itself, the product tori $T(x,y)$ and $T(x',y')$ are Hamiltonian isotopic if and only if $(x,y)$ and $(x',y')$ agree up to permutation. The presence of $S$ allows one to shift the tori by an amount depending on the area and the Maslov class of $S$. The key idea of our proof is restricting our attention to elements satisfying $c_1(S)=0$, since these have trivial monodromy by \eqref{eq:chi_mon}. For general $n$, the corresponding Hamiltonian isotopy is obtained as a product of $\psi_{m,s}$ with the identity in the remaining components, and we obtain the following as a straightforward consequence of \cite[Proposition 6.1]{CheSch16}.

  \begin{proposition}
    Let $S \in \cp_S(X)=\ker (c_1 \colon \pi_2(X) \rightarrow \Z)$, set $s = \omega(S)$, and let $d_1, \ldots, d_k$ for some $k \geqslant 1$. If furthermore $d_1 + \ldots + d_k + s < b$, then there is a neighbourhood $U$ of $(0,\ldots,0,d_1,\ldots,d_k) \in \R_{\geqslant 0}^n$ and a Hamiltonian isotopy $\chi$ such that 
    \begin{equation}
      \chi(\varphi(T(a_1,\ldots,a_{n-1},a_n))) = \varphi(T(a_1,\ldots,a_{n-1},a_n + s))
    \end{equation}
    for all $a = (a_1,\ldots,a_n) \in U$. Furthermore, $\chi$ induces trivial monodromy under the identification of the mapping class groups under the obvious Lagrangian isotopy.
  \end{proposition}

  This is the main building block of the proof of \cite[Theorem 1.5]{CheSch16}. The remainder of their strategy can be implemented by replacing $\pi_2(X)$ by $\ker c_1 \subset \pi_2(X)$ and adding the conclusion about trivial monodromy everywhere.
\end{proof}

\begin{corollary}
	If $\mathcal{P}_S(X)$ is not discrete, then there is a Lagrangian torus $L \subset X$ such that $\FM_L$ has non-discrete image.  
\end{corollary}

\begin{proof}
  Let $\varphi \colon B^{2n}(b) \hookrightarrow X$ be a symplectic ball embedding. Choose $a>0$, $k \geqslant 1$, and a sequence $d_i = (d_{i,1},\ldots,d_{i,k})$ of vectors converging componentwise to $d = (d_1,\ldots,d_k)$ such that $d_j - d_{i,j} \in \mathcal{P}_S(X)$ for all $j \in \{1, \ldots, k\}$. Since $\mathcal{P}_S(X) \subset \R$ is dense, we can assume that all $d_i$ are pairwise distinct. Let $L = \varphi(T(a,\ldots,a,a+d_1,\ldots,a+d_k))$ and $\cl$ the space of its Lagrangian deformations. Define $\Lambda_i \in \pi_1(\cl)$ to be the loop obtained by applying first the Hamiltonian isotopy from $L$ to $\varphi(T(a,\ldots,a,a+d_{i,1},\ldots,a+d_{i,k}))$ obtained from \cref{thm:CS_monodromy}, followed by the obvious Lagrangian isotopy. The monodromy of all $\Lambda_i$ is trivial, again by \cref{thm:CS_monodromy}. Its flux is given by $\Flux \Lambda_i = (0,\ldots,0,d_1-d_{i,1},\ldots,d_k-d_{i,k})$ in the pushforward of the standard basis of the homology of the product torus. By our choice of $d_i$, this proves the claim. 
\end{proof}

\subsection{$C^1$-local path connectedness} \label{sec:local-path-connectedness}
We now move to the proofs of the part of Propositions~\ref{prop:FM-top} and~\ref{prop:Flux-top} having to do with local path connectedness. As a byproduct, we also give a proof of \cref{prop:few-cases_discreteness}. We start with the most straightforward direction.

\begin{lemma} \label{lem:FM-discrete_to_path-connected}
	If 0 is isolated in $\fl(L)$ (resp.\ $\Gamma(L)$), then $\LHam(L)$ (resp.\ $\LHam(\iota)$) is locally path connected in the $C^1$ topology.
\end{lemma}
Again, 0 being isolated in $\Gamma(L)$ is equivalent to it being discrete, since it is a group. \par

\begin{proof}
  Suppose that $\LHam(L)$ is not $C^1$-locally path connected. Then, there are a Weinstein neighbourhood $\Psi:D^*_rL\to X$ of $L$ and a sequence of Lagrangians $\{L_k\}$ in $U=\Psi(D^*_rL)$, which are Hamiltonian isotopic to $L$ in $X$, but not in $U$, such that $L_k\to L$ in the $C^1$ topology. Indeed, if $L_k$ were Hamiltonian isotopic to $L$ in $U$ and $C^1$-close to it, then it would need to be the graph of a (closed) 1-form $\sigma_k$ that vanishes on all loops, i.e.\ $\sigma_k=df_k$ for some $f_k:L\to\R$. But then, $t\mapsto\Psi(\mathrm{graph}(tdf_k))$ would be a $C^1$-small path in $\LHam(L)$ from $L$ to $L_k$ in $U$, contradicting the failure of path-connectedness. \par
  
  By $C^1$ convergence, we have that $L_k=\Psi(\mathrm{graph}\ \sigma_k)$ for $\sigma_k\to 0$ in $C^1$ topology. As noted above, $\sigma_k$ is not exact. In fact, the flux of the path $t\mapsto \Psi(\mathrm{graph}(t\sigma_k))$ is precisely $[\sigma_k]\neq 0$. But we know that there is an isotopy $t\mapsto \phi^t_H(L)$ from $L$ to $L_k$ which is Hamiltonian and thus must have vanishing flux. Therefore, the concatenation of these isotopies forms a Lagrangian loop with nonvanishing flux $[\sigma_k]\to 0$. Therefore, 0 is not isolated in $\fl(L)$. \par
  
  For the case of $\LHam(\iota)$, the argument is essentially the same. However, $\iota$ and the members $\phi_k\circ\iota$ of the sequence may now be chosen to be $C^1$-close, so that the ensuing loop has no monodromy. Therefore, we instead get that 0 is not isolated in $\Gamma_0(L)$~---~and thus neither in $\Gamma(L)$. Recall that $\Gamma_0(L) = \FM(\ker \cm)$ is the subgroup of $\Gamma(L)$ formed by the image under the flux-monodromy map of loops with trivial monodromy.
\end{proof}

We now show a partial converse statement to the above lemma. 

\begin{lemma} \label{lem:nondiscrete-flux_flux-C1-conjA}
	Suppose that 0 is not isolated in $\fl(L)$. There is a sequence $\{L_k\}$ such that
	\begin{enumerate}[label=(\roman*)]
		\item $L_k\to L$ in $C^1$ topology;
		\item there exists a Lagrangian isotopy from $L$ to $L_k$ with flux zero;
		\item $L_k$ is not Hamiltonian isotopic to $L$ in any Weinstein neighbourhood of $L$.
	\end{enumerate}
	Moreover, if $\Gamma(L)$ is not discrete, then we may take a parametrization $I_k=\{\iota_{k,t}\}$ of the Lagrangian isotopy from $L$ to $L_k$ in \textit{(ii)} so that $\iota_{k,1}\to\iota$ in $C^1$ topology.
\end{lemma}

\begin{proof}
  Let $\Lambda_k$ be a sequence of Lagrangian loops with $\Flux \Lambda_k = a_k \neq 0$ such that $a_k \to 0$. Let $\Psi$ be a Weinstein neighbourhood of $L$. For large enough $k$, there are closed $1$-forms $\sigma_k$ with $[\sigma_k] = a_k$, such that the graph of $\sigma_k$ is contained in the domain of $\Psi$. We set $L_k = \Psi(\mathrm{graph}(\sigma_k))$ and show that this has the desired properties. Since $a_k\to 0$, we have that $L_k\to L$ in the $C^1$-topology, proving \emph{(i)}. For \emph{(ii)}, consider the Lagrangian isotopy obtained by concatenating $t \mapsto \Psi(\mathrm{graph}((1-t)\sigma_k))$ with the loop $\Lambda_k$. This has flux $-a_k + a_k = 0$ as desired. Since $a_k\neq 0$, $L_k$ is not Hamiltonian isotopic to $L$ in any Weinstein neighbourhood, proving \emph{(iii)}. \par
  
  Finally, when $\Gamma(L)$ is not discrete, then $\Gamma_0(L)$ is not either; see \cref{rk:discreteness_equiv}. Therefore, in that case, $\Lambda_k$ may be chosen to have trivial monodromy. In other words, for a parametrization $I_k$ of the loop, the diffeomorphism $f_k=\iota_{k,1}^{-1}\circ\iota$ is in the identity component of $\MCG(L)$. But then, any isotopy from the identity to $f_k^{-1}$ in $\mathrm{Diff}(L)$ extends via the Weinstein neighbourhood to a Hamiltonian isotopy of $X$ fixing $L$. Therefore, by concatenating with such an isotopy, we can suppose that $f_k$ is the identity without affecting the flux. The rest of the proof then follows as before.
\end{proof}

We can upgrade the Lagrangian isotopy from the above lemma to a Hamiltonian one in two important cases. \par

\begin{proposition}[\cite{Che96}] \label{prop:flux_C1-conjA_elem-tori}
	Let $L$ be a product torus in $\C^n$. Then, the Lagrangian isotopy of Lemma~\ref{lem:nondiscrete-flux_flux-C1-conjA} may be chosen to be Hamiltonian.
\end{proposition}

Note that the same results of Chekanov imply that if a product torus has non-discrete flux set, then the flux set is dense in $H^1(L;\R)$. In particular, 0 is not isolated.

\begin{proof}
  Suppose that 0 is not isolated in $\fl(L)$. Note that the existence~---~from Lemma~\ref{lem:nondiscrete-flux_flux-C1-conjA}~---~of a Lagrangian isotopy from $L$ to $L_k$ with zero flux together with the area identity \eqref{eq:areaformula} implies the existence of isomorphisms $\Phi_k:\pi_2(\C^n,L)\to \pi_2(\C^n,L_k)$ such that
  \begin{align*}
    \Phi^*_k[\omega]_{L_k}=[\omega]_{L} \qquad\text{and}\qquad \Phi^*_k\mu_{L_k}=\mu_{L}.
  \end{align*}
  Furthermore, for $k$ large enough, we can assume that $L_k$ is also an elementary torus. Therefore, the result follows from Chekanov's classification of elementary tori in $\C^n$~\cite{Che96}.
\end{proof}

\begin{proposition} \label{prop:flux_C1-conjA_injection}
  Let $L\subseteq X$ be a Lagrangian such that the map $\iota_*:H_1(L;\R)\rightarrow H_1(X;\R)$, induced by the inclusion, is injective. Then, any Lagrangian isotopy as in Lemma~\ref{lem:nondiscrete-flux_flux-C1-conjA} is homotopic to a Hamiltonian one.
\end{proposition}

\begin{proof}
  Suppose that 0 is not isolated in $\fl(L)$. Since $\iota_*$ is injective, the boundary map $H_2(X,L;\R)\to H_1(L;\R)$ is zero, so that every Lagrangian isotopy of $L$ preserves the area spectrum by the area identity (\ref{eq:areaformula}). By Proposition~\ref{prop:path-lag}, there is thus a symplectic isotopy $\{\psi_t\}$ generating any Lagrangian isotopy $\{L_t\}$ based at $L$. In particular, for a Lagrangian isotopy $\{L_t\}$ as constructed in Lemma~\ref{lem:nondiscrete-flux_flux-C1-conjA}, we have that
  \begin{align*}
    \iota^*\Flux(\{\psi_t\})=\Flux(\{L_t\})=0.
  \end{align*}
  But note that $\Ker\iota^*$ is precisely the image of $H^1(X,L;\R)\to H^1(L;\R)$, so that there is a closed 1-form $\sigma$ on $M$ that vanishes on (a neighbourhood of) $L$. Note that such a form generates a symplectic isotopy that fixes $L$ and $L_1$. Therefore, by applying the symplectic isotopy generated by $-\sigma$, we get a homotopy from our original Lagrangian isotopy to one that is generated by $\{\psi_t\}$ such that 
  \begin{align*}
    \Flux(\{\psi_t\})=0.
  \end{align*}
  But then, up to a homotopy relative to endpoints, $\{\psi_t\}$ is in fact Hamiltonian (see, for example, \cite[Section 10.2]{McDSal17}).
\end{proof}

\begin{remark} \label{rem:flux-0_to_Ham}
	In general, one cannot conclude that a Lagrangian isotopy with vanishing flux is homotopic to a Hamiltonian isotopy. For example, the Clifford and Chekanov tori are not Hamiltonian isotopic in $\C^n$, but are Lagrangian isotopic~\cite{Che96}. If both tori have the same monotonicity constant, such a Lagrangian isotopy must have vanishing flux by the area formula~\eqref{eq:areaformula}, since the Maslov class must be preserved. \par
	
	If one looks at the classical proof that a symplectic isotopy with vanishing flux is homotopic to a Hamiltonian one~\cite{Ban97}, we see that two problems arise in the Lagrangian setting: 
	\begin{enumerate}
	\item if $\{f_t\}$ and $\{g_t\}$ are ambient Lagrangian isotopies starting at $L$, $\{f_tg_t\}$ is not necessarily one, and
	\item even when we can compose ambient Lagrangian isotopies, the homotopy of the Lagrangian isotopy will result in a family of Lagrangian \emph{immersions}.
	\end{enumerate}
See Appendix~\ref{app:Lag_homotopies} for a solution to the first problem through $h$-principles.
\end{remark}

To close this subsection, we finally prove \cref{prop:few-cases_discreteness}.

\begin{proof}[Proof of \cref{prop:few-cases_discreteness}]
  As noted before, \cref{prop:flux-monodromy_monotone} gives a proof of Case~(c).

  Suppose that we are in Case~(a), i.e. $L$ is a weakly exact Lagrangian. If $L'\in\LHam(L)$ were $C^1$-close to $L$ but not Hamiltonian isotopic to $L$ in a Weinstein neighbourhood, then it would be the graph of a nonexact closed 1-form $\sigma$ in that neighbourhood. However, every nonexact closed 1-form is cohomologous to a 1-form with at most one zero~\cite{FarberSchutz2006}, i.e.\ $L'$ is Hamiltonian isotopic in a Weinstein neighbourhood of $L$ to some Lagrangian with only one intersection point with $L$. By the degenerate Lagrangian Arnol'd conjecture for weakly exact Lagrangians~\cite{Flo89c,Hof88}, this is not possible. Therefore, $\LHam(L)$ must be $C^1$-locally path connected. By \cref{prop:flux_C1-conjA_injection}, this implies the discreteness of $\fl(L)$~---~closedness follows from \cref{prop:discrete_closed_unparam} below.

  For Case~(b), i.e.\ when $L$ is a non-displaceable torus, the proof is essentially the same: if $L'$ were the graph of a nonexact closed 1-form, it would be displaceable from $L$ in a Weinstein neighbourhood, contradicting the nondisplaceability of $L$. 
\end{proof}

\subsection{$C^1$-closedness} \label{sec:closedness}
In~\cite{Sol13}, it is claimed that discreteness of $\fl(L)$ implies that the Hamiltonian orbit of $L$ in $X$ is $C^1$-closed in the space of Lagrangians of $X$ diffeomorphic to $L$. WHowever, \cite{Sol13} does not take into account monodromy and erroneously claims that $\fl(L)$ is a subgroup of $H^1(L;\R)$. We now fix this gap.

\begin{proposition} \label{prop:discrete_closed_unparam}
	If 0 is isolated in $\fl(L)$, then the Hamiltonian orbit of $L$ is $C^1$-closed in $\cl(L)$.
\end{proposition}

\begin{proof}
  Suppose that $\{L_i\}\subseteq\LHam(L)$ $C^1$-converges to some Lagrangian $K$. Then, for $i$ large, $L_i$ is the graph of a closed 1-form in a Weinstein neighbourhood of $K$. For each $i$, make a Lagrangian loop $\Lambda_i$ based at $L$ by concatenating
  \begin{enumerate}[label=(\arabic*)]
  \item a Hamiltonian isotopy from $L$ to $L_i$;
  \item the obvious Lagrangian isotopy from $L_i$ to $K$ in the Weinstein neighbourhood;
  \item the obvious Lagrangian isotopy from $K$ to $L_{i+1}$;
  \item a Hamiltonian isotopy from $L_{i+1}$ to $L$.
  \end{enumerate} 
  On the one hand, the flux of $\Lambda_i$ goes to 0 as $i\to\infty$ since $L_i\to K$ in the $C^1$ topology. On the other hand, 0 is isolated in $\fl(L)$. Therefore, $\Flux(\Lambda_i)=0$ for large $i$. But the flux of $\Lambda_i$ is simply the flux of the two obvious isotopies from $L_i$ to $L_{i+1}$, which we can compute to be $[\sigma_{i+1}-\sigma_i]$. Therefore, for $i$ large, $[\sigma_i]\in H^1(K;\R)$ is constant. Since $[\sigma_i]\to 0$ by $C^1$-convergence, we must have that $[\sigma_i]=0$, i.e. $L_i$ is the graph of an exact form for $i$ large. Therefore, $L_i$ is Hamiltonian isotopic to $K$ (in its Weinstein neighbourhood).
\end{proof}

By a slight adaptation of Solomon's arguments, we can prove the parametrized form of his statement. \par

\begin{proposition} \label{prop:discrete_closed_param}
	If $\Gamma(L)$ is discrete, then the Hamiltonian orbit of $\iota:L\hookrightarrow X$ is $C^1$-closed in $\pcl(L)$, the space of Lagrangian embeddings isotopic to $\iota:L\hookrightarrow X$.
\end{proposition}

\begin{proof}
  If $\Gamma(L)$ is discrete, then so is $G=\Flux(\ker \cm)\cong\Gamma_0(L)$. But then, $\Flux$ defines a map $\overline{\Flux}:\pcl\to H^1(L;\R)/G$, where $H^1(L;\R)/G$ is a manifold. It is easy to see that 0 is a regular value of $\overline{\Flux}$. Therefore, $\overline{\Flux}^{-1}(0)$ is a closed manifold. It follows essentially from Proposition~\ref{prop:path-lag} below or Corollary~6.4 of~\cite{Sol13} that the connected component of $\overline{\Flux}^{-1}(0)$ containing $\iota$ in $\pcl$~---~and thus in the space of all Lagrangian embeddings~---~is its Hamiltonian orbit, which concludes the proof.
\end{proof}

We also prove the converse, under certain topological conditions, by essentially following the proof in the symplectic case.

\begin{lemma} \label{lem:converse_discrete-to-closed}
	Suppose that $\iota_*:H_1(L;\R)\to H_1(X;\R)$ is injective or that $L$ is a product torus in $\C^n$. If $\LHam(L)$ (resp.\ $\LHam(\iota)$) is $C^1$-closed in $\cl(L)$ (resp.\ $\pcl(\iota)$), then $\fl(L)$ (resp.\ $\Gamma(L)$) is discrete.
\end{lemma}

\begin{proof}
  First, note that non-discreteness implies density in our setting. For $\Gamma(L)$ or for $\fl(L)$ with $\iota_*$ injective, this follows from the fact that these are subgroups of some real vector space. For the case when $L$ is a product torus, this follows from the classification of product tori by Chekanov~\cite{Che96}. \par

  Suppose that $\fl(L)$ is dense, and take a $C^1$-small closed 1-form $\sigma$ on $L$ such that $[\sigma]\in H^1(L;\R)\setminus\fl(L)$. Note that its graph $L'$ is not Hamiltonian isotopic to $L$ in $X$, since otherwise the concatenation of the obvious isotopy from $L$ to $L'$ with a Hamiltonian isotopy from $L'$ to $L$ would be a Lagrangian loop of flux $[\sigma]$. By an analogous construction as in \cref{lem:nondiscrete-flux_flux-C1-conjA} (also using Propositions~\ref{prop:flux_C1-conjA_elem-tori} and~\ref{prop:flux_C1-conjA_injection}), we can take a sequence of $C^1$-small closed 1-forms ${\sigma_k}$ with graphs $L_k$ such that
  \begin{enumerate}[label=(\roman*)]
  \item $L_k\to L'$ in $C^1$ topology;
  \item $L$ and $L_k$ are Hamiltonian isotopic in $X$;
  \item $[\sigma_k]\in\fl(L)$.
  \end{enumerate}
  Since $L'\notin\LHam(L)$, the Hamiltonian orbit is not closed. \par
  
  When $\Gamma(L)$ is dense, we mirror the construction using the density of $\Gamma_0(L)$, just as in \cref{lem:nondiscrete-flux_flux-C1-conjA}.
\end{proof}

\section{In $\mathbb{C}^2$: A case study}
\label{sec:C2}
\commentintoc{We review preliminary material about Arnol'd--Liouville tori in \S\,\ref{sec:prep-isot-arnold} and almost toric fibrations in \S\,\ref{sec:almost-toric-fibr}. We compute the flux-monodromy group of Lagrangians tori in $\C^{2}$ in \S\,\ref{sec:flux-monodr-morph}, and prove the existence of a section to the flux-monodromy morphism in \S\,\ref{ssec:immersed}. \cref{thm:B} follows from these results. We also establish a complete description of the fundamental group of the space of \emph{immersed} Lagrangians in \S\,\ref{sec:relation-immersions}. In \S\,\ref{ssec:ham-action}, we study the Hamiltonian action on $\cl(L)$, when $L$ is a Lagrangian torus in several natural 4-manifolds. In \S\,\ref{sec:hausdorff-local-path}, we prove \cref{prop:few-cases_discreteness}, which states a strong form of local path connectedness of the Hamiltonian orbit of any Lagrangian torus in $\C^{2}$ in the Hausdorff metric.}

We review preliminary material about Arnol'd--Liouville tori in \S\,\ref{sec:prep-isot-arnold} and almost toric fibrations in \S\,\ref{sec:almost-toric-fibr}. We compute the flux-monodromy group of Lagrangians tori in $\C^{2}$ in \S\,\ref{sec:flux-monodr-morph}, and prove the existence of a section to the flux-monodromy morphism in \S\,\ref{ssec:immersed}. \cref{thm:B} follows from these results. We also establish a complete description of the fundamental group of the space of \emph{immersed} Lagrangians in \S\,\ref{sec:relation-immersions}. In \S\,\ref{ssec:ham-action}, we study the Hamiltonian action on $\cl(L)$, when $L$ is a Lagrangian torus in several natural 4-manifolds. In \S\,\ref{sec:hausdorff-local-path}, we prove \cref{prop:few-cases_discreteness}, which states a strong form of local path connectedness of the Hamiltonian orbit of any Lagrangian torus in $\C^{2}$ in the Hausdorff metric.

\subsection{Preparation: isotopies of Arnol'd--Liouville tori}
\label{sec:prep-isot-arnold}
Although we will only use the following in the special case of an almost toric fibration on $\C^2$, we briefly discuss properties of the flux-monodromy morphism on lifts of curves in the base of completely integrable Hamiltonian systems. This discussion is based on \cite{Dui80} and \cite{Eva23}. Let $F \colon (X,\omega) \rightarrow \R^n$ be a completely integrable Hamiltonian system on a symplectic $2n$-manifold. This means that $\{F_i,F_j\}=0$ and that $F$ has compact, connected fibres. Furthermore, we ask that the set of regular values $B_{\rm reg} \subset \R^n$ is open and dense in the image of $F$. The classical Arnol'd--Liouville theorem implies that $F^{-1}(b) \subset X$ is a Lagrangian torus for all $b \in B_{\rm reg}$. These are called \emph{Arnol'd--Liouville tori}. Lifting paths in $B_{\rm reg}$ yields Lagrangian isotopies.

\begin{definition}
	\label{def:liftisotopy}
	Let $\gamma \colon [0,1] \rightarrow B_{\rm reg}$ be a smooth curve in the regular values of an integrable system $F$. We call the Lagrangian isotopy $\Pi_{\gamma} = \{F^{-1}(\gamma(t))\}_{t \in [0,1]}$ the \emph{lift of $\gamma$}.
\end{definition}

There is a strong relationship between action coordinates of the system $F$ and the Lagrangian flux of the lift of a curve. Following Evans \cite[Definition 2.14]{Eva23}, we fix a regular fibre $L = F^{-1}(b_0)$ and define the \emph{flux map} $I \colon \tilde{B}_{\rm reg} \rightarrow H^1(L;\R)$ on the universal cover $\tilde{B}_{\rm reg} \rightarrow B_{\rm reg}$ as the Lagrangian flux of the lift of a path between $b_0$ and $b$. Note that we have passed to the universal cover, since the Lagrangian flux may depend on the choice of path. From that point of view, we obtain
\begin{equation}
	\label{eq:flux_lifts}
	\Flux \Pi_{\gamma} = I(\tilde{\gamma}(1)) - I(\tilde{\gamma}(0)),
\end{equation}
where $\tilde{\gamma} \colon [0,1] \rightarrow \tilde{B}_{\rm reg}$ is a lift of $\gamma$ to the universal cover $\tilde{B}_{\rm reg}$. As it turns out, the map $I$ computes action coordinates.

\begin{proposition}
\label{prop:action_I}
Identifying $H^1(T^n;\R) = \R^n$, the map $I$ computes action coordinates on small enough subsets $U \subset \tilde{B}_{\rm reg}$. Conversely, any set of action coordinates coincides with $I$ up to an integral affine transformation, i.e.\ up to post-composing with a map in $\R^n \rtimes \GL(n;\Z)$.
\end{proposition}

This proposition is a combination of \cite[Lemma 2.15]{Eva23} and the uniqueness of action coordinates; see \cite[Section 2.5]{Eva23}. Note that if $F$ is already expressed in action coordinates, then $\Flux \Pi_{\gamma} = \gamma(1) - \gamma(0)$, where we identify $H^1(T^n;\R) \cong \R^n$ via the basis on $H_1(T^n)$ generated by the closed orbits generated by the components $F_i$ of $F$. An example of such an $F$ is any moment map generating an effective Hamiltonian $T^n$-action. This occurs, for example, when $F \colon X \rightarrow \R^n$ is a toric moment map.

\subsection{Almost toric fibrations}
\label{sec:almost-toric-fibr}

Recall that by $T(a,b) \subset \C^2$, we denote the product torus $S^1(a) \times S^1(b)$. For any $a> 0$, the torus of the form $T(a,a)$ is monotone and is called the \emph{Clifford torus}. We denote it by $\tcl(a)$. Another torus we will discuss is the so-called \emph{Chekanov torus} $\tch(a)$, for all $a > 0$. It was introduced in \cite{Che96}; see also \cite{EliPol97}. By now, there are several equivalent points of view on the Chekanov torus. In this paper, we will use Chekanov's original definition as well as the fact that the Chekanov torus appears as a monotone fibre in the almost toric fibration obtained from performing a so-called \emph{nodal trade} on the standard toric fibration on $\C^2$. Let us briefly outline the approach using almost toric fibrations (ATFs); for more details, we recommend \cite{Eva23}. The map $F \colon \C^2 \rightarrow \R^2_{\geqslant 0}$ defined by $F_i = \pi \vert z_i \vert^2$ for $i \in \{1,2\}$ is the standard moment map defining a toric Hamiltonian $T^2$-action. From a slightly different point of view, it can be viewed as a special type of completely integrable Hamiltonian system:
\begin{enumerate}
	\item Its singular values are on the boundary $\pp \R^2_{\geqslant 0}$. They consist of two families of singularities of \emph{elliptic-regular type} at points of the form $(z_1,0)$ for $z_1 \neq 0$ and $(0,z_2)$ for $z_2 \neq 0$, respectively. These singular points live over the interior of the edges of $\R^2_{\geqslant 0}$. These two families meet at the point $(0,0)$ which is of \emph{elliptic-elliptic type} and a fixed point of the $T^2$-action.
	\item Over the regular locus $\R^2_{>0}$, the map $F$ defines action coordinates whose corresponding angle coordinates are induced by the $T^2$-action. Thus, $F$ can be viewed as an Arnol'd--Liouville chart on the open dense set of regular points of $\C^2$. Its Arnol'd--Liouville tori are just the product tori $T(a,b)$.
\end{enumerate}

There is a smooth deformation $\{F_t \colon \C^2 \rightarrow \R^2 \}$ of the integrable system $F = F_0$, in which the elliptic-elliptic singularity disappears for $t > 0$ and a so-called \emph{focus-focus} singularity appears in the system. Under this deformation, the two elliptic-regular families are merged into one such family. In almost toric geometry, this deformation is called a \emph{nodal trade}, since the image of the focus-focus singularity under $F_t$ is called a \emph{node}. The node is isolated and has a neighbourhood in which every point except for itself is regular. In the spirit of Duistermaat \cite{Dui80}, one can try to recover action-coordinates on the complement of the focus-focus singularity. This fails, since the focus-focus singularity has non-trivial \emph{monodromy} and the latter is an obstruction to finding global action-coordinates; see \cite[Theorem 2.2]{Dui80}. Monodromy here means that the torus bundle over the punctured neighbourhood of a node is non-trivial. In fact, there is a basis in which it has monodromy given by the matrix 
\begin{equation}
	\label{eq:shearmatrix}
	\begin{pmatrix}
		1 & 1 \\
		0 & 1
	\end{pmatrix}.
\end{equation}
The best one can do is to remove a ray connecting the node to the boundary of $\im F_t \subset \R^2$ (or to infinity) and compute action coordinates on the complement of the ray. This yields the so-called \emph{almost toric base diagram}, depicted in the case of $\C^2$ in the middle and on the right-hand side of Figure~\ref{fig:nodal_trade}. In the middle picture, the Arnol'd--Liouville tori over the dashed line are Hamiltonian isotopic to Chekanov tori. 

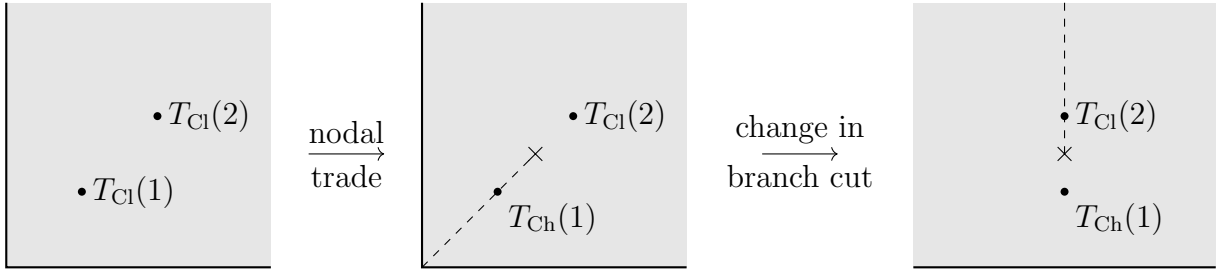
\begin{figure}[ht]
  \centering
  \begin{tikzpicture}
    \begin{scope}[shift={(-2.5,-1.5)}]
      \fill[opacity=0.1] (0,3.5) rectangle (3.5,0);
      \draw[thick] (0,3.5) -- (0,0) -- (3.5,0);
      \fill (1,1) circle[radius=1.5pt] node[anchor=west] {$\tcl(1)$};
      \fill (2,2) circle[radius=1.5pt] node[anchor=west] {$\tcl(2)$};
    \end{scope}
    \draw[->] (1.5,0) -- node[anchor=south] {nodal} node[anchor=north] {trade} (2.5,0);
    \begin{scope}[shift={(3,-1.5)}]
      \fill[opacity=0.1] (0,3.5) rectangle (3.5,0);
      \draw[thick] (0,3.5) -- (0,0) -- (3.5,0);
      \draw[dashed] (0,0) -- (1.5,1.5) node[cross] {};
      \fill (1,1) circle[radius=1.5pt] node[anchor=north west] {$\tch(1)$};
      \fill (2,2) circle[radius=1.5pt] node[anchor=west] {$\tcl(2)$};
    \end{scope}
    \draw[->] (7.5,0) -- node[anchor=south] {change in} node[anchor=north] {branch cut} (8.5,0);
    \begin{scope}[shift={(9.5,-1.5)}]
      \fill[opacity=0.1] (0,3.5) rectangle (4,0);
      \draw[thick]  (0,0) -- (4,0);
      \draw[dashed] (2,3.5) -- (2,1.5) node[cross] {};
      \fill (2,1) circle[radius=1.5pt] node[anchor=north west] {$\tch(1)$};
      \fill (2,2) circle[radius=1.5pt] node[anchor=west] {$\tcl(2)$};
    \end{scope}
  \end{tikzpicture}
  \caption{Almost toric base diagrams corresponding to the nodal trade in~$\C^2$. The image of $F$ is depicted on the left-hand side. The Clifford torus $\tcl(a)$ is the fibre over $(a,a)$. In the middle is the almost toric base diagram after a nodal trade, whose node sits between the points $(1,1)$ and $(2,2)$. The torus over $(1,1)$ is the Chekanov torus $\tch(1)$ and the one over $(2,2)$ is the Clifford torus $\tcl(2)$. On the right hand side is another ATF-base diagram of the same fibration obtained from the middle one by applying a change in branch cut and a shear transformation in $\GL(2;\Z)$.}
  \label{fig:nodal_trade}
\end{figure}

We use the almost toric base diagram in the right hand side of Figure~\ref{fig:nodal_trade} to define a non-trivial Lagrangian loop $\Xi \in \pi_1(\cl,L)$ for $L = \tcl$ or $L = \tch$.

\begin{proposition}
\label{prop:C2}
Let $L$ be any Clifford or Chekanov torus in $\C^2$. Then, there is a loop $\Xi \in \pi_1(\cl,L)$ with
\begin{equation}
    \label{eq:sigma}
	\FM(\Xi) = \left(
	0,
	\begin{pmatrix}
		1 & 1 \\
		0 & 1
	\end{pmatrix}
	\right)
\end{equation}
in a suitable basis of $H_1(T^2)$, which yields an identification of $H^1(T^2;\R^2) \rtimes \MCG(T^2) \cong \R^2 \rtimes \GL(2;\Z)$. 
\end{proposition}

\begin{proof}
  Let us prove the claim for the Clifford torus $L = \tcl(a)$. The proof for the Chekanov torus is analogous. As illustrated in Figure~\ref{fig:nodal_trade}, a nodal trade followed by a change in branch cut and a $\GL(2;\Z)$-shear produce an almost toric base diagram of $\C^2$. It is obtained by decorating the upper half-plane $\R \times \R_{\geqslant 0}$ with a toric boundary on the edge $\R \times \{0\}$ and a node at $(0,c)$, with the branch cut line $\{0\} \times [c,+\infty)$. By nodal slides, we can choose $c > 0$ freely without changing the symplectomorphism type of the space. This almost toric base diagram encodes a completely integrable system on $\C^2$ with a fibre at $(0,c)$ containing a focus-focus singularity and whose Arnol'd--Liouville fibre over the point $(0,a)$ is the Clifford torus $\tcl(a)$. The set of regular values of the corresponding almost toric fibration is given by $B_{\rm reg} = \R \times \R_{>0} \setminus \{(0,c)\}$.
  
  \begin{figure}[ht]
  	\centering
  	\begin{tikzpicture}[scale=1.25]
  		\fill[opacity=0.1] (0,4) rectangle (4,0);
  		\draw[thick] (0,0) -- (4,0);
  		\draw[dashed] (2,4) -- (2,2) node[cross] {};
  		\fill (2,3) circle[radius=1.3pt] node[anchor=south east] {$\tcl$};
  		\fill (2.8,1) circle[radius=0pt] node[anchor=south west] {$\xi$};
  		\draw[
  		decoration={markings, mark=at position 0.9 with {\arrow{<}}},
  		postaction={decorate}
  		]
  		(2,2) circle (1);
  	\end{tikzpicture}
  	\caption{The loop $\xi$ which lifts to the Lagrangian loop $\Xi \in \pi_1(\cl,\tcl)$.}
  	\label{fig:loop_xi}
  \end{figure}
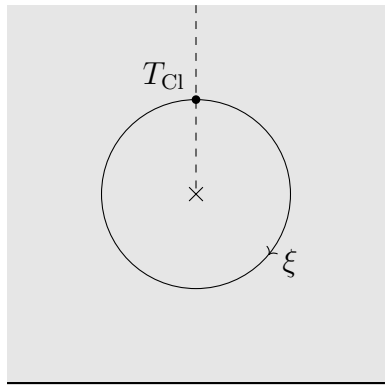

  Now let $\xi \colon S^1 \rightarrow \R \times \R_{\geqslant 0}$ be a loop with $\xi(0) = \xi(1) = (0,a)$ going once around the node at $(0,c)$ in the clockwise direction; see Figure~\ref{fig:loop_xi}. The precise shape of $\xi$ is irrelevant; only its homotopy class in $B_{\rm reg}$ preserving $\xi(0)=\xi(1)=(0,a)$ matters. The lift, as in Definition~\ref{def:liftisotopy}, of $\xi$ yields an element~$\Xi \in \pi_1(\cl,\tcl)$. Recall that the almost toric base diagram allows us to pick action coordinates $F=(F_1,F_2)$ on the complement $\R \times \R_{\geqslant 0} \setminus \left(\{0\} \times [c,+\infty)\right)$ of the branch cut. The branch cut tells us that these action coordinates are identified by the shear $(F_1,F_2) \mapsto (F_1,F_2+F_1)$ when traversing the branch cut from right to left; see \cite[Chapter 7]{Eva23}. Now let $f_1,f_2 \in H_1(\tcl)$ be the classes defined by the closed orbits generated by the action coordinates $F_1,F_2$. The coordinate $F_1$ is well defined, but since $\tcl$ is a fibre on the branch cut line, $F_2$ is defined only up to adding integer multiples of $F_1$. We pick one such choice for $F_2$ to define $f_2$. In that basis, the topological monodromy of the torus bundle over the image of $\xi$ (traversed in the clockwise direction) is precisely the matrix in \eqref{eq:shearmatrix}; see \cite[Lemma 7.2]{Eva23} (note that in Evans' book, the coordinates are rotated by $\frac{\pi}{2}$). This proves that the monodromy $\cm(\Xi)$ is the one claimed in \eqref{eq:sigma}. To compute the flux of $\Xi$, we use~\eqref{eq:flux_lifts}. We claim that $I(\tilde{\xi}(1)) = I(\tilde{\xi}(0))$ for any lift $\tilde{\xi}$ of $\xi$ to the universal cover $\tilde{B}_{\rm reg}$. Indeed, by Proposition~\ref{prop:action_I}, the flux map $I$ coincides with $F$ up to integral affine transformations. Since the loop $\xi$ (in the image of $F$) closes up, we find that $I(\tilde{\xi}(1)) = I(\tilde{\xi}(0))$ by continuity of $I$.
\end{proof}

\subsection{The flux-monodromy morphism for tori}
\label{sec:flux-monodr-morph}

Let us now prepare and prove Theorem~\ref{thm:B}. Besides the loop $\Xi \in \pi_1(\cl,\tcl)$ constructed in Proposition~\ref{prop:C2}, we define one other loop based at the Clifford torus $\tcl \subset \C^2$ needed to generate $\Gamma(\tcl)$. Let $\phi_t \in \Ham \C^2$ be a Hamiltonian isotopy satisfying $\phi_1(z_1,z_2) = (z_2,z_1)$ on a compact neighbourhood of $\tcl$. For example, one can choose $\phi_t \in U(2)$ and use a cut-off function to make it compactly supported. Now set 
\begin{equation}
	\label{eq:Sigmaloop}
	\Sigma = \{\phi_t(\tcl)\}_{t \in [0,1]}.
\end{equation}
Although $\phi_t$ itself is not a loop, its time one map $\phi_1$ maps $\tcl$ to itself, and thus we obtain an element $\Sigma \in \pi_1(\cl,\tcl)$. Its flux vanishes, since it is induced by a Hamiltonian flow. \par

In the basis defined by $e_1 = S^1 \times \{*\}$ and $e_2 = \{*\} \times S^1$, the monodromy of $\Sigma$ swaps coordinates, $\cm(\Sigma)e_1 = e_2$ and $\cm(\Sigma)e_2 = e_1$. As it turns out, this is the only non-trivial monodromy that can be realized by a Lagrangian loop induced by a Hamiltonian isotopy.
\begin{proposition}
	\label{prop:deform_to_ham}
    Let $\Lambda$ be a loop in $\cl$ based at $\tcl$. If $\Lambda$ can be deformed to a loop in $\LHam(\tcl)$, then, in the basis $e_1,e_2$, we have that
    \begin{equation}    
        \label{eq:swap}
        \cm(\Lambda)
        \in
        \left\{
        \begin{pmatrix} 
            1 & 0 \\ 
            0 & 1
        \end{pmatrix},
        \begin{pmatrix} 
            0 & 1 \\ 
            1 & 0
        \end{pmatrix}
        \right\}.
    \end{equation}
\end{proposition}

\begin{proof}
  The monodromy is an invariant up to deformation of loops. Every loop in $\LHam(\tcl)$ is induced by a Hamiltonian isotopy; see \cref{prop:path-lag}. The conclusion then follows from Chekanov's classification of possible monodromies of symplectomorphisms mapping a product torus to a product torus \cite[Theorem 4.5]{Che96}; see also later work by Yau~\cite{Yau09}. Heuristically, the elements $e_1,e_2 \in H_1(L)$ appear as boundaries of Maslov two holomorphic disks, and hence the set $\{e_1,e_2\}$ must be preserved by any monodromy element realized by a Hamiltonian isotopy.
\end{proof}

\begin{corollary}
	The loop $\Xi$ from Proposition~\ref{prop:C2} cannot be deformed into $\LHam(\tcl)$. 
\end{corollary}

This follows immediately from the fact that $\cm(\Xi)$ has infinite order. Note that this is consistent with the fact that $\Xi$ must traverse a wall of Chekanov tori~---~see Figure~\ref{fig:loop_xi}~---~which are not Hamiltonian isotopic to the Clifford torus. 

Let us now prove that the loops $\Xi,\Sigma \in \pi_1(\cl,\tcl)$ generate the image of the flux-monodromy map. This follows from the elementary fact that Lagrangian isotopies preserve the Maslov class. First, note that we have used different bases to express the monodromy of $\Xi$ and that of $\Sigma$. In what follows, we will use the basis $f_1,f_2 \in H_1(\tcl)$ defined in the proof of Proposition \ref{prop:C2}, since it yields cleaner expressions for monodromy. The basis $e_1,e_2 \in H_1(\tcl)$ is related to it by the shear $f_1 = e_1 - e_2, f_2 = e_2$. Hence, in the basis $f_1,f_2$, we obtain 
\begin{equation}
	\cm(\Sigma) = \begin{pmatrix} -1 & 1 \\ 0 & 1 \end{pmatrix}.
\end{equation}
Similarly, the Maslov class is $m_{\tcl} = 2(e_1^* + e_2^*) = 2f_2^* \in H^1(\tcl)$ under the natural isomorphism $H^{2}(\C^{2}, \tcl)\cong H^{1}(\tcl)$.

\begin{proposition} \label{prop:image-FM_clifford-C2}
    Let $T_{\rm Cl}$ be a Clifford torus in $\C^2$. In the basis $f_1=e_1 - e_2, f_2 = e_2$ of $H_1(\tcl)$, we have that
    \begin{align}
    	\label{eq:CliffordGamma}
      \Gamma(\tcl) = \left\{\left(0,
      \begin{pmatrix}
            \pm 1 & k \\ 0 & 1
      \end{pmatrix}\right)
      \Big| \; k\in\Z
      \right\}\cong \Z\rtimes\Z_2.
    \end{align}
    In particular, the elements $\Xi,\Sigma$ generate $\Gamma(\tcl)$.
\end{proposition}

\begin{proof}
  By \cref{cor:flux_monotone}, the flux is always zero, so we only need to compute the possible monodromies. The matrices $\cm(\Xi)$ and $\cm(\Sigma)$ generate the group on the right-hand side of \eqref{eq:CliffordGamma}, thus proving the inclusion $\supseteq$. To prove that $\Gamma(\tcl)$ does not contain any other elements, let
  \begin{align*}
    M=\begin{pmatrix}
        a & b \\ c & d
      \end{pmatrix}
    \in \GL(2;\Z),
  \end{align*}
  be a monodromy matrix expressed in the basis $f_1,f_2$. Since the Maslov index is preserved by Lagrangian isotopies, we find that $m_{T_{\rm Cl}}M = m_{T_{\rm Cl}}$ for $m_{T_{\rm Cl}}=(0, 2)$. It follows that $c=0$ and $d=1$. Since $M \in \GL(2;\Z)$, we find $ad-bc=\pm 1$, so that $a=\pm 1$. This proves the inclusion $\subseteq$.
\end{proof}

We can now prove the first part of \cref{thm:B}, i.e.\ everything but the existence of a section of $\FM$. \par

\begin{proof}[Proof of the 1$^{st}$ part of \cref{thm:B}]
  By~\cite{GooIvrRiz16}, $L$ is Lagrangian isotopic to $\tcl$. Let $\Pi$ be a path in $\mathcal{L}$ starting at $\tcl$ and ending at $L$. This fixes a preferred isomorphism $\pi_1(\cl,\tcl) \cong \pi_1(\cl,L)$, namely the one given by $\Lambda \mapsto \Pi^- * \Lambda * \Pi$. On $H_1(L)$, we fix the basis $g_1 =(\phi^{\Pi})_*f_1, g_2 = (\phi^{\Pi})_*f_2$, where $\phi^{\Pi}$ is defined as in \cref{ssec:basic}. Now let $\Lambda' = \Pi^- * \Lambda * \Pi \in \pi_1(\cl,L)$, and suppose that $\Flux \Pi = (x,y) \in H^1(\tcl;\R) \cong \R^2$ in the basis $f_1^*,f_2^*$. Then the change-of-basepoint formula from Proposition~\ref{prop:fm_basic} shows that there is $k \in \Z$ such that
    \begin{align*}
        \Flux(\Lambda')
        &= \cm(\Lambda)^* \Flux \Pi + \Flux \Lambda - \Flux \Pi \\
        &= \begin{pmatrix}
		\pm 1 & 0 \\
		k & 1
	   \end{pmatrix}\begin{pmatrix}
	    x \\ y
	   \end{pmatrix}-\begin{pmatrix}
	    x \\ y
	   \end{pmatrix} = x\begin{pmatrix}
	    \pm 1-1 \\ k
	   \end{pmatrix}
    \end{align*}
and 
	\begin{equation}
		\cm(\Lambda')
		=
		\begin{pmatrix}
            \pm 1 & k \\ 0 & 1
        \end{pmatrix}
	\end{equation}
        in the basis $g_1,g_2 \in H_1(L)$ and its dual basis. Since $\Lambda \in \pi_1(\cl,\tcl)$, we have used Proposition~\ref{prop:image-FM_clifford-C2} to determine $\cm(\Lambda)$ and to show that $\Flux(\Lambda) = 0$.
      \end{proof}

Interestingly, we find the following characterization of monotone tori in $\C^2$ in terms of their flux-monodromy group, as a consequence of \cref{thm:B}.

\begin{corollary} \label{cor:charac_monotone-tori_C2}
A Lagrangian torus $L \subset \C^2$ is monotone if and only if $\Gamma(L)$ contains a non-trivial element of the form $(0,M)$.
\end{corollary}

\begin{proof}
  Let $L \subset \C^2$ be monotone with $[\omega_{0}]_L = \frac{C}{2} m_L$ for $C > 0$. Let $\Pi$ be a path of Lagrangians starting at $\tcl = T(1,1)$ and ending at $L$. As in the proof of \cref{thm:B}, we define the basis $g_i = \phi^{\Pi}_*f_i \in H_1(L)$ for $i \in \{1,2\}$. As we have seen in the proof of Proposition~\ref{prop:image-FM_clifford-C2}, the Maslov class is given by $m_{\tcl} = 2f_2^*$ and thus $m_L = 2g_2^*$. Since $(\C^2, \omega_0 = d\lambda)$ is exact, we use monotonicity to find $[\lambda]_L = C g_2^*$ and $[\lambda]_{\tcl} = f_2^*$ in case of the Clifford torus. The area identity \eqref{eq:areaformula} for the path $\Pi$ yields 
  \begin{equation*}
    \Flux \Pi
    = (\phi^{\Pi})^* [\lambda]_{L} - [\lambda]_{\tcl} 
    = (\phi^{\Pi})^*Cg_2^*  - f_2^*
    = (C-1)f_2^*.
  \end{equation*}
  In other words, we have $\Flux \Pi = (0,C-1)$ in the basis $f_1^*,f_2^*$. Since $x=0$, using the notation of \cref{thm:B}, the flux-monodromy group $\Gamma(L)$ contains elements of the form $(0,M)$ with $M \neq \id$ by \cref{thm:B}. Conversely, suppose that there is a loop $\Lambda \in \pi_1(\cl,L)$ with $\FM(\Lambda) = (0,M)$ for $M \neq \id$. We again use the area identity to find $M^*[\lambda]_L - [\lambda]_L = 0$, meaning that $M^*$ fixes the area class. Since it fixes the Maslov class, too, this implies that $M = \id$ if the area and Maslov classes are linearly independent.
\end{proof}

\subsection{A section for the flux-monodromy morphism}
\label{ssec:immersed}
We now complete the proof of Theorem~\ref{thm:B}. More precisely, we prove the following stronger statement.

\begin{theorem} \label{thm:section_FM}
	Let $L$ be a Lagrangian torus in $\C^2$. The flux-monodromy morphism $\FM:\pi_1(\cl)\to\Gamma(L)$ admits a section, i.e.\ there is a morphism $s:\Gamma(L)\to\pi_1(\cl)$ such that $\FM\circ \,s=id$.
\end{theorem}

The proof follows from a detailed understanding of the generators $\Xi$, $\Sigma$ of $\Gamma(L)$ above and will take up the remainder of this subsection.

\medskip
We start with an explicit description of the $\Z_2$ factor. \par

\begin{lemma} \label{lem:switch^2}
	There is a Hamiltonian Lagrangian loop exchanging the two factors of $\tcl=S^1(1)\times S^1(1)$ such that the following holds. Let $\Sigma\in \pi_1(\cl,\tcl)$ be the class this loop represents. We may choose the Hamiltonian isotopy so that $\Sigma^2$ is trivial in $\LHam(\tcl)$~---~and thus \emph{a fortiori} in $\cl(\tcl)$.
\end{lemma}

\begin{proof}
  Note that the standard action of $\mathrm{U}(2)$ on $\C^2$ is Hamiltonian. Therefore, if $A:[0,1]\to \mathrm{U}(2)$ is a path of matrices starting at the identity and with
  \begin{align*}
    A(1)=\begin{pmatrix}
           0 & 1 \\ 1 & 0
         \end{pmatrix},
  \end{align*}
  then we can realize $\Sigma$ as the Hamiltonian isotopy $t\mapsto A(t)\cdot \tcl$. \par

  Note that $\det A^2:[0,1]\to S^1$ is a loop and that $\Sigma^2$ is represented by $t\mapsto A(t)^2\cdot \tcl$. Since the determinant induces an isomorphism $\pi_1(\mathrm{U}(2))\xrightarrow{\sim} \pi_1(S^1)=\Z$, the loop $A^2$ in $\mathrm{U}(2)$ must be homotopic to the loop $t\mapsto \mathrm{diag}(e^{2\pi k it},1)$ for some $k\in\Z$. In particular, $t\mapsto A(t)^2\cdot \tcl$ is homotopic to the constant loop at $\tcl$ through Hamiltonian isotopies.
\end{proof}

\begin{remark} \label{rem:parametrized-Sigma}
	As we shall see below, the loop $t\mapsto \mathrm{diag}(e^{2\pi it},1)\cdot\tcl$ is not contractible as a loop of Lagrangian immersions~---~and thus \textit{a fortiori} in $\widetilde{\cl}$. In fact, it represents an element of infinite order in the fundamental group of that space. This shows an important difference between the topology of unparametrized Lagrangians and that of parametrized ones.
\end{remark}

The previous lemma gives a section from the $\Z_2$ factor of $\Gamma(\tcl)$. Moreover, the fact that $\cm (\Xi)$ has infinite order in the monodromy group of $\tcl$ ensures that $\Xi$ also has infinite order in $\pi_1(\cl,\tcl)$. Therefore, Theorem~\ref{thm:section_FM} follows directly from the following proposition.

\begin{proposition} \label{prop:XiSigmaXiSigma_contractible}
	The loop $\Xi\Sigma\Xi\Sigma^{-1}$ is contractible in $\cl$.
\end{proposition}

\begin{proof}
  Let $t\mapsto A(t)\in\mathrm{U}(2)$ be the path of unitary matrices realizing the loop $\Sigma$, as in the proof of Lemma~\ref{lem:switch^2}.
  We fix an isotopy $\{\phi^t\} \subset \mathrm{Diff}(\C^2)$ realizing $\Xi$, i.e.\ $\Xi=\{\phi^t(\tcl)\}$. We define parametrizations $I=\{\iota_t\}$ and $I'=\{\iota_t'\}$ of $\Xi$ by setting respectively $\iota_t=\phi^t\circ\iota$ with $\iota_0=\iota$, the standard parametrization of the Clifford torus, and $\iota_t'=\phi^t\circ\iota_0'$ with $\iota_0'=A(1)\circ\iota_1$. Notice that $\iota_t'=\phi^t\circ A(1)\circ\phi^1\circ\iota$ and that $\Xi\Sigma\Xi\Sigma^{-1}$ is parametrized by the concatenation
  \begin{align*}
    \{\iota_t\}\sharp\{A(t)\circ\iota_1\}\sharp\{\iota_t'\}\sharp\{A(t)^{-1}\circ\iota_1'\}.
  \end{align*}
  
  \begin{claim}
  	\label{claim:path_homotopy_pcl}
  	This path is homotopic relative to endpoints to $\{\iota_t\}\sharp\{A(1)\circ\iota_t'\}$ in $\pcl$.
  \end{claim}

  \begin{claim}
  	\label{claim:isotopy_nice-rep}
  	The isotopy $\{\phi^t\}$ may be chosen so that $A(1)\circ\phi^t=\phi^{1-t}\circ(\phi^1)^{-1}\circ A(1)$.
  \end{claim}

  Assuming both claims for now, the proposition easily follows. Indeed,
  \begin{align*}
    A(1)\circ\iota_t'=\phi^{1-t}\circ(\phi^1)^{-1}\circ A(1)\circ A(1)\circ\phi^1\circ\iota=\phi^{1-t}\circ\iota=\iota_{1-t},
  \end{align*}
  so  that $\{A(1)\circ\iota_t'\}$ is simply the path $\{\iota_t\}$ backwards.
\end{proof}

\begin{proof}[Proof of Claim~\ref{claim:path_homotopy_pcl}]
  First, note that if $I=\{\iota_t\}$ is the parametrization of a Lagrangian isotopy and $\{\psi_t\}$ is a symplectic isotopy starting at the identity, then
  \begin{align*}
    \{\iota_t\}\sharp\{\psi_t\circ\iota_1\}\approx \{\psi_t\circ\iota_t\}\approx \{\psi_t\circ\iota_{0}\}\sharp\{\psi_1\circ\iota_t\}.
  \end{align*}
  To see this, we use an old trick. Define $F:[0,1]^2\to \pcl$ by $F(s,t)=\psi_t\iota_s$. Going from $(0,0)$ to $(1,1)$ in the square via the lower edges gives the first path, via the diagonal the second path, and via the upper edges the last path; see Figure~\ref{fig:square_homotopy}. Then, any homotopy in the square between these paths defines a homotopy in $\pcl$ between the above paths. \par

  \begin{figure}[ht]
    \centering
    \begin{tikzpicture}[scale=4]
      \draw[thick] (0,0) rectangle (1,1);
      \draw[thick] (0, 0) -- (1, 1);
      \node at (0.5,-0.1) {$\iota_t$};
      \node at (1.1,0.5) {$\psi_t\iota_1$};
      \node at (-0.1,0.5) {$\psi_t\iota$};
      \node at (0.5,1.1) {$\psi_1\iota_t$};
      \node at (0.57,0.43) {$\psi_t\iota_t$};
    \end{tikzpicture}
    \vspace*{8pt}
    \caption{The square realizing the homotopy.}
    \label{fig:square_homotopy}
  \end{figure}
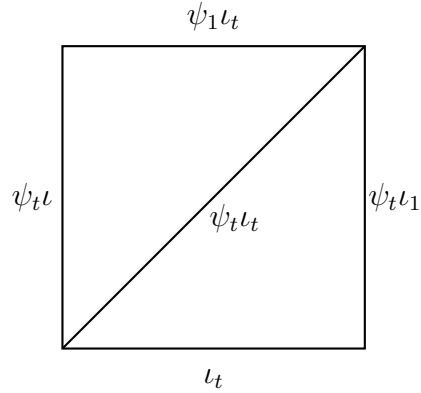

  To conclude, we simply apply these homotopies to our setting:
  \begin{align*}
    \{\iota_t\}\sharp\{A(t)\circ\iota_1\}\sharp\{\iota_t'\}\sharp\{A(t)^{-1}\circ\iota_1'\}
    &\approx \{\iota_t\}\sharp\{A(t)\circ\iota_1\}\sharp\{A(t)^{-1}\circ\iota_0'\}\sharp\{A(1)\circ\iota_t'\} \\
    &= \{\iota_t\}\sharp\{A(t)\circ\iota_1\}\sharp\{A(t)^{-1}\circ A(1)\circ \iota_1\}\sharp\{A(1)\circ\iota_t'\} \\
    &\approx \{\iota_t\}\sharp\{A(t)^{-1}\circ A(t)\circ\iota_1\}\sharp\{A(1)\circ\iota_t'\} \\
    &= \{\iota_t\}\sharp\{\iota_1\}\sharp\{A(1)\circ\iota_t'\} \\
    &\approx \{\iota_t\}\sharp\{A(1)\circ\iota_t'\}
  \end{align*}
  Here, we have made use of the fact that $A(1)^{-1}=A(1)$ and $\iota_0'=A(1)\circ\iota_1$.
\end{proof}

\begin{proof}[Proof of Claim~\ref{claim:isotopy_nice-rep}]
  Here, we use the explicit \emph{Auroux (integrable) system} underlying the almost toric base diagram depicted in \cref{fig:loop_xi}. It is explicitly given by
  \begin{equation*}
    F \colon \C^2 \rightarrow \R^2, \quad
    (z_1,z_2) \mapsto (\pi \vert z_2 \vert^2 - \pi \vert z_1 \vert^2, \vert z_1z_2 - c \vert^2)
  \end{equation*}
  for some $c > 0$. It has image $\R \times \R_{\geqslant 0}$, and computing action coordinates on the complement of the ray $\{0\} \times [c^2,+\infty)$ yields the same diagram as in \cref{fig:loop_xi}; see \cite[\S 7.1]{Eva23} for more details. This integrable system has the following properties:
  \begin{enumerate}
  \item For all $d > c$, the fibre $F^{-1}(0,d^2)$ is Hamiltonian isotopic to the Clifford torus $T_{\rm Cl}(a(d))$ for some $a(d) > 0$.
  \item We have $F \circ A(1) = I \circ F$, where $I$ denotes the involution $I(x,y) = (-x,y)$ on $\R^2$. 
  \end{enumerate}
  The second property is immediate by the definition of $F$. The first property follows from the fact that, under the (singular) reduction by the $S^1$-action generated by $F_1 = \pi \vert z_2 \vert^2 - \pi \vert z_1 \vert^2$, both $T_{\rm Cl}(a)$ and the fibre $F^{-1}(0,d^2)$ map to a circle enclosing the puncture of the reduced space. Setting $a(d) > 0$ equal to the symplectic area of the reduced circle of $F^{-1}(0,d^2)$ yields the claim. 

  The element $\Xi$ can be viewed as the lift of a curve running through a circle centred at $(0,c^2)$ in the image of $F$ in $\R^2$. From now on, we fix $c > 0$ such that $a(c) < 1$. This means that there is some $d > c$ with $F^{-1}(0,d^2)$ Hamiltonian isotopic to $T_{\rm Cl}(1) \subset \C^2$. Therefore, up to conjugation by a Hamiltonian isotopy, we can assume that $F^{-1}(0,d^2) = T_{\rm Cl}(1)$. Take $X'$ to be the vector field on $\R^2$ whose flow $\{(\psi')^t\}_{t\in [0,1]}$ is a full rotation around the singular value $(0,c^2)$ of the Auroux system $F$. Let $A_1\Subset A_2$ be annuli centered at a point $(0,x)$ such that $(0,d^2)\in A_1$ and that $A_2$ is disjoint from $(0,c^2)$ and $\R\times\{0\}$.
  Take $\beta:\R^2\to [0,1]$ rotationally symmetric such that $\beta|_{A_1}\equiv 1$ and $\supp\beta\subseteq A_2$. We then set $X:=\beta X'$. \par

  Since the flow $\{\psi^t\}$ of $X$ is supported in the regular locus of $F$, we can lift it to a flow $\{\phi^t\}$ of $\C^2$. By construction, $\phi^t(0,d^2)$ is homotopic to the loop $\xi$ of the proof of Proposition~\ref{prop:C2}. In particular, $\phi^t(\tcl)$ is the loop $\Xi$. By construction, we have that $I\circ\psi^t=\psi^{-t}\circ I$. By point (2) above, $A(1)$ is a lift of the involution $I$ on the base, so that we have $A(1)\circ\phi^t=\phi^{-t}\circ A(1)$. But since $\phi^t$ is obtained from integration by a time-independent vector field, we have that
  \begin{equation*}
    (\phi^t)^{-1}=\phi^{-t}=\phi^{1-t}\circ(\phi^1)^{-1},
  \end{equation*}
  thus giving the desired relation. 
\end{proof}

\subsection{Relation to immersions}
\label{sec:relation-immersions}
We now relate our study of $\cl$ to what happens in the space
\begin{align*}
	\cl_{imm}:=\left\{f:\T^2\looparrowright \C^2\ \middle|\ f^*\omega=0\right\}\!/\,\mathrm{Diff}(\T^2)
\end{align*}
of unparametrized immersed Lagrangian tori in $\C^2$. 

Note that there is an obvious inclusion $\cl\subseteq\cl_{imm}$. Likewise, whenever $L$ is embedded, $\FM_L$ factors through $\pi_{1}(\cl_{imm},L)$. Therefore, \cref{thm:section_FM} also gives a splitting $\pi_1(\cl_{imm},\tcl)=N\rtimes(\Z\rtimes\Z_2)$. However, we can prove more.

\begin{proposition} \label{prop:pi_1-cl_imm}
	The fundamental group of $\cl_{imm}$ based at $\tcl$ is isomorphic to $\Z\oplus(\Z\rtimes\Z_2)$, i.e.\ $N=\Z$ sits in the centre of $\pi_1(\cl_{imm},\tcl)$.
\end{proposition}

The proposition follows from \cref{prop:immersions_ses} and \cref{lem:action-conj} below. 

\begin{proposition} \label{prop:immersions_ses}
	The monodromy map $\cm$ on $\pi_1(\cl_{imm},\tcl)$ has the same image as in the embedded case, and its kernel is infinite cyclic. 
\end{proposition}

\begin{proof}
  Let $\pcl_{imm}$ be the space of Lagrangian immersions $\T^2 \looparrowright \C^2$. The fibration $\mathrm{Diff}(\T^2)\hookrightarrow \pcl_{imm}\twoheadrightarrow \cl_{imm}$ induces an exact sequence
  \begin{center}
    \begin{tikzcd}
      \pi_1(\mathrm{Diff}(\T^2)) \arrow[r,"i_1"] & \pi_1(\pcl_{imm}) \arrow[r,"j_1"] & \pi_1(\cl_{imm}) \arrow[r,"\del"] & \pi_0(\mathrm{Diff}(\T^2)) \arrow[r,"i_0"] & \pi_0(\pcl_{imm}).
    \end{tikzcd}
  \end{center}
  Since $\mathrm{Diff}(\T^2)$ deformation retracts onto $\T^2\times\mathrm{GL}(2;\Z)$~\cite{Gramain1973}, we have that $\pi_0(\mathrm{Diff}(\T^2))=\mathrm{GL}(2;\Z)$ and $\pi_1(\mathrm{Diff}(\T^2))=\Z^2$. \par

  Likewise, the Gromov--Lees $h$-principle~\cite{Gromov1970,Lees1976} implies that $\pcl_{imm}$ is homotopy equivalent to $C^\infty(\T^2,\mathrm{U}(2))$. Therefore, we have that 
  \begin{align*}
    \pi_0(\pcl_{imm})&=[\T^2,\mathrm{U}(2)]=[\T^2,S^1]=H^1(\T^2;\Z)=\Z^2 \\
    \intertext{and}
    \pi_1(\pcl_{imm})&=\pi_1(\Lambda^2\mathrm{U}(2))=\pi_1(\Lambda\mathrm{U}(2))\oplus\pi_2(\Lambda\mathrm{U}(2)) \\
                   &=\pi_1(\mathrm{U}(2))\oplus\pi_2(\mathrm{U}(2))^2\oplus\pi_3(\mathrm{U}(2)) \\
                   &=\Z^2, 
  \end{align*}
  where $[\cdot,\cdot]$ denotes homotopy of unbased maps and $\Lambda X$ the free loop space of the space $X$. In the $\pi_0$ computation, we have made use of the fact that $\mathrm{U}(2)=S^1\times S^3$ (topologically) and that $S^1$ is a $K(\Z,1)$-space. The equality $[\T^2,S^3]=0$ follows from cellular approximation, because the 2-skeleton of $S^3$ may be chosen to be a point. In the $\pi_1$ computation, we have also used the adjunction $C^\infty(\T^2,\mathrm{U}(2))=C^\infty(S^1,C^\infty(S^1,\mathrm{U}(2)))$ together with the fact that the fibration $\Omega X\to \Lambda X\to X$ admits a section via constant loops, so that the associated long exact sequence splits in each degree. Note that this only implies in general that $\pi_1(\Lambda X)=\pi_1(X)\ltimes \pi_1(\Omega X)= \pi_1(X)\ltimes \pi_2(X)$. However, since $X$ is a topological group in our case, its fundamental group must be abelian, so that the semidirect product is trivial, i.e.\ it is direct. Likewise, this ensures that all connected components of $\Lambda X$ are homeomorphic, so that the splitting holds whatever base point we choose for $\pi_1(\Lambda X)$. \par

  It is well known that the identification of $\pi_0(\pcl_{imm})$ with $H^1(\T^2;\Z)$ sends a Lagrangian immersion to (half of) its Maslov class~---~see e.g.~\cite{Vit90}. Through this identification, $i_0$ must thus send a matrix $M\in \mathrm{GL}(2;\Z)$ to half the Maslov class $\frac{1}{2} m^*M$ of the embedding $i_0\circ M$, where $m \in H^1(\T^2;\Z)$ is the Maslov class of the standard embedding of the Clifford torus. 

  Moreover, recall that the homotopy long exact sequence of a fibration $F\hookrightarrow E\to B$ comes from the long exact sequence of the pair $(E,F)$ by identifying $\pi_k(E,F)$ with $\pi_k(B)$. Under this identification, the map $\pi_k(B)\to \pi_{k-1}(F)$ corresponds to the usual boundary map $\pi_k(E,F)\to \pi_{k-1}(F)$. Therefore, we may think of a loop in $\cl_{imm}$ as a path in $\pcl_{imm}$ with endpoints in the fibre over the Clifford torus, with $\del$ sending this path to its endpoints. But under the identification of the fibre over the Clifford torus with $\mathrm{Diff}(\T^2)$, this means that $\del$ is precisely the monodromy $\cm$. \par

  To summarize the previous steps, the monodromy map has the same image in the immersed case as in the embedded one, and there is an exact sequence
  \begin{center}
    \begin{tikzcd}
      \Z^2 \arrow[r,"i_1"] & \Z^2 \arrow[r,"j_1"] & \pi_1(\cl_{imm}) \arrow[r,"\cm"] & \Z\rtimes\Z_2 \arrow[r] & 0.
    \end{tikzcd}
  \end{center}

  Following the identifications above, we see that the generators of $\Z^2=\pi_1(\mathrm{Diff}(\T^2))$ are sent by $i_1$ to the classes represented by $t\mapsto \left((\theta,\phi)\mapsto\mathrm{diag}(e^{2\pi i (t+\theta)},e^{2\pi i \phi})\right)$ and $t\mapsto \left((\theta,\phi)\mapsto\mathrm{diag}(e^{2\pi i \theta},e^{2\pi i (t+\phi)})\right)$ in the fundamental group $C^\infty(\T^2,\mathrm{U}(2))$ based at the usual embedding of $\tcl$. Just as noticed in the proof of Lemma~\ref{lem:switch^2}, both these loops are homotopic and generate the $\pi_1(\mathrm{U}(2))=\Z$ factor of $\pi_1(\pcl_{imm})=\Z^2$. Equivalently, as maps $\T^3\to\mathrm{U}(2)=S^1\times S^3$, both maps pull back $d\theta$ on $S^1$ to $dt$ and have zero degree on $S^3$~---~it is not surjective~---~so that they both generate the first $\Z$ factor of $[\T^3,\mathrm{U}(2)]/[\T^2,\mathrm{U}(2)]$. Therefore, $j_1$ has kernel and image isomorphic to $\Z$, which proves the statement.
\end{proof}

\begin{remark}
  The fact that the monodromy has the same image in both the immersed and embedded worlds already follows from the facts that the based point $\tcl$ is \emph{embedded} so that the monodromy factors through $\pi_{1}(\cl_{imm},L)$, and that $\cm(\pi_{1}(\cl,\tcl))$ is already as large as it could be by \cref{prop:image-FM_clifford-C2}.
\end{remark}

As stated in Remark~\ref{rem:parametrized-Sigma}, we see from this proof that, seen as a parametrized loop, $\Sigma$ has infinite order in $\pi_1(\pcl_{imm})$. However, that group also has another generator that survives in $\pi_1(\cl_{imm})$; we denote it by $\Upsilon$.

\begin{remark} \label{rem:immersions_pi-k}
	Using the same approach, we can actually compute all homotopy groups of $\cl_{imm}$~---~or, at least, reduce it to another hard problem:
	\begin{align*}
		\pi_k(\cl_{imm})\cong\begin{cases}
			\pi_k(S^3)\oplus\pi_{k+1}(S^3)^2\oplus\pi_{k+2}(S^3) &\text{if } k\geqslant 3; \\
			\Z^3\oplus\Z_2 &\text{if } k= 2.
		\end{cases}
	\end{align*}
	In particular, it is finite for $k\geqslant 4$.
	
	Indeed, we have $\pi_k(\cl_{imm})=\pi_k(\pcl_{imm})$ for $k\geqslant 3$. For $k=2$, we instead have a short exact sequence
	\begin{center}
		\begin{tikzcd}
			0 \arrow[r] & \pi_2(\pcl_{imm}) \arrow[r] & \pi_2(\cl_{imm}) \arrow[r] & \Z \arrow[r] & 0.
		\end{tikzcd}
	\end{center}
	This is because $\pi_k(\mathrm{Diff}(\T^2))=\pi_k(\T^2)=0$ for $k\geqslant 2$ and $\Ker i_1=\Z$. But the above sequence must split since $\Z$ is free abelian. The result then follows from the same computation as above, together with knowing the first four homotopy groups of $S^3$.
\end{remark}

\begin{lemma} \label{lem:action-conj}
	The kernel of $\cm$ in $\pi_1(\cl_{imm})$ is in the centre of the group.
\end{lemma}

\begin{proof}
  From the proof of Proposition~\ref{prop:immersions_ses}, we see that this kernel is precisely the image of the natural map $\pi_1(\pcl_{imm})\to \pi_1(\cl_{imm})$. Therefore, the action of $\pi_1(\cl_{imm})$ on the kernel by conjugation comes from changing basepoints: If $F=\{f_t\}$ is a loop in $\pcl_{imm}$ based at $f=f_0$, $\Lambda$ is a loop in $\cl_{imm}$ based at $L=f(\T^2)$, and $G=\{g_t\}$ is a parametrization of $\Lambda$ with $g_0=f$, then
  \begin{align*}
    \Lambda\cdot j_1(F)\cdot \Lambda^{-1}=j_1'(G\sharp F\sharp \overline{G}),
  \end{align*}
  where $\overline{G}=\{g_{1-t}\}$. But for this formula to make sense, we need to identify $\pi_1(\pcl_{imm},g_1)$ with $\pi_1(\pcl_{imm},f)$, hence the prime on the right-hand side above. However, since $\pcl_{imm}$ is homotopy equivalent to a topological group, there is a canonical way to do so: conjugation by any path from $g_1$ to $f$. Moreover, looking again at the proof of Proposition~\ref{prop:immersions_ses}, we see that this isomorphism commutes with the identification $\pi_1(\pcl_{imm})=\Z^2$. Therefore, by taking the path $\overline{G}$, we have that $j_1'(G\sharp F\sharp \overline{G})=j_1(F)$, i.e.\ the action is trivial.
\end{proof}

\begin{remark} \label{rem:cl_imm-to-cl_mon}
	The $h$-principle for monotone Lagrangian immersions in $\C^n$~\cite[Theorem~D]{EvaKed15} implies that the connected component of $\cl_{imm}$ containing the Clifford torus is weakly homotopically equivalent to the space of monotone immersed Lagrangians~---~see also \cref{thm:flux-0-to-exact} below. In particular, this means that all loops in $\pi_1(\cl_{imm},\tcl)$ can be realized by \emph{exact} Lagrangian homotopies, i.e.\ with flux vanishing identically along the isotopy.
\end{remark}

\subsection{Hamiltonian action on the fundamental group} \label{ssec:ham-action}
For any Lagrangian $L\subseteq X$, there is a natural action of $\Ham(X)$ on $\cl(L)$ whose orbits are precisely the Hamiltonian orbits. In particular, this induces a map $\pi_1(\Ham(X))\to \pi(\cl,L)$. Note that this map has image in $\ker\FM_L$,  so that they are not detected by the above methods. \par

Nonetheless, this map can be nontrivial: in the case $\RP^2\subseteq\CP^2$, nontrivial loops in $\mathrm{PU}(3)$~---~which is a deformation retract of $\Ham(\CP^2)$~\cite{Gro85}~---~induce nontrivial loops in $\cl$ (see \cref{ssec:review} above for a review of this fact). We now show that this is not the case for many Lagrangian tori in 4-manifolds. \par

To state our result, we use the notion of ball-embeddability: a Lagrangian torus $L$ is \emph{ball-embeddable} if there are a symplectic embedding $e:B\hookrightarrow X$ of a ball $B$ and a product torus $T\subseteq B$ such that $e(T)=L$. \par

\begin{proposition} \label{prop:ham-action}
	Let $X$ be diffeomorphic to either $S^2\times S^2$ or $\CP^2\sharp k\overline{\CP ^2}$ for $k\leqslant 8$. Suppose that $L$ is Lagrangian isotopic to a ball-embeddable Lagrangian. Then, the map $\pi_1(\Ham(X))\to \pi_1(\cl,L)$ is trivial. \par
	
	Let $X=\C^2$ with its standard structure, then the result holds with $\Ham(\C^2)$ replaced with $\Symp_{cyl}(\C^2)$, the space of symplectomorphisms $\psi$ of $\C^2$ such that $\psi^*\lambda_0=\lambda_0$ outside some compact set.
\end{proposition}

From \cite{GooIvrRiz16}, every Lagrangian torus in $\C^2$, $\CP^2$ or the monotone $S^2\times S^2$ is Lagrangian isotopic to a ball-embeddable one. Here, $\lambda_0=\frac{1}{2}\sum_i q_idp_i+p_idq_i$ is the standard radial Liouville form. \par

\begin{proof}
  We start with the case where $X$ is closed. Note that, given a Hamiltonian loop $\{\phi^t\}$, a Lagrangian isotopy from $L$ to $L'$ defines a free homotopy from $\{\phi^t(L)\}$ to $\{\phi^t(L')\}$. Therefore, it suffices to prove the case where $L$ is ball-embeddable. In fact, by a Lagrangian isotopy within the ball, we may suppose that $L$ is the image of the Clifford torus under a symplectic embedding $e:B_r\hookrightarrow X$ for $r$ arbitrarily small. \par
  
  But for $r>0$ small enough and $X$ as above, Anjos, Li, Li, and Pinsonnault~\cite{AnjLiLiPin2024} showed that the space $\mathrm{Emb}_\omega(B_r,X)$ of symplectic embeddings $B_r\hookrightarrow X$ is weakly homotopy equivalent to the space $\mathrm{SF}(X)$ of symplectic frames of $X$~---~see \cite{LalPin2004,Pin08} for the original proofs for $X=S^2\times S^2$ and $X=\CP^2$, respectively. Moreover, the weak homotopy equivalence is given by the evaluation of the embedding at $0\in B_r$ and its differential at that point. But $\{\phi^t(e(0))\}$ is a trivial loop in $X$, since every loop in $X$ coming from a loop in $\Ham(X)$ is trivial (this is a direct consequence of Arnol'd conjecture). Therefore, $\{\phi^t(e)\}$ is homotopic in $\mathrm{Emb}_\omega(B_r,X)$ to a loop $\{e\circ A_t\}$ for some loop of linear symplectic transformations $\{A_t\}$. Note that this induces a homotopy in $\cl$ from $\{\phi^t(L)\}$ to $\{e(A_t\cdot \tcl)\}$. But the group of linear symplectic transformations deformation retracts onto $\mathrm{U}(2)$, and the map $\pi_1(\mathrm{U}(2))\to \pi_1(\cl,\tcl)$ is trivial: this is the content of the proof of \cref{lem:switch^2}. This thus proves the compact case. \par
  
  The case $X=\C^2$ is proven analogously, since $\Symp_{cyl}(\C^2)$ deformation retracts onto $\mathrm{U}(2)$. This is a folkloric fact, but we sketch the argument and thank Emmy Murphy for pointing it out to us. Essentially, the map $\Symp_{cyl}(\C^2)\to \mathrm{Cont}(S^3)$ sending a symplectomorphism that is cylindrical at infinity to the contactomorphism it induces on the ideal boundary $S^3$ of $\C^2$ is a Serre fibration with fibre $\Symp_{c}(\C^2)$. But $\Symp_{c}(\C^2)$ is contractible by a classical result of Gromov~\cite{Gro85}, so that the map is a homotopy equivalence. It follows from a multiparametric version of Eliashberg's classical result on the extendability of contactormorphisms of $S^3$~\cite{Eli92} that $\mathrm{Cont}(S^3)$ deformation retracts onto $\mathrm{U}(2)$~---~see \cite{FerMarPre2020} for another proof. 
\end{proof}

\begin{remark}
	In terms of the map $\mathrm{Emb}_\omega(B_r,X)\to \cl$ defined by $e\mapsto e(\tcl)$, the above proof can be summarized as saying that
	\begin{center}
	\begin{tikzcd}
		{\pi_1(\mathrm{Emb}_\omega(B_r,X))} \arrow[r] & {\pi_1·(\cl)} \\
		{\pi_1(\Ham(X))} \arrow[u] \arrow[ur]
	\end{tikzcd}
	\end{center}
	commutes for $r$ small enough and that the horizontal arrow is trivial.
\end{remark}

\subsection{Hausdorff-local path connectedness}\label{sec:hausdorff-local-path}
In order to prove \cref{thm:B_Hausdorff}, we first show this important local estimate. We denote by $\T^2$ the 2-torus, equipped with the flat Riemannian metric, and let $D_r^*\T^2$ be its cotangent disk bundle of capacity $r > 0$, equipped with the canonical symplectic form.

\begin{theorem} \label{thm:nearby-2-tori}
    Let $L\subseteq (D_r^*S^1)^2$ be a Lagrangian torus. Denote by $A^2_{\min}(L)$ the infimal area of a smooth Maslov-2 disk with boundary on $L$; we set $A^2_{\min}(L)=0$ if there are no such disks. We have that 
    \begin{align*}
    	A^2_{\min}(L)\leqslant \frac{4r}{3}.
    \end{align*}
\end{theorem}

\begin{proof}
  Denote by $\pi_{L}$ the composition $L\to T^*\T^2\to \T^2$ and by $(\pi_{L})_*$ the map it induces on $H_1(\!~\cdot\!~;\R)$. \par

  Suppose that $\dim\Ker (\pi_{L})_*=2$. Note that the Lagrangian embedding mapping $\T^2$ to the product torus $T(r,r)\subset (\R^4,\omega_0)$ extends to a symplectic embedding $\Psi$ of $(D_r^*S^1)^2$ into the ball $B^4(4r)$ of capacity $4r$. By the hypothesis on $(\pi_{L})_*$, this induces an isomorphism $\Psi_*:\pi_2((D_r^*S^1)^2,L)\to \pi_2(B^4(2r),\Psi(L))$ that preserves areas. By~\cite{CieMoh18}, $\Psi(L)$ bounds a Maslov-2 disk of area at most $\frac{4r}{3}$. Therefore, we have that $A^2_{\min}(L)\leqslant \frac{4r}{3}$. 

  Suppose now that $\dim\Ker (\pi_{L})_*=1$. If $L$ does not bound a Maslov-2 disk $u$ of positive area, then there is nothing to prove. Therefore, we may suppose that it does. Note that such an element must generate $\pi_2((D_r^*S^1)^2,L)\cong\Z$, since $u$ must be primitive by the Maslov condition. Furthermore, the symplectic embedding sending $(D_r^*S^1)^2$ to either $D_r^*S^1\times B^2(2r)$ or $B^2(2r)\times D_r^*S^1$ (which extend the maps $\T^2\to S^1\times T(r)$ or $\T^2\to T(r)\times S^1$, respectively) induces an isomorphism on relative $\pi_2$~---~this is the argument in the proof of Conjecture~C in~\cite{ChaLec24}. Therefore, by the Chekanov bound on displacement energy~\cite{Che98}, the area of the smallest $J$-disk with boundary on $L$ must have area at most $r$. By the above discussion, this disk is some positive multiple of $u$. Therefore, we have that $A^2_{\min}(L)\leqslant r$. \par

  Suppose finally that $\dim\Ker (\pi_{L'})_*=0$. Then, $L$ does not bound any disk by the characterization lemma~\cite[Lemma 9]{ChaLec24}, i.e. $A^2_{\min}(L)=0$.
\end{proof}

Let us now prove the Hausdorff-local path connectedness in \cref{thm:B_Hausdorff}.

\begin{corollary} \label{cor:local-connected_2-tori}
	Every Lagrangian torus $L$ in $\C^2$ has a system of neighbourhoods $\{U_n\}$ with the following property. There is $n_0 \in \N$ such that, for all $n \geqslant n_0$, if $L'$ is Hamiltonian isotopic to $L$ in $\C^2$ and $L'\subseteq U_n$, then $L'$ is Hamiltonian isotopic to $L$ in $U_n$. In particular, $\LHam(L)$ is locally path connected in the Hausdorff topology.
\end{corollary}

\begin{proof}
  Let $\Psi \colon D_{r_0}^*\T^2 \hookrightarrow \C^2$ be a Weinstein neighbourhood of $L$. We set $U_n = \Psi(D_{1/n}^*\T^2)$. Although this sequence may not be defined for a finite set of indices, it makes sense for large enough $n \in \N$. By \cite{CieMoh18}, $L$ bounds a Maslov-$2$ disk in $\C^2$ of positive area. Since Maslov-0 disks form a subspace of $\pi_2(\C^2,L)$ of rank 1, this means that the minimal positive area $A^2_{\min}(\C^2,L)$ of Maslov-2 disks in $\C^2$ with boundary on $L$ is positive. Now let $n_0 \in \N$ be large enough so that $\frac{2}{n_0} < A^2_{\min}(\C^2,L)$. Take $n \geqslant n_0$, and suppose that $L' \subset U_n$ is Hamiltonian isotopic to $L$ in $\C^2$. We need to prove that $L'$ is Hamiltonian isotopic to $L$ in $U_n$. Since $L' \subset U_n$, we have $\Psi^{-1}(L') \subset D_{1/n}^* \T^2$, and, by \cref{thm:nearby-2-tori}, we find that
  \begin{equation}
    \label{eq:A_min_estimates}
    A^2_{\min}(D_{1/n}^* \T^2, \Psi^{-1}(L'))
    \leqslant \frac{2}{n}
    \leqslant \frac{2}{n_0}
    < A^2_{\min}(\C^2, L).
  \end{equation}
  Therefore, $\Psi^{-1}(L')$ does not bound any Maslov-$2$ disks of positive area in $D_{1/n}^* \T^2$. Otherwise, $L'$ would bound a Malsov-$2$ disk of positive area at most $ \frac{2}{n_0}$ in $\C^2$, which \eqref{eq:A_min_estimates} does not allow for. Indeed, since $L'$ is Hamiltonian isotopic to $L$, we have $A^2_{\min}(\C^2, L') = A^2_{\min}(\C^2, L)$. But since there are no Maslov-$2$ disks of positive area, the third point of Theorem~4.1.1 of~\cite{AudLalPol94} is violated. By the first point of this theorem, we find that $\Psi^{-1}(L')$ is homologous to the zero section in $D_{1/n}^* \T^2$. \par

  This allows us to apply \cite[Theorem B (1)]{Riz19} to find that $\Psi^{-1}(L')$ is Hamiltonian isotopic inside $D_{1/n}^* \T^2$ to the graph of a closed 1-form. Since cohomologous 1-forms have Hamiltonian isotopic graphs, $\Psi^{-1}(L')$ is Hamiltonian isotopic to some $\T^2\times\{z\}$ in $D_{1/n}^* \T^2=\T^2\times B^2(1/n)$. In particular, $L'$ is Hamiltonian isotopic in $U_n$ to a Lagrangian whose $C^1$ distance (in the appropriate metric) to $L$ is at most $\frac{1}{n}$. The result then follows from discreteness of the flux (\cref{thm:B}) and the usual result on $C^1$-local path connectedness (\cref{prop:Flux-top}).
\end{proof}

We can now finish the proof of Theorem~\ref{thm:B_Hausdorff}, i.e.\ prove Hausdorff-closedness of $\LHam(L)$ in $\cl(L)$.

\begin{proof}[Proof of~\cref{thm:B_Hausdorff}]
  Let $\{L_i\}$ be a sequence in $\LHam(L)$ Hausdorff-converging to some Lagrangian torus $K$. For $i$ large, $L_i$ is in a Weinstein neighbourhood of $K$ of size $r_i>0$, and $r_i\to 0$. Therefore, for $i$ large, we have that $A^2_{\min}(L_i,D^*_{r_i}\T^2)\leqslant 2r_i< A^2_{\min}(L_i,\C^2)=A^2_{\min}(L,\C^2)$, which means that $L_i$ does not bound any Maslov-2 disk in the Weinstein neighbourhood of $K$. Just as in the proof of \cref{cor:local-connected_2-tori}, we can thus assume that $L_i\to K$ in the $C^1$ topology. The result then follows from \cref{thm:B} and \cref{prop:Flux-top}.
\end{proof}

\begin{remark}
	Given $\tau>0$, a similar argument shows that the subspace of $\cl$ formed by Lagrangians $L$ with $A^2_{\min}(L)\geqslant \tau$ is Hausdorff-locally path connected and closed in $\cl$. \\
	However, we cannot expect the whole space $\cl$ to have those properties. Indeed, take any submanifold $N$ and a sequence of smooth curves $\{c_k:S^1\to N\}$ $C^0$-converging to a space-filling curve of $N$. By~\cite[Proposition~A1.4]{LeOno1995}, for each $k$, there is a Hamiltonian isotopy $\phi_k$ of $\C^2$ sending $S^1(|\lambda_0(c_k)|)\times\{0\}$ to $c_k(S^1)$. In particular, if $\{\epsilon_k\}\subseteq\R_{>0}$ converges to zero fast enough, $\phi_k$ sends $S^1(|\lambda_0(c_k)|)\times S^1(\epsilon_k)$ to a Lagrangian torus $T_k$ such that $T_k\to N$ in the Hausdorff metric. However, $T_k$ is not topologically isotopic to $N$ in any small neighbourhood: the loop corresponding to the $S^1(\epsilon_k)$ factor is contractible in such a neighbourhood.
\end{remark}

\section{Monodromy and flux in $\CP^n$}
\commentintoc{We compute the monodromy group of the Clifford torus in complex projective spaces. The case $n=2$ is done in \S\,\ref{sec:case-cp2}, where we prove \cref{thm:Cliff_CP2}, while the case $n\geqslant 3$ is proven in \S\,\ref{ssec:mon_CPn}. \cref{thm:realizability_intro} then follows directly from these results. The shape invariant is computed in \S\,\ref{ssec:fluxCPn}, when $n\geqslant 3$, and in \S\,\ref{ssec:fluxcp2}, when $n=2$. This latter section is also where we prove \cref{thm:possible-fluxes_CP2_intro}.}

We now turn to complex projective spaces, where we are able to determine the monodromy group of the Clifford torus explicitly as well as (most of) its shape invariant. The computation of the monodromy group requires different techniques when $n=2$, which we explain in \S\,\ref{sec:case-cp2}, where we prove \cref{thm:Cliff_CP2}. The case $n\geqslant 3$ is proven in \S\,\ref{ssec:mon_CPn}. \cref{thm:realizability_intro} then follows directly from these results. The  shape invariant is computed in \S\,\ref{ssec:fluxCPn}, when $n\geqslant 3$, and in \S\,\ref{ssec:fluxcp2}, when $n=2$. This latter section is also where we prove \cref{thm:possible-fluxes_CP2_intro}.

\subsection{Lagrangian monodromy in $\CP^2$}
\label{sec:case-cp2}
In this section, we prove Theorem \ref{thm:Cliff_CP2} from the introduction.

\medskip
First, we claim that we can restrict our attention to the Clifford torus $\tcl$. Indeed, because any Lagrangian torus in $\CP^2$ is Lagrangian isotopic to $\tcl$, by the change-of-basepoint formula (\ref{eq:fm_base_point_change}), any two Lagrangian tori have isomorphic homological monodromy.

\medskip
It is easy to see that the matrices 
\begin{equation}
	\label{eq:monodromy_gen_torus_in_CP2}
		A = \begin{pmatrix}
			0 & 1 \\
			1 & 0
		\end{pmatrix}\qquad
		M =   \begin{pmatrix}
			2 & 1 \\
			-1 & 0
		\end{pmatrix}\qquad
		R =   \begin{pmatrix}
			0 & -1 \\
			1 & -1
		\end{pmatrix}
\end{equation}
belong to $\m M(\tcl)$.
Indeed, $A$ and $M$ are induced by the generators $\cm(\Sigma)$ and $\cm(\Xi)$ of $\Z_2\ltimes\Z$ from Sections \ref{sec:almost-toric-fibr} and \ref{sec:flux-monodr-morph} in $\C^2$, while $R$ comes from a rotation of the moment polytope of $\CP^2$.
Note that $A$ and $R$ generate all coordinate permutations in homogeneous coordinates.

\begin{remark}
  This shows that \cref{cor:charac_monotone-tori_C2}, which characterizes monotone Lagrangian tori in $\C^{2}$ in terms of their flux-monodromy group, fails in other spaces. The Hamiltonian isotopy of $\CP^{2}$ which swaps homogeneous coordinates induces an element $\Lambda \in \pi_1(\cl,L)$ whose flux-monodromy is $(0,A)$, for any non-monotone toric fibre of the form $L = T(a,a) \subset \CP^2$.
\end{remark}

We denote the mod-3 reduction by $\rho : \GL(2;\Z) \to \GL(2;\Z/3\Z)$, and write $\1$ for the vector $(1,1)$. Here is the algebraic statement that is central to the proof of Theorem \ref{thm:Cliff_CP2} for $\CP^{2}$.

\begin{proposition} \label{prop:monodromy_CP2}
	The monodromy group $\m M(\tcl)$ is isomorphic to a subgroup of $\GL(2;\Z)$ that can be described equivalently as
	\begin{enumerate}
		\item the subgroup $G_0 = \langle A,R,M \rangle$ generated by $A$, $R$, and $M$;
		\item the congruence subgroup $G_1 = \rho^{-1}(F)$ where $F=\langle \rho(A),\rho(M) \rangle$ ;
		\item the congruence subgroup $G_2 = \rho^{-1}(\mathrm{Stab}(\1)) = \{ P \in \GL(2;\Z) \,|\, \1 P = \1 \;\mathrm{mod}\; 3 \}$. 
	\end{enumerate}
\end{proposition}

We first show that Theorem \ref{thm:Cliff_CP2} immediately follows from the proposition.

\begin{proof}[Proof of Theorem~\ref{thm:Cliff_CP2}]
  Since $\CP^{2}$ is $H$-monotone, \cref{prop:few-cases_discreteness} ensures that the flux set $\fl(\tcl)$ vanishes, so that the monodromy and flux-monodromy groups of $\tcl$ are isomorphic.

    The first claim of~\cref{thm:Cliff_CP2} is simply \textit{(3)} from the proposition above, with the notation $H'$ from the introduction changed to $G_{2}$ for readability.

  We now compute its index by using description \textit{(2)}.
    The matrix $A$ has order 2, $M$ has order 3, and $AMA=M^{-1}$. Hence, the group generated by $A$ and $M$ is isomorphic to the dihedral group of order 6. Since $\rho(A)$ and $\rho(M)$ are non-trivial in $\GL(2;\Z/3\Z)$, the same holds for $F$.
    Since $G_1$ is the congruence subgroup $\rho^{-1}(F)$, its index in $\GL(2;\Z)$ is
  $$[\GL(2;\Z):G_1] = [\GL(2;\Z/3\Z):F].$$
  Notice that there are $48$ ordered pairs of non-proportional vectors in $(\Z/3\Z)^{2}$. Any matrix formed of such a pair has determinant $\pm 1$ mod $3$, and thus is invertible  in $\mathrm M_{2}(\Z/3\Z)$. This shows that the order of $\GL(2;\Z/3\Z)$ is 48 and thus that $[\GL(2;\Z):G_1] = 8$.

  The claim concerning the generators obviously follows from \textit{(1)}.
\end{proof}

The following two facts have been established in the proof above but are independent of Proposition \ref{prop:monodromy_CP2}: (i) $F$ has order $6$, and (ii) $G_{1}$ has index $8$ in $\GL(2;\Z)$.
They will be used in the proof of said proposition, which we now explain.

\begin{proof}[Proof of~\cref{prop:monodromy_CP2}]
  First, notice that the columns of any matrix in $\GL(2;\Z/3\Z)$ that stabilizes $\1$ are of the form
  $\big(\begin{smallmatrix} 0\\1 \end{smallmatrix}\big)$,
  $\big(\begin{smallmatrix} 1\\0 \end{smallmatrix}\big)$, or
  $\big(\begin{smallmatrix} 2\\2 \end{smallmatrix}\big)$.
  Since any choice of two different columns among these three yields an invertible matrix in $\GL(2;\Z/3\Z)$, $\mathrm{Stab}(\1)$ is a subgroup of $\GL(2;\Z/3\Z)$ of order 6.
  Because $F$ is of order 6 and $F \subset \mathrm{Stab}(\1)$, we deduce that these groups coincide, so that $G_1=G_2$ and consequently $[\GL(2;\Z):G_{2}]=8$.

  Second, since $A$, $R$, and $M$ belong to $\m M(L)$, and since any Lagrangian isotopy preserves the Maslov class, we have the inclusions $G_0 \subset \m M(L) \subset G_2 = G_1$. Thus, we only need to show that $G_0$ and $G_2$ have the same index in $\GL(2;\Z)$, \textit{i.e.}\ that $[\GL(2;\Z):G_0]=8$. We compute this index in two steps.

  \smallskip	
  \textit{Step 1.} Since $G_0 \subset G_2$, the matrix $-I_2$ is not in $G_0$. Let $H_0 = \langle -I_2,R,M \rangle \subset \SL(2;\Z)$ and $K=\langle R,M \rangle$. Since $A$ has order 2 and $AKA^{-1}\subset K$, $[G_{0}:K]=2$. For the same reason, $[H_{0}:K]=2$, and thus
  \begin{equation}
    \label{eq:compute_index_1}
    [\GL(2;\Z):G_0]=[\GL(2;\Z):H_0]=2 \cdot [\SL(2;\Z):H_0].
  \end{equation}

  \textit{Step 2.} The following composition of surjective morphisms preserves the index of $H_0$
  \begin{equation*}
    \begin{tikzcd}
      \SL(2;\Z) \ar[r, twoheadrightarrow, "\pi"] & \PSL(2;\Z) \ar[r, twoheadrightarrow, "\rho"] & \PSL(2;\Z/3\Z).
    \end{tikzcd}
  \end{equation*}
  Indeed, $H_0$ obviously contains the kernel $\{\pm I_2\}$ of $\pi$, so that
  \begin{equation}
    \label{eq:compute_index_2}
    [\SL(2;\Z):H_0] = [\PSL(2;\Z):\pi(H_0)]
  \end{equation}
  with $\pi(H_0) = \pi(K) = \langle \pi(R), \pi(M) \rangle$.

  On the one hand, note that $MR = \big(\begin{smallmatrix} 1 & -3\\0 & 1 \end{smallmatrix}\big)$, so that $\rho(\pi(M)) = \rho(\pi(R))^{-1}$, and that $\rho(\pi(R))$ has order 3. Hence, $\rho \circ \pi (K)$ is the cyclic group of order 3, generated by the image of $R$.
  On the other hand, the kernel of $\rho$ is the principal congruence subgroup $\Gamma(3)$ of $\PSL(2;\Z)$. We make the following claim, whose proof we postpone to \cref{remark_generators_Gamma3} below.
  \begin{claim}
    \label{claim:monodromy_gen_torus_in_CP2}
    $\Gamma(3)$ is generated by the following elements of $\PSL(2;\Z)$:
    \begin{equation}
      \label{eq:set_generators_Gamma3}
      X = \begin{pmatrix}
            1 & 3 \\
            0 & 1
          \end{pmatrix},\quad
          Y =   \begin{pmatrix}
                  -8 & 3 \\
                  -3 & 1
                \end{pmatrix}, \quad \mbox{and }
                Z =   \begin{pmatrix}
                        4 & -3 \\
                        3 & -2
                      \end{pmatrix}.
                    \end{equation}
                  \end{claim}

                  It is easy to check that
                  \begin{equation*}
                    X = (MR)^{-1}, \quad Y = XRXR^{-1}= (RM^2R)^{-1}, \quad \mbox{and } Z = R^2X^{-1}R = R^2MR^2.
                  \end{equation*}
                  Hence, $\ker \rho$ is a subgroup of $\pi(K)$, and
                  \begin{equation*}
                    [\PSL(2;\Z):\pi(H_0)] = [\PSL(2;\Z/3\Z):\rho\circ\pi(K)] = \frac{12}{3} = 4,
                  \end{equation*}
                  which concludes the proof together with \eqref{eq:compute_index_1} and \eqref{eq:compute_index_2}.
                \end{proof}

\begin{remark}
  \label{remark_generators_Gamma3}
  In \cite{kulkarni}, Kulkarni establishes two procedures, which respectively
  \begin{enumerate}
    \item associate a Farey symbol to any principal congruence subgroup $\Gamma(N)$ \cite[\S 12-13]{kulkarni},
    \item determine an independent set of generators for any subgroup given by a Farey symbol \cite[Theorem 6.1]{kulkarni}.
    \end{enumerate}
    Both procedures are implemented in Sage by the following two commands:
    \begin{lstlisting}[numbers=none,numberstyle=\tiny,numbersep=0pt]
    sage: f = FareySymbol(Gamma(3)); f
    FareySymbol(Congruence Subgroup Gamma(3))
    sage: f.generators()
    [
    [1 3]  [-8  3]  [4 -3]
    [0 1], [-3  1], [3 -2]
    ]
  \end{lstlisting}
  These define $\tt f$ as the principal congruence subgroup $\Gamma(3)$ and give the set of generators \eqref{eq:set_generators_Gamma3}.
\end{remark}

        \begin{remark}
          The main steps of the proof of \cref{prop:monodromy_CP2} above were suggested by Yves Benoist, except for the last one. Indeed, in order to compute the index of $H_0$ in $\SL(2;\Z)$ (\textit{Step 2} above), he suggested using the action of $\PSL(2;\Z)$ on the Poincaré (hyperbolic) half-plane $\mathbb H = \{z\in \C \,|\, \im(z) > 0\}$ by Möbius transformations:
          \begin{equation}
            \label{eq:3}
            A = \begin{pmatrix}
                  a & b\\ c & d
		\end{pmatrix} \mapsto \left[ h_A : z \mapsto \frac{az+b}{cz+d} \right].
              \end{equation}
              Because $\pi(H_0)$ acts by isometries, the index $[\SL(2;\Z):H_0]$ is given by the ratio of the hyperbolic area of the respective fundamental domains. It is well-known that a fundamental domain of the action of $\PSL(2;\Z)$ is the hyperbolic triangle of vertices $x_1 = e^{i \pi/3}$, $\infty$, $x_2 = e^{2i \pi/3}$, whose area is $\tfrac{\pi}{3}$.

              Because the Möbius transformation $h_R$ associated to $R$ maps the geodesic joining $x_1$ to $\infty$ to that joining $0$ to $x_1$, and $h_M$ maps the geodesic joining $-1$ to $0$ to that joining $-1$ to $\infty$, the geodesic polygon of vertices $\infty$, $-1$, $0$, and $x_1$ is a fundamental domain of the action of $\pi(H_0)$. The Gauss--Bonnet formula then gives that its area is $\tfrac{4\pi}{3}$, so that $[\SL(2;\Z):H_0]=4$.
            \end{remark}

\subsection{Monodromy in $\CP^n$}
\label{ssec:mon_CPn}
Recall from \S\,\ref{ssec:mtofm} the definition of the \emph{homological Lagrangian monodromy} group
\begin{align*}
	{\cm}^\Z(L) 
	&=
	\{ (\phi^\Lambda)_{*} \in \Aut(H_1(L;\Z)) \sth \Lambda \in \pi_1(\cl,L) \}
\intertext{and the \emph{extended homological Lagrangian monodromy} group}
	\widehat{\cm}^\Z(L) 
	&=
	\{ \Phi^\Lambda \in \Aut(H_2(X,L;\Z)) \sth \Lambda \in \pi_1(\cl,L) \}.
·\end{align*}
By $\fmh^{\Z}(L)$ we denote the corresponding subgroup generated by Hamiltonian isotopies, which we call the \emph{extended (homological) Hamiltonian monodromy group}.

\begin{remark}
   \label{rem:MCG-vs-Aut}
 Since $\del:H_2(X,L;\Z)\to H_1(L;\Z)$ is surjective in this context, the extended homological monodromy group $\fml^{\Z}(L)$ fully determines the homological monodromy group $\cm^{\Z}(L)$. 
 How much this determines the monodromy group $\cm(L) \subset \MCG(L)$ depends on $n$. Indeed, it is known that $\MCG(\T^n)$ is isomorphic to $\Aut H_1(\T^n;\Z)$ when $n=3$~\cite{Wal68}, but not when $n\geqslant 5$~\cite{Hat78}~---~the case $n=4$ remains open. Therefore, when $n\geqslant 4$, our computations may not fully determine $\cm(L)$.
 \end{remark}

 Since we will only consider \emph{homological} monodromies below, we drop the superscript $\Z$ from the notation throughout this subsection. \qedhere

Any element $\Psi \in \fml$ preserves the Maslov class and acts by the identity on ambient elements, i.e.\ on the image $J = \im (j \colon H_2(X) \rightarrow H_2(X,L))$. \cref{q:realizability} asks whether $\fml$ contains precisely \emph{all} such elements. The goal of this section is to answer the question positively for tori that are Lagrangian isotopic to the Clifford torus in $\CP^n$. 

Moreover, we determine the corresponding homological Lagrangian monodromy group. Let $\1 = (1,\ldots,1)$, and define the following subgroups of $\GL(n;\Z)$:
\begin{align}
	\label{eq:Hsubgroup}
	H &= \{A \in \GL(n;\Z) \sth \1 A = \1 \}, \\
	\label{eq:Hprimesubgroup}
	H'&= \{A \in \GL(n;\Z) \sth \1 A \equiv \1  \mod n+1\}.
\end{align}
We obviously have $H \subset H'$. The group $H$ is isomorphic to $\Z^{n-1} \rtimes \GL(n-1;\Z)$. The group $H'$ is the preimage of the mod-$(n+1)$ reduction of $H$ and has finite index in $\GL(n;\Z)$.

\begin{theorem}
	\label{thm:realizability}
	Let $n \geqslant 3$. The answer to \cref{q:realizability} is positive for
	\begin{enumerate}
		\item any Lagrangian torus $L \subset \C^n$ that is Lagrangian isotopic to a product torus $T_0$;
		\item any Lagrangian torus $L \subset \CP^n$ that is Lagrangian isotopic to a toric fibre. 
	\end{enumerate}
	More precisely, in the natural bases induced by $T_0 \subset \C^n$, and $T_{\rm Cl} \subset \CP^n$, the corresponding (non-extended) homological Lagrangian monodromy groups are given by
	\begin{equation*}
		\cm(T_0) = H, \quad
		\cm(T_{\rm Cl}) = H'.
	\end{equation*}
\end{theorem}

For the case of $T_0 \subset \C^n$, this result goes back to Chekanov~\cite[Proposition 5.1 (a)]{Che96}. We reprove it for the reader's convenience and because our proof of \emph{(2)} will be based on it. In both cases, the strategy of the proof is as follows. Let $L_0,L_1 \subset X$ be Lagrangian isotopic Lagrangians. Then,
\begin{equation}
	\label{eq:subg_L0_L1}
	\fmh(L_1) \subset \fml(L_1) \quad\text{and}\quad
	\Phi^{\Pi} \fmh(L_0) (\Phi^{\Pi})^{-1} \subset \fml(L_1),
\end{equation}
where $\Pi = \{L_t\}_{t \in [0,1]}$ is a Lagrangian isotopy from $L_0$ to $L_1$ and $\Phi^{\Pi} \colon H_2(X,L_0) \rightarrow H_2(X,L_1)$ denotes the isomorphism it induces. In both cases appearing in \cref{thm:realizability}, there is a convenient choice of $L_0$ and $L_1$ such that the union of the two subgroups \eqref{eq:subg_L0_L1} generates all elements preserving the Maslov class and acting by the identity on $J$.  \smallskip

Before proving \cref{thm:realizability}, we first need some algebraic preparations. These are surely known to experts on discrete groups, but we were unable to find a suitable reference. 

\begin{lemma}
	\label{lem:gen}
	Let $n \geqslant 3$. The group $H'$ is generated by the set of all permutation matrices, together with the matrices
	\begin{equation*}
		S_1 = 
		\begin{pmatrix}
			1 & 0 & 0 & \dots & 0 \\
			1 & 1 & 0 & \dots & 0 \\
			-1 & 0 & 1 & \dots & 0 \\
			\vdots & \vdots & \vdots & \ddots & \vdots \\
			0 & 0 & 0 & \dots & 1 
		\end{pmatrix}, \quad
		T = 
		\begin{pmatrix}
			1 & 0 & 0 & \dots & 0 \\
			n+1 & 1 & 0 & \dots & 0 \\
			0 & 0 & 1 & \dots & 0 \\
			\vdots & \vdots & \vdots & \ddots & \vdots \\
			0 & 0 & 0 & \dots & 1 
		\end{pmatrix}.
	\end{equation*}
	Similarly, the group $H$ is generated by simply adding $S_{1}$ to the set of all permutations.
\end{lemma}

The proof works by a modification of the Gau{\ss} algorithm. For the full group $\GL(n;\Z)$, this algorithm shows that a set of generators is given by all permutations and an elementary shear, i.e.\ a matrix differing from the identity by adding a $1$ somewhere off the diagonal. It works by tranforming a given matrix in the group into an upper triangular one by taking repeated linear combinations of rows and permuting rows and columns. This corresponds to multiplying by permutation matrices and the elementary shear and thus proves that these generate all of $\GL(n;\Z)$.  \smallskip

\begin{proof}[Proof of Lemma \ref{lem:gen}]
  Let $K \subset H$ be the subgroup generated by permutations and $S_1$. Similarly, let $K' \subset H'$ be the subgroup generated by permutations, $S_1$, and $T$. Our goal is to show that $K = H$, and $K' = H'$ by proving that any element in $H$ and $H'$, respectively, can be reduced to the identity $I \in \GL(n;\Z)$ through multiplication by elements in $K$ and $K'$, respectively. Let us denote by $E_{i,j}$ the elementary matrix whose only non-zero entry is a 1 at the intersection of the $i$-th row and the $j$-th column. Now let $l,i,j \in \{1,\ldots,n\}$ be pairwise distinct indices. Throughout the proof, we will denote 
  \begin{equation*}
    S_{l,i,j} = I + E_{il} - E_{jl}.
  \end{equation*}
  These matrices are obviously in $K,K'$ and satisfy $S_{l,i,j}^k = I + kE_{il} - kE_{jl}$. For any $(n \times n)$-matrix, we use the following notation 
  \begin{equation*}
    B = 
    \begin{pmatrix}
      R_1 \\
      \vdots \\
      R_n
    \end{pmatrix}
    = (C_1 \dots C_n)
  \end{equation*}
  for its rows and columns. Left-multiplication by $S^k_{l,i,j}$ yields the new rows $R_i'$ given by
  \begin{equation}
    \label{eq:rowtrafo}
    R_i'
    = 
    R_i + k R_l,\quad
    R_j'
    = 
    R_j - k R_l, \quad
    \text{and} \quad
    R_h' = R_h \text{ for all } h \notin \{i,j\}.
  \end{equation}
  Now let $B\in H'\supset H$·. We use the transformation on rows to show the following.
  
\medskip
  {\bf Step 1.} We can bring $B$ to the form 
  \begin{equation*}
    B' = 
    \begin{pmatrix}
      1 & c \\
      0 & D
    \end{pmatrix},
  \end{equation*}
  for $c \in \Z^{n-1}$ and $D \in \GL(n-1;\Z)$, by multiplying $B$ with elements in
  \begin{enumerate}
  \item $K$ if $B \in H$;
  \item $K'$ if $B \in H'$. 
  \end{enumerate}
  To that end, denote (the transpose of) the first column of $B = B^{(0)}$ by $(a^{(0)}_1,\ldots,a^{(0)}_n) \in \Z^n$. Up to applying transformations of the form \eqref{eq:rowtrafo} and permutations, we can assume that $a^{(0)}_1 \in \Z_{> 0}$. The idea is to subsequently multiply $B$ on the left by elements of the form $S^k_{l,i,j}$ such that the first entry in the first column is positive and strictly decreasing with each step of the algorithm. This means that it will be equal to $1$ after a finite number of steps.

  Here is the inductive algorithm. Let $B^{(m)}$ be a matrix with (the transpose of) the first column equal to $(a^{(m)}_1, \ldots, a^{(m)}_n) \in \Z^n$ with $a^{(m)}_1 \in \Z_{>0}$. If $a^{(m)}_1 = 1$, then terminate the algorithm. If not, then note that multiplying $B^{(m)}$ on the left by $S^{k}_{1,i,j}$ followed by a permutation of the $j$-th and first columns yield a matrix $B^{(m+1)}$ with first row $a^{(m+1)}$ given by
  \begin{align*}
    a_1^{(m+1)} &= a_j^{(m)} - k a_1^{(m)} \\
    a_i^{(m+1)} &= a_i^{(m)} + k a_1^{(m)} \\
    a_j^{(m+1)} &= a_1^{(m)} \\
    a_h^{(m+1)} &= a_h^{(m)} \quad \text{for all } h \notin \{1,i,j\}.
  \end{align*}
  We claim that there exist $k = k_m$ and $j = j_m$ in $\Z$ for which
  \begin{equation}
    \label{eq:a_1_range}
    a_1^{(m+1)} \in \{1,\ldots,a_1^{(m)} - 1\}.
  \end{equation}
  To that end, first note that, for any $j \in \{2,\ldots,n\}$, Euclid's division lemma provides an integer $k \in \Z$ such that $a_j^{(m)} - k a_1^{(m)} \in \{0,\ldots,a_1^{(m)} - 1\}$. This means that we need to show that there is one $j \in \{2,\ldots,n\}$ such that this number is different from $0$. Assume, by contradiction, that this is not the case. Then, for any $j \in \{2,\ldots,n\}$ there exists $k_j$ such that $a^{(m)}_j = k_j a_1^{(m)}$. But then we can write 
  \begin{equation*}
    (a_1^{(m)}, \ldots, a_n^{(m)}) = a_1^{(m)} (1, k_2, \ldots, k_n),
  \end{equation*}
  which means that the determinant $\det B^{(m)} = \pm\det B$ is divisible by $a_1^{(m)}$. Since the matrices are in $\GL(n;\Z)$ and $a_1^{(m)} >0 $ by hypothesis, we have $a_1^{(m)}=1$, contradicting the induction hypothesis. We have proven the existence of $k,j$ for which \eqref{eq:a_1_range} holds. Take any $i \notin \{k,j\}$ and multiply on the left by $S^k_{1,i,j}$. This proves that the sequence $(a^{(m)}_1)_{m \in \N}$ is positive and strictly decreasing, allowing us to conclude that there exists $m_* \in \N$ such that $a^{(m_*)}_1 = 1$ and the algorithm terminates. 

  By multiplying $B^{(m_*)}$ on the left by $S^{a^{(m_*)}_n}_{1,2,n}$, we can bring its first column to the form $(1,a_2',\ldots,a_{n-1}',0)$ for some $a_2',\ldots,a_{n-1}' \in \Z$. Further left-multiplication by matrices of the form $S^{a}_{1,2,j}$, for $j$ from 3 to $n-1$ and appropriate integers $a$, brings it to the form $(1,a_2'', 0,\ldots,0)$. 

  Note that up until now, we have only multiplied $B$ by elements in the subgroup $K \subset H$. In the case we have started with $B \in H$, then the matrix we end up with belongs to $H$, too. In that case, we find that the first column $(1,a''_2,0,\ldots,0)$ actually satisfies $a_2''=0$, since we have $1 + a_2'' = 1$ by definition of $H$. This finishes step 1 for the case $B \in H$. In the case $B \in H'$, we find that $a_2'' = (n+1) k'$ for some $k' \in \Z$. Applying $T^{-k'}$ then proves the claim of step 1 for the case $B \in H'$.

  \medskip
  {\bf Step 2.} We consider the following subgroup:
  \begin{equation}
    \label{eq:subgroupG}
    G 
    =
    \left\{ \left.
        \begin{pmatrix}
          1 & c \\
          0 & D
        \end{pmatrix}
      \right\vert 
      \begin{aligned}
        & c = (c_2,\ldots,c_n) \in \Z^{n-1}, 
          D \in \GL(n-1; \Z) \\
        & \forall j \in \{2,\ldots,n\},\;c_j + \textstyle\sum_i D_{ij} = 1
      \end{aligned}
    \right\}
    \subset H.
  \end{equation}
  Let $G' \subset H'$ be the corresponding subgroup, where the latter condition on $G$ is replaced by $c_i + \textstyle\sum_j D_{ji} = 1 \mod n+1$. In the first step, we have shown that, up to applying elements in $K$ or $K'$, we can assume that $B \in G$ or $B \in G'$, respectively. Therefore, to prove the lemma, it suffices to show that any element in $G$ is contained in $H$ and that any element in $G'$ is contained in $H'$. We start with the case of $G$. Note that $G \cong \GL(n-1;\Z)$ via the natural isomorphism 
  \begin{equation*}
    \chi \colon \GL(n-1;\Z) \rightarrow G, \quad
    \chi(D) =
    \begin{pmatrix}
      1 & c \\
      0 & D
    \end{pmatrix},\quad
    c_i = 1 - \textstyle \sum_j D_{ji}.
  \end{equation*}
  Now let $\chi(D_0)$ be the matrix obtained from $B \in H$ after the first step. Then, we know that $D_0$ can be written as a product $D_0 = F_1 \cdot \ldots \cdot F_N$ of permutations $\Pi_{i,j}$ and one elementary shear $\Sigma$~---~this is just a consequence of the ordinary Gau{\ss} algorithm proving that permutations and one shear generate all of $\GL(n-1;\Z)$. Now note that both $\chi(\Pi_{i,j})$ and $\chi(\Sigma)$ are in $K$; indeed, the image of a permutation in $\GL(n-1;\Z)$ is a permutation in $\GL(n;\Z)$, and the image of the shear $\Sigma$ is, up to further coordinate permutations, equal to $S_1$. Therefore, the decomposition 
  \begin{equation*}
    \chi(D_0) = \chi(F_1) \cdot \ldots \cdot \chi(F_N), \quad
    F_l \in \{\Pi_{1,2}, \ldots, \Pi_{n-1,n},\Sigma\}
  \end{equation*}
  proves that $\chi(D_0) \in K$. 

  The case of $G'$ is similar, with one small difference: Since the equation $c_i + \textstyle\sum_j D_{ji} = 1$ only holds modulo $n+1$ in this case, we cannot assume that the result $B'$ of the first step lies in the image of the map $\chi$. However, we can still decompose its lower right part $D_0 \in \GL(n-1;\Z)$ as $D_0 = F_1 \cdot \ldots \cdot F_N$ as above. By a simple computation, we find that $\chi(F_N)^{-1} \cdot \ldots \cdot \chi(F_1)^{-1} \cdot B'$ is, in the notation of \eqref{eq:subgroupG}, of the form $D = I_{n-1}$ (the unit matrix in $\GL(n-1;\Z)$) and $c_i \equiv 0 \mod n+1$. This latter matrix can be written as a product of permutations and powers of $T$.
\end{proof}

\begin{proof}[Proof of Theorem~\ref{thm:realizability}]
  We start by proving \emph{(1)}. Let $L_0 = T(1,\ldots,1) \subset \C^n$ be the monotone product torus. Since $\C^n$ is contractible, we have that $H_2(\C^n,L) \cong H_1(L)$. Let $e_1,\ldots,e_n \in H_1(L_0)$ be the standard basis generated by elements of the form $e_i = [\{*\} \times \ldots \times S^1 \times \ldots \times \{*\} ]$ on product tori. All coordinate permutations on the $e_i$ are contained in $\fmh(L_0)$, since coordinate permutations on $\C^n$ can be realized by Hamiltonian diffeomorphisms and map $L_0$ to itself \footnote{Although we will not need this, $\fmh(L_0)$ is actually \emph{equal} to the set of these permutations. This follows from \cite[Theorem 4.5]{Che96} and the fact that the $e_i$ are \emph{distinguished classes} in Chekanov's language. }. 

  Here we have $H_2(\C^n) = 0$, and thus \cref{q:realizability} asks whether, in the basis $\{e_i\}$, the group of Lagrangian monodromies is equal to $H$ as defined in \eqref{eq:Hsubgroup}. By Lemma~\ref{lem:gen}, it is sufficient to show that $S_1$ is in $\fml(L_0)$. To that end, let $L_1$ be the \emph{fully twisted} Chekanov torus. In Chekanov's notation in \cite{Che96}, it corresponds to $L_1 = T^n_1(1) = \Theta_0^{n-1}(S^1(1))$, where $\Theta_0^{n-1}$ denotes the $(n-1)$-fold twisting operation. However, we will follow the exposition in \cite[Section 5]{Bre20}, where $L_1$ is constructed as the lift of a curve from a symplectic $T^{n-1}$-quotient. This yields the same torus $L_1$ as defined by Chekanov, but it allows us to understand the Lagrangian isotopy to $L_0$. As in the proof of \cite[Proposition 5.5]{Bre20}, we pick a basis $\xi,\zeta_2,\ldots,\zeta_n \in H_1(L_1)$, where $\xi$ is a lift of the class of the circle in the two-dimensional symplectic quotient and $\zeta_2,\ldots,\zeta_n$ are in the orbit of the $T^{n-1}$-action by which we reduce. Perturbing the reduced space as in the proof of \cite[Proposition 5.5]{Bre20} yields a Lagrangian isotopy $\Pi = \{L_t\}_{t \in [0,1]}$ for which
  \begin{equation}
    \label{eq:basechange}
    \Phi^{\Pi}(e_n) = \xi, \quad
    \Phi^{\Pi}(e_i) = \zeta_i + \xi \quad
    \text{for } i \in \{1,\ldots,n-1\}.
  \end{equation}
  The distinguished class is $\{\xi\}$. Any monodromy matrix preserving the Maslov class (by monotonicity, the area class is then automatically preserved) and the distinguished class is in $\fmh(L_1)$ by \cite[Theorem 4.5]{Che96}. In particular, since 
  \begin{equation*}
    \1 S_1 = \1, \;\mbox{ and }\;
    S_1 e_n = e_n,  
  \end{equation*}
  we find that $S_1 \in\Phi^{\Pi} \fmh(L_1) (\Phi^{\Pi})^{-1} \subset \fml(L_0)$, which, using Lemma~\ref{lem:gen}, proves \emph{(1)}.

  \medskip
  Let us now move to the case of $\CP^n$. Let $L = T_{\rm Cl}$ be the monotone Clifford torus. It coincides with the image of a product torus $T(\frac{1}{n+1}, \ldots, \frac{1}{n+1}) \subset \C^n$ under a symplectic ball embedding $\varphi \colon B^{2n}(1) \hookrightarrow \CP^n$. Through this embedding, we find a natural basis $\varphi_* e_i \in H_1(L)$. Capping these off by disks contained in the image $\varphi$ (where these disks are homotopically unique), we obtain elements in $H_2(\CP^n,L)$. By abuse of notation, we will call these $e_1,\ldots,e_n$, too. Together with the ambient spherical element $S=j([\CP^1]) \in H_2(\CP^n,L)$, which generates~$J = \im (j \colon H_2(\CP^n) \hookrightarrow H_2(\CP^n,L))$, we obtain a basis
  \begin{equation}
    \label{eq:basis_relh2}
    S,e_1,\ldots,e_n \in H_2(\CP^n,L).
  \end{equation}
  This can be seen by extracting a short exact sequence from the long exact sequence for relative homology and using the above capping to obtain a splitting of the relative second homology. \cref{q:realizability} asks whether the group of Lagrangian monodromies is given by 
  \begin{equation*}
    \wh G = \{ \phi \in \Aut H_2(\CP^n,L) \sth \phi^* m_{L} = m_{L}, \; \phi(S) = S \}.  
  \end{equation*}
  In the dual basis to \eqref{eq:basis_relh2}, the Maslov class is $m_{L} = (2(n+1),2,\ldots,2)$, and we find
  \begin{align*}
    \wh G &= \left\{ \left.
	\begin{pmatrix}
          1 & a \\
          0 & B
	\end{pmatrix}
	\right\vert
	(n+1)a=\1-\1 B
	\right\} \\
      &= \left\{ \left.
	\begin{pmatrix}
          1 & a \\
          0 & B
	\end{pmatrix}
	\right\vert
	\forall j\in \{1,\ldots,n\},\; (n+1)a_j+\sum_i b_{ij}=1
	\right\},
  \end{align*}
  where $a = (a_1\;\ldots\; a_n) \in \Z^n$ and $B = (b_{ij})_{i,j} \in \GL(n;\Z)$. Note that this implies that $B \in H'$. Hence, we obtain that $\wh G$ is isomorphic to $H'$ as defined in \eqref{eq:Hprimesubgroup}. As an isomorphism, we choose the obvious inclusion~$H' \hookrightarrow \wh G$. Now note that every Lagrangian loop in $\C^n$ with endpoints on $L_0$ can be included into $\CP^n$ by scaling appropriately and then composing with $\varphi$. Since $L$ is Lagrangian isotopic to the image of $L_0$ under $\varphi$, we find that 
  \begin{equation}
    \fml(L_0 \subset \C^n) \hookrightarrow \fml(L \subset \CP^n). 
  \end{equation}
  By our choice of basis, this corresponds to the natural inclusion $H \subset H' \hookrightarrow \wh G$. In other words, we can generate all of $H \subset \wh G$ just by Lagrangian isotopies in the image of the ball embedding. By Lemma~\ref{lem:gen}, it is thus enough to show that $T \in \GL(n+1;\Z)$ can be realized by a Lagrangian loop based at $L$. To that end, recall that permutations of homogeneous coordinates of $\CP^n$ can be realized by Hamiltonian (and therefore also Lagrangian) isotopies. In our choice of basis, this yields all permutations of the classes
  \begin{equation*}
    S - e_1 - \ldots - e_n, e_1, \ldots, e_n \in H_2(\CP^n,L). 
  \end{equation*}
  The matrix of the elementary permutation exchanging the first two elements is given by 
  \begin{equation*}
    \begin{pmatrix}
      1 & 1 & 0 & \dots & 0 \\
      0 & -1 & 0 & \dots & 0 \\
      0 & -1 & 1 & \dots & 0 \\
      \vdots & \vdots & \vdots & \ddots & \vdots \\
      0 & -1 & 0 & \dots & 1 
    \end{pmatrix}
    \in \wh G \subset \GL(n+1;\Z)
  \end{equation*}
  Denote the lower-right $n\times n$ submatrix of this matrix by $Q \in \GL(n;\Z)$. We have $Q \in H'$ as expected. Now, note that $\1 QT=\1$, i.e.\ $QT$ is in the subgroup $H $. Since all elements of that subgroup are realizable as monodromies (even in $\C^n$), so is $T$, which concludes the proof.
\end{proof}

\subsection{Flux in $\CP^n$}
\label{ssec:fluxCPn}

Understanding the homological monodromy group is sometimes enough to understand $\Gamma^\Z(L)$ and the flux set $\fl(L)$. Indeed, by \cref{prop:flux-monodromy_monotone}, we find that $\Gamma^\Z(T_{\rm Cl}) \cong \cm^\Z(T_{\rm Cl})$ in $\CP^n$. As for the flux set, we find the following.

\begin{corollary} \label{cor:description_fl}
	Let $L \subset \CP^n$ be the endpoint of a Lagrangian isotopy that starts at $T_{\rm Cl}$ and has flux $\eta \in H^1(T_{\rm Cl};\R)$. In the basis induced by the natural basis on $T_{\rm Cl}$, we have that
	\begin{equation*}
		\fl(L) = \{ \eta(A - id) \sth A \in H'\}.
	\end{equation*}
\end{corollary}

\begin{proof}
  Let $\Pi$ be a path with $\Pi_0 = T_{\rm Cl}$ and $\Flux \Pi = \eta$. Loops in $\pi_1(\cl,L)$ are in bijection with loops in $\pi_1(\cl,T_{\rm Cl})$ by conjugating with $\Pi$. Write $\Lambda = \Pi^{-} \circ \Lambda_0 \circ \Pi$ for $\Lambda_0 \in \pi_1(\cl,T_{\rm Cl})$. Recall that $\Flux \Lambda_0 = 0$, since $T_{\rm Cl} \subset \CP^n$ is monotone, see \cref{cor:flux_monotone}. By the change of base point formula in \cref{prop:fm_basic}, we find that
  \begin{equation*}
    \Flux(\Lambda) = (\phi^\Pi)^{-*}(\cm(\Lambda_0)^*\eta - \eta). 
  \end{equation*}
  Since $\cm(\Lambda_0) \in H'$ in the natural basis, by \cref{thm:realizability}, the claim follows.
\end{proof}

Let us now turn to the flux of Lagrangian \emph{paths}. In other words, we now try to determine \emph{the shape invariant}
\begin{equation*}
	\Sh(T_{\rm Cl}) = \{\Flux \{\Pi\} \in H^1(T_{\rm Cl};\R) \sth \Pi_0 = T_{\rm Cl} \subset \CP^n \}
\end{equation*}
for the Clifford torus in $\CP^n$. 

\begin{corollary} \label{cor:shape_irrational_n3}
	Let $n \geqslant 3$. Identifying $H^1(T_{\rm Cl};\R) = \R^n$ via the standard basis, the set 
	\begin{equation}
		\label{eq:subset_of_fluxset}
		\R^n \setminus (\R \cdot \Z^n)
		=
		\R^n \setminus \{r y \in \R^n \sth r \in \R, y \in \Z^n \}
	\end{equation}		
	is contained in $\Sh(T_{\rm Cl})$.
\end{corollary}

\begin{proof}
  We only consider Lagrangian isotopies of the form $\Pi = \Lambda * \Pi'$, where $\Lambda \in \pi_1(\cl,T_{\rm Cl})$ and $\Pi'$ is an isotopy among toric fibres. Then we have 
  \begin{equation*}
    \Flux \Pi = \cm(\Lambda)^* \Flux \Pi'.
  \end{equation*}
  Indeed, by \cref{cor:flux_monotone}, the flux of $\Lambda$ vanishes. Note that
  \begin{enumerate}
  \item for any $\eta$ in the moment polytope $\Delta_{\CP^n}$, there is an isotopy $\Pi'$ with $\Flux \Pi' = \eta$; see also \cref{prop:action_I};
  \item for any $A \in H'$, there is $\Lambda \in \pi_1(\cl,T_{\rm Cl})$ with $\cm^\Z(\Lambda) = A$ in the standard basis.
  \end{enumerate}
  Therefore, we find that 
  \begin{equation}
    \label{eq:flux_toric_isotopies}
    \{\eta A \in \R^n \sth A \in H', \eta \in \Delta_{\CP^n}\} \subset \Sh(T_{\rm Cl}).
  \end{equation}
  Let us prove that this set contains $\R^n \setminus (\R \cdot \Z^n)$. To that end, let us first point out that $x = (x_1,\ldots,x_n)$ belongs to $\R \cdot \Z^n$ if and only if, for all pairs $i\neq j \in \{1,\ldots,n\}$, there exist $k,l \in \Z$ such that $k x_i + l x_j = 0$. Let $\varepsilon > 0$. We show that we can map $x \in \R^n \setminus(\R\cdot\Z^n)$ into $B^n(\varepsilon)\subseteq \Delta_{\CP^n}$ using right-multiplication by elements in $H'$. By the above discussion, there exists a pair $i,j$ for which $k x_i + l x_j \neq 0$ for all $k,l \in \Z$. Assume, without loss of generality, that $i=1$ and $j=2$. By a standard division algorithm argument, this implies that there is a sequence $(k_r,l_r) \in \Z^2$ such that $k_r x_1 + l_r x_2 \rightarrow 0$ for $r \rightarrow \infty$. By Lemma~\ref{lem:gen}, we can map $x$ to any element of the form 
  \begin{equation*}
    (x_1,x_2,\ldots,x_{n-1},\widehat{x}_n = x_n + c(n+1)(k x_1 + l x_2))
  \end{equation*}
  for any $c,k,l \in \Z$. Note that here we use that $n \geqslant 3$. By the above, we can make the term $k x_1 + l x_2$ arbitrarily small, so that we can make $\widehat{x}_n$ arbitrarily small. Let us therefore assume that $\vert \widehat{x}_n \vert < \frac{\varepsilon}{\sqrt{n}(n+1)}$. Again by Lemma~\ref{lem:gen}, we can map $(x_1,\ldots,x_{n-1},\widehat{x}_n)$ to a vector having components $\widehat{x}_i = x_i - s_i(n+1)\widehat{x}_n$ for $i \in \{1,\ldots,n-1\}$. For a suitable choice of $s_i$, we therefore have $\vert \widehat{x}_i \vert < \frac{\varepsilon}{\sqrt{n}}$ for all $i$, and thus $\widehat{x} \in B^n(\varepsilon)$, proving the claim. 
\end{proof}

\begin{remark}
	The set in \eqref{eq:flux_toric_isotopies} is strictly larger than the one in \eqref{eq:subset_of_fluxset}. In particular, the former contains $\Delta_{\CP^n}$. Trying to determine \eqref{eq:flux_toric_isotopies} runs into the question of trying to understand the orbits of $H'$, which is tricky for $n \geqslant 3$, as we will see below. 
\end{remark}

Let us now turn to obstructions. Based on \cite{CieMoh18}, we find that some values in $H^1(T_{\rm Cl}; \R)$ cannot be realized as the flux of some Lagrangian isotopy. For simplicity, we again identify $H^1(T_{\rm Cl};\R) = \R^n$ via the standard basis of $H_1(T_{\rm Cl})$. Under this identification, the Maslov class induces a class $m=2 \1 = (2,\ldots,2) \in H^1(T_{\rm Cl};\Z) = \Z^n$, which must be preserved modulo $2(n+1)$ by Lagrangian isotopies. Let 
\begin{equation}
	\label{eq:Lambda_complement_lattice}
	\Lambda = \{ \lambda \in H^1(T_{\rm Cl};\Z) \sth \lambda \cdot \xi \neq 0 \text{ for all } \xi \in H_1(T_{\rm Cl};\Z) \text{ with } \1 \cdot \xi \equiv 1· \mod (n+1)\}.
\end{equation}
In other words, $\Lambda$ is the set-theoretic complement in $H^1(T_{\rm Cl};\Z)$ of the orthogonal complement of the affine lattice $\{\xi \in H_1(T_{\rm Cl}) \sth \1 \cdot \xi \equiv 1 \mod (n+1)\}$. 

\begin{proposition}
	\label{prop:open_flux_obstructions}
	The set 
	\begin{equation*}
		B = \{\eta = r \lambda \in H^1(T_{\rm Cl};\R) \sth \lambda \in \Lambda, r \geqslant 1\},
	\end{equation*}
	where $\Lambda$ is defined as in \eqref{eq:Lambda_complement_lattice}, is not in $\Sh(T_{\rm Cl})$. In particular, no $\eta = r \lambda$ where $r\geqslant 0$, and where $\lambda \equiv 1 \mod p$ for some $p$ dividing $n+1$ is in $\Sh(T_{\rm Cl})$.
\end{proposition}

\begin{proof}
  Let $\Pi$ be a Lagrangian isotopy in $\CP^n$ with starting point $T_{\rm Cl}$ and endpoint $L$. Suppose that $\Pi$ has symplectic flux $\eta = r\lambda$ for $r \geqslant 1$ and\footnote{Note that we have shown that all elements that are not in $\R \cdot H^1(T_{\rm Cl};\Z) \subset H^1(T_{\rm Cl};\R)$ \emph{are} actually contained in $\Sh(T_{\rm Cl})$. Therefore, supposing that $\eta = r\lambda$ with $\lambda$ in the lattice $H^1(T_{\rm Cl};\Z)$ is not that much of a restriction.} $\lambda \in H^1(T_{\rm Cl};\Z)$. We will show that $\lambda$ must satisfy
  \begin{equation}
    \label{eq:lambda_conditions}
    \lambda \cdot \xi = 0 \text{ for some } \xi \in H_1(T_{\rm Cl};\Z) \text{ with } \1 \cdot \xi = 1 \mod (n+1);
  \end{equation}	
  c.f.\ \eqref{eq:Lambda_complement_lattice}. By \cite[Theorem 1.1]{CieMoh18}, the endpoint $L$ of the Lagrangian isotopy bounds a topological disk whose class $\beta \in H_2(\CP^n,L)$  
  \begin{enumerate}
  \item has Maslov index $m(\beta) = 2$, and 
  \item has area $\langle [\omega]_L , \beta \rangle \in (0,1]$ (in our normalization of the symplectic form). 
  \end{enumerate}
  Let $\beta_0 \in H_2(\CP^n,T_{\rm Cl})$ be the class such that (in the notation of \S\ref{ssec:basic}) $\Phi^\Pi \beta_0 = \beta$, meaning that $\beta_0$ is pushed to $\beta$ by the Lagrangian isotopy $\Pi$. Since Lagrangian isotopies preserve the Maslov index, we have $m(\beta_0) = 2$. Let us now compute its area $\langle [\omega]_L, \beta \rangle$:
  \begin{align*}
    \langle [\omega]_L, \beta \rangle
    &= \langle [\omega]_{T_{\rm Cl}} , \beta_0 \rangle + \Flux\Pi \cdot \pp \beta_0\\
    &= 1 + \Flux \Pi \cdot \pp \beta_0.
  \end{align*}
  In the first line, we have used the area formula (\ref{eq:areaformula}). In the second line, we have used the fact that $[\omega]_{T_{\rm Cl}} = \frac{m}{2}$, i.e.\ $T_{\rm Cl}$ is monotone. Now setting $\xi = \pp \beta_0$, we can prove \eqref{eq:lambda_conditions}. Indeed, since $m(\beta_0) = 2$, we find that $\1 \cdot \xi = 1 \mod (n+1)$. By point (2) and the area computation above, we find $\Flux \Pi \cdot \xi \in (-1,0]$. Since we have supposed that $\Flux \Pi = r \lambda$ for an integral class $\lambda \in H^1(T_{\rm Cl}; \Z)$ and $r \geqslant 1$, we deduce that $\lambda \cdot \xi \in (-\frac{1}{r},0]$. Since $\lambda \cdot \xi \in \Z$, we deduce that it vanishes, as claimed in \eqref{eq:lambda_conditions}. This proves that $B$ is in the complement of $\Sh(T_{\rm Cl})$.

  For the second claim, let $\lambda \equiv 1 \mod p$ for some $p$ dividing $n+1$. We show that $\lambda \in \Lambda$ as defined in \eqref{eq:Lambda_complement_lattice}. By the first part of the proof, this implies the claim. Let $\xi \in H_1(T_{\rm Cl})$ with $\1 \cdot \xi \equiv 1 \mod (n+1)$. Then we find 
  \begin{equation*}
    \lambda \cdot \xi \equiv \1 \cdot \xi \equiv 1 \mod p,
  \end{equation*}
  and therefore $\lambda \cdot \xi \neq 0$. 
\end{proof}

\begin{example}
	\label{ex:flux_obstructions}
	Every vector $r \1 \in H^1(T_{\rm Cl}; \Z)$ for $\vert r \vert \geqslant 1$ is contained in $B$.	This follows directly from the fact that $\pm \1$ are in $\Lambda$, which can either be seen by hand or by using the second claim in Proposition~\ref{prop:open_flux_obstructions}. 
	Actually, we can do slightly better in this particular case. Namely, we find that $r\1 \in B$ for all $r \in \R \setminus (-1,\frac{1}{n})$, meaning that, additionally to the above, the segment $r \in [\frac{1}{n},1)$ is obstructed, too. Indeed, taking $\Flux\Pi = r \1$ as in the proof of Proposition~\ref{prop:open_flux_obstructions} (with $\lambda = \1$), we find that $\1 \cdot \xi \in (-\frac{1}{r},0]$, meaning that $\1 \cdot \xi \in \{-(n-1), \ldots, -1,0\}$, since it is an integer. Hence there is no such $\xi$ with $\1 \cdot \xi \equiv 1 \mod (n+1)$. In particular, we deduce that the obvious Lagrangian isotopies which can be realized in the moment polytope $\Delta_{\CP^n}$ are sharp on the diagonal $\R \1 = \{(r,\ldots,r) \sth r \in \R \} \subset H^1(T_{\rm Cl} ; \R)$:
\begin{equation*}
	\Sh (T_{\rm Cl}) \cap \R \1 = \Delta_{\CP^n} \cap \R \1. \qedhere
\end{equation*}
\end{example}

We suspect that the second part of Proposition \ref{prop:open_flux_obstructions} captures all directions which can be obstructed by the above method, but were unable to prove it; see also \cref{q:cpnflux}.

\subsection{Flux in $\CP^2$}
\label{ssec:fluxcp2}
As stated in Theorem~\ref{thm:B}, the flux set of any Lagrangian torus in $\C^2$ is discrete. We now show that it is very much not the case when we move to $\CP^2$.
	
\begin{proposition} \label{prop:flux-set_CP^2}
	The flux set of a Lagrangian torus $L$ in $\CP^2$ is dense if and only if there is a Lagrangian isotopy $\Pi$ from $\tcl$ to $L$ such that $\Flux \Pi = (x,y)$ satisfies $\frac{x}{y} \notin \Q$. Moreover, whenever $\fl(L)$ is not dense, it is discrete.
\end{proposition}
	
Note that all $L$ with $\frac{x}{y}\notin\Q$ are irrational, i.e.\ $\omega(\pi_2(\CP^2,L))$ is dense in $\R$. However, there are irrational tori with $\frac{x}{y}\in\Q$, e.g.\ rational tori in the open ball $B^4$, seen as a subset of $\CP^2$. 

\begin{remark} \label{rem:flux-set_loc-connected_CP^2}
	By the classification of toric fibres in $\CP^2$ up to Hamiltonian isotopies \cite[Proposition~7.1]{SheTonVia19}, $\LHam(L)$ is $C^1$-locally path connected for all Lagrangian fibres $L=T_{(x,y)}\subseteq\CP^2$. Indeed, for any such $L$, there is a $C^1$-neighbourhood composed of Lagrangians $L'$ that are Hamiltonian isotopic (in that neighbourhood) to a toric fibre. But, by the classification, we can ensure that no toric fibre in a (potentially smaller) neighbourhood is Hamiltonian isotopic to $L$, hence local path connectedness. Note that this holds whether $\frac{x}{y}$ is rational or not. Therefore, there are some toric fibres for which $\fl(L)$ is dense, but where $\LHam(L)$ is locally path connected. This shows the stark contrast between isotopies with vanishing flux and Hamiltonian isotopies in the Lagrangian context.
\end{remark}
	
\begin{proof}[Proof of~\cref{prop:flux-set_CP^2}]
  We consider the group $\cm(L)$ of Lagrangian monodromies of $L$ and its intersection $H'=\langle R,M\rangle$ with $H=\SL(2;\Z)$. Note that when $\frac{x}{y}\in\Q$, then $\fl(L)$ is a subset of an affine lattice by the change-of-basepoint formula (\ref{eq:fm_base_point_change}), so that it is discrete. \par
  
  We thus suppose that $\frac{x}{y}\notin\Q$. Clearly, it suffices to prove the density of the subset $\fl_+(L)$ of fluxes associated with orientation-preserving Lagrangian loops. Recall that the flux of all loops based at $T_{\rm Cl}$ vanishes. Therefore, by the change-of-basepoint formula (\ref{eq:fm_base_point_change}), this set takes on the form
  \begin{align*}
    \fl_+(L)=\left\{(x,y)(A-id)\ \middle|\ A\in H'\right\}.
  \end{align*}
  Moreover, this set is dense if and only if the $H'$-orbit of $v=(x,y)$ is dense in $\R^2-\{0\}$. \par
  
  But note that $H_{\R}=\SL(2;\R)$ acts on $\R^2-\{0\}$ via right multiplication. Since every nonzero vector of $\R^2$ can be completed to a frame of oriented area 1, this action is transitive. Furthermore, the stabilizer at $(1,0)$ is the subgroup $T$ of lower unitriangular matrices. Therefore, we can identify $\R^2-\{0\}$ with the quotient under left-multiplication $T\backslash H_{\R}$, and it is easy to see that this identification is $H$-equivariant on the right. \par
  
  Note that the $H'$-orbit of $T\cdot g_v$ is dense if and only if for every $g$ there is a sequence $(A_k,B_k) \in H' \times T$ such that $B_kg_vA_k\to g$, which is itself equivalent to the $T$-orbit of $g_v\cdot H'$ being dense. But by the proof of Proposition~\ref{prop:monodromy_CP2}, we know that $H'$ has finite index in $H$. In particular, $H'$ defines a lattice in $H_{\R}$. By classical work of Hedlund~\cite{Hed36} on the horocylic flow, the $T$-orbits in $H_{\R}/H'$ are thus either periodic or dense. In fact, almost all orbits are dense. We are thus left with characterizing what precisely these dense orbits are. 
  
  From the proof of Proposition~\ref{prop:monodromy_CP2}, we know that $[H:H']=8$. Therefore, there is an 8-to-1 cover $H_{\R}/H'\to H_{\R}/H$ that is $T$-equivariant. But since covers are local homeomorphisms, this mean that, if $g'\cdot H$ is in the closure of $T(g\cdot H)$, then there are some lifts $\tilde{g}\cdot H'$ and $\tilde{g}'\cdot H'$ of $g\cdot H$ and $g'\cdot H$, respectively, such that $\tilde{g}'\cdot H'$ is in the closure of $T(\tilde{g}\cdot H')$. Moreover, if the original point was not in the orbit, then neither is its lift. From the above discusion, this means that the $H'$-orbit of $v$ is dense if and only if its $H$-orbit is. \par
  
  By irrational approximation, for every $\epsilon>0$, there are $a,b\in\Z$ such that $|ax+by|<\epsilon$. Without loss of generality, we may suppose that $\gcd(a,b)=1$. By Bézout's identity, there exist $c_0,d_0\in\Z$ such that $ad_0-bc_0=1$, and every other solution takes on the form $c=c_0+ka$ and $d=d_0+kb$ for some $k\in\Z$. In particular, we have that
  \begin{align*}
    (x\;\; y)\begin{pmatrix}
               a & c \\ b & d
             \end{pmatrix}=(ax+by\;\;\; cx+dy)=(ax+by\;\;\; c_0x+d_0y+k(ax+by)).
  \end{align*}
  On the one hand, the above matrix is in $H$ for any $k$ by construction. On the other hand, we can ensure that both coordinates on the ensuing vector are smaller than $\epsilon$ for an appropriate $k$, since $|ax+by|<\epsilon$. Therefore, 0 is not isolated in $H\cdot v$, and the orbit is dense.
\end{proof}
	
\begin{remark} \label{rem:flux-set_other}
	The proof mostly relies on the fact that $\cm (L)$ defines a sublattice in $\SL(2;\Z)$. For example, if \cref{qu:realizibility_ext-monodromy} has a positive answer in the monotone $S^2\times S^2$, then the same result will apply to Lagrangian tori in that symplectic manifold. Note that we already know that $\fl(L)$ is dense for some product tori in nonmonotone $S^2\times S^2$'s~\cite{BreKim23}.
\end{remark}

Let us now move to determining the values flux can take for Lagrangian isotopies that do not necessarily close up to loops. This is a continuation of the discussion from \S\ref{ssec:fluxCPn} for general $\CP^n$ to the case $n=2$. The difference is that, in that case, we can classify $H'$-orbits.

\begin{lemma}
	\label{lem:H_orbits_n=2}
	The action of $H'$ on the set of primitive vectors $\mathcal{P}^2$ of $\Z^2$ has three orbits:
	\begin{align*}
		\mathcal{O}_+ &\defi\left\{\lambda\in\mathcal{P}^2\ \middle|\ \lambda\equiv \1\!\!\mod 3 \right\}= \1\cdot H' \\
		\mathcal{O}_- &\defi\left\{\lambda\in\mathcal{P}^2\ \middle|\ \lambda\equiv -\1\!\!\mod 3 \right\}=-\1\cdot H' \\
		\mathcal{O}_R &\defi\mathcal{P}^2\setminus(\mathcal{O}_+\cup \mathcal{O}_-)=(1,0)\cdot H'
	\end{align*}
\end{lemma}

\begin{proof}
  Since the action of $H'$ preserves the set $\mathcal{O}_+$, the orbit of $\1$ is included in it. Moreover, if $\lambda\in\mathcal{P}^2$, then there is $A\in \GL(2;\Z)$ such that $\1 A=\lambda$. But if $\lambda\equiv\1\mod 3$, i.e.\ if $\lambda\in\mathcal{O}_+$, then $A\in H'$, which proves equality between the sets. The proof that $\mathcal{O}_-=-\1\cdot H'$ is completely analogous. \par
  
  To conclude that $\mathcal{O}_R$ consists of one orbit, note that if $v\in\mathcal{P}^2$ is integrally transverse to some $\lambda\in\mathcal{O}_\pm$, then there is some $A\in\GL(2;\Z)$ such that $(1,0)A=v$ and $\pm\1 A= \lambda$. This is because $\GL(2;\Z)$ acts transitively on integral bases. In particular, $A\in H'$, so that $v\in (1,0)\cdot H'$. Therefore, we only need to show that every vector in $\mathcal{O}_R$ is integrally transverse to some vector in $\mathcal{O}_{\pm}$.
  
  Let $v=(v_1,v_2)\in\mathcal{O}_R$. By Bézout's identity, every integrally transverse vector to $v$ is of the form $(\mp q,\pm p)+kv$ for some $k\in\Z$, where $p$ and $q$ form a solution to $pv_1+qv_2=1$. We show that there is a vector of this form in $\mathcal{O}_\pm$.
  
  First of all, note that $v\notin\mathcal{O}_\pm$ implies that, up to permutations, either $v_2\equiv 0\mod 3$ or $v\equiv(1,-1)\mod 3$. In the first case, $pv_1\equiv 1\mod 3$, so that $p,v_1\equiv\pm 1\mod 3$. In the positive case, $(-q,p)+kv\equiv (-q+k,1)\mod 3$, which is in $\mathcal{O}_+$ for $k=q+1$. In the negative case, $(-q,p)+kv\equiv (-q-k,-1)\mod 3$, which is in $\mathcal{O}_-$ for the same $k$. In the second case, $p-q\equiv 1\mod 3$, so that $(q,-p)+kv\equiv (q+k,-1-q-k)\mod 3$, which is in $\mathcal{O}_+$ for $k=-q+1$.
\end{proof}

\begin{remark}
	The situation in higher dimensions is more complex. In general, there are more orbits than the three appearing in the case $n=2$. For example, in the case $n=3$, the vectors $(1,0,0)$ and $(1,-1,1)$ are neither in the orbits $\1 \cdot H', -\1 \cdot H'$ nor are they in the same orbit. Indeed, since $2$ divides $n+1 = 4$, their mod-2 reductions, which are invariants of the $H'$-orbits, differ. 
\end{remark}

We can now determine most of $\Sh(T_{\rm Cl})$ for the Clifford torus $T_{\rm Cl} \subset \CP^2$.\smallskip

\begin{proof}[Proof of \cref{thm:possible-fluxes_CP2_intro}]
  As we have seen in the proof of \cref{prop:flux-set_CP^2}, the $H'$-orbit of $\eta \in H^1(T_{\rm Cl};\R) = \R^2$ is dense if and only if $\eta \notin \R \cdot \Z^2$. Therefore, there is $h \in H'$ such that $\eta h^{-1} = \eta_0 = (x_0,y_0) \in \Delta_{\CP^2}$ in that case. Since $H'$ is the Lagrangian monodromy group of $T_{\rm Cl}$, we can pick a Lagrangian loop $\Lambda$ with $\cm(\Lambda) = h$. Furthermore, let $\Pi$ be a Lagrangian isotopy through toric fibres which starts at $T_{\rm Cl}$ and ends at $T_{(x_0,y_0)}$. Then, $\Flux \Pi = \eta_0$, and we find that
  \begin{equation*}
    \Flux(\Lambda * \Pi)
    = \Flux \Lambda + \cm(\Lambda)^* \Flux \Pi
    = 0 + \eta_0 h = \eta.
  \end{equation*}
  This proves the claim for any irrational $\eta$. \par

  Now let $\eta \in \R \cdot \Z^2$, and decompose it as $\eta = r \lambda$ as above. As above, we can act by the Lagrangian monodromy group $H'$ on $\lambda$ to send it
  \begin{enumerate}
  \item to $(1,1)$ in case $\lambda \in \mathcal{O}_+$;
  \item to $(-1,-1)$ in case $\lambda \in \mathcal{O}_-$;
  \item to $(0,1)$ in case $\lambda \in \mathcal{O}_R$.
  \end{enumerate}
  The obstructions in the cases $\lambda \in \mathcal{O}_{\pm}$ follow from \cref{prop:open_flux_obstructions} and \cref{ex:flux_obstructions}. Note that $(1,0)$, $(-1,-1) \in \pp \Delta_{\CP^2}$, whereas, in the $(1,1)$ direction, the critical value is $(\frac{1}{2},\frac{1}{2}) \in \pp \Delta_{\CP^2}$.
\end{proof}

\section{Questions and outlook}
\commentintoc{We close the paper with a series of open questions related to Lagrangians and their flux-monodromy morphism.}

\medskip
A given symplectic isotopy can be deformed, by a homotopy preserving its endpoints, to a Hamiltonian isotopy if and only if its \emph{symplectic} flux vanishes. This is obviously not true in the Lagrangian case, as there is an isotopy from the Clifford to the Chekanov torus in $\C^2$ which has vanishing Lagrangian flux. 

\begin{question}
	Under which conditions can a Lagrangian isotopy be deformed into a Hamiltonian one?
\end{question} 

Since we have mostly considered loops in this paper, it may be appropriate to ask which conditions on $\Lambda \in \pi_1(\cl)$ ensures that it is deformable into a loop in $\cl \Ham$. See \cref{prop:deform_to_ham} for a complete answer for tori in $\C^2$. A related question is:

\begin{question}
	Can every loop $\Lambda \in \pi_1(\cl)$ with $\FM(\Lambda) = 0$ be deformed to a loop in $\cl \Ham$?
\end{question}

Here are a few questions surrounding Lagrangian tori in $\C^2$. \cref{thm:B} shows that the flux-monodromy group $\Gamma(L)$ of any Lagrangian torus is generated by the images of the loops $\Xi$ and $\Sigma$ (after a change of basepoint) as defined in Section \ref{sec:C2}. Although we have studied $\pi_1$ of the space of immersed Lagrangian tori in \S\ref{ssec:immersed}, the embedded case is open.

\begin{question}
	Is $\pi_1(\cl,\tcl)$, for the Clifford torus $\tcl \subset \C^2$, generated by $\Xi$ and $\Sigma$, as defined in Section \ref{sec:C2}?
\end{question}

As we have seen, preserving the Maslov class is the only obstruction for a given element in $H^1(L;\R) \rtimes \MCG(L)$ to be realized in $\Gamma(L)$, for $L \subset \C^2$ a Lagrangian torus. For general $L \subset X$, the Maslov class is an element $m_L \in H^2(X,L;\Z)$. 

\begin{question} \label{qu:realizibility_ext-monodromy}
  Let $\Psi \in \Aut H_2(X,L)$ preserve the Maslov class, $\Psi^*m_L = m_L$, and the subgroup of ambient classes $j(H_2(X))$. Does there exist $\Lambda \in \pi_1(\cl,L)$ that has extended monodromy $\widehat{\cm}{(\Lambda)} = \Psi$?
\end{question}

We define \emph{the extended flux-monodromy group} $\widehat{\Gamma}(L)$ of $L$ as the image of the extended flux-monodromy morphism $\widehat{\FM}:\pi_1(\cl)\to H^1(L;\R)\rtimes \MCG(X,L); \Lambda\mapsto (\Flux(\Lambda),\widehat{\cm}(\Lambda))$, where $\MCG(X,L)\defi \pi_0(\mathrm{Diff}(X,L))$ and $\mathrm{Diff}(X,L)\defi\{\phi\in \mathrm{Diff}(X)| \phi(L)=L\}$.
\begin{question}
	Is $\widehat{\Gamma}(L)$ discrete?
\end{question}
If $\del:H_2(X,L;\R)\to H_1(L;\R)$ is surjective, then yes: Just like in Lemma~\ref{prop:flux-monodromy_monotone}, (\ref{eq:areaformula}) implies that $\widehat{\Gamma}(L)$ is topologically isomorphic to $\widehat{\cm}(L)$. At the other end of the spectrum, if $\del=0$, then every Lagrangian isotopy comes from a symplectic one (\cref{prop:path-lag}). Therefore, the question becomes one of the relative topology of $\Symp(X,L)$, the group of symplectomorphisms fixing $L$, in the full group $\Symp(L)$, which goes beyond the scope of this paper. \par

In \S\ref{ssec:fluxCPn}, we have discussed the set of elements in $H^1(T_{\rm Cl};\R)$ that can be realized as the flux of a Lagrangian isotopy starting at the Clifford torus $T_{\rm Cl}$. In the case $n=2$, discussed in \S\ref{ssec:fluxcp2}, we have found (\cref{thm:possible-fluxes_CP2_intro}) a complete answer up to the following values in $H^1(T_{\rm Cl};\R)$. 

\begin{question}
	\label{q:cp2flux}
	Let $\eta \in H^1(T_{\rm Cl}; \R)$ be of the form
	\begin{equation*}
		\eta = r\lambda \text{ with } \lambda \in \mathcal{O}_R = (1,0) \cdot H' \text{ and } r \geqslant 0,
	\end{equation*}		
	where we have used the standard basis on $H_1(T_{\rm Cl})$, and where $H'$ is defined in (\ref{eq:Hprimesubgroup}). Is there a Lagrangian isotopy $\Pi$ with $\Pi_0 = T_{\rm Cl}$ and $\Flux \Pi = \eta$? 
\end{question}

If the answer is yes, then $\Pi$ does not come from the method of construction we have used, i.e.\ it cannot be written as the concatenation of a loop $\Lambda \in \pi_1(\cl,T_{\rm Cl})$ and an isotopy through toric fibres. If the answer is no, this should come from additional information about the disk with boundary on $T_{\rm Cl}$ constructed in \cite{CieMoh18}, e.g.\ a constraint to which class in $H_1(T_{\rm Cl})$ its boundary can lie. Since all of the vectors $\lambda \in \mathcal{O}_R$ are in the same $H'$-orbit, it suffices to settle the question for one such $\lambda$.

In the case $n \geqslant 3$, we have a less complete picture: On the one hand, all irrational vectors $\R^n \setminus (\R \cdot \Z^n)$ can be realized as flux. On the other hand, values of the form $\eta = r \lambda$ for $r \geqslant 0$ and $\lambda \in \Lambda$ (where $\Lambda$ is defined in \eqref{eq:Lambda_complement_lattice}) are obstructed. As in the case $n = 2$, we do not know whether vectors of the form $\eta = r \lambda$ for $r \geqslant 1$ and $\lambda \in \Z^n \setminus \Lambda$ are obstructed, meaning that we can ask the same question as \cref{q:cp2flux}. However, in the case $n \geqslant 3$, we do even not know that the complement of $\Lambda$ forms a single $H'$-orbit. We suspect this to be true; see \cref{q:cpnflux}. We cannot completely settle the question for $\eta = r\lambda$ with $0 < r < 1$, either. Of course, for small enough $r$, the vector $\eta$ lies in the moment polytope $\Delta_{\CP^n}$ and can thus be realized as flux. This question depends on how close to the moment polytope $\Delta_{\CP^n}$ we can map $\lambda$ by an element in $H'$. In other words, this again boils down to understanding the orbit strucuture of the action of $H'$ on $\Z^n$. 

\begin{question}
\label{q:cpnflux} Consider the orbits of the action of $H'$ on primitive vectors in $\Z^n$, for $n \geqslant 3$. Does the direct generalization of \cref{lem:H_orbits_n=2} hold, i.e.\ is it true that
	\begin{enumerate}
		\item $\Lambda$ is given by all vectors $\lambda \in H_1(T_{\rm Cl})$ that satisfy $\lambda \equiv \1\!\mod p$ for some $p$ dividing $n+1$?
		\item the complement $\cp \setminus \Lambda$ of $\Lambda$ among primitive vectors is given by the orbit $(1,0,\ldots,0) \cdot H'$?
	\end{enumerate}
\end{question}

As we have seen in \cref{cor:flux_monotone}, $\Flux \Lambda = 0$ for any loop $\Lambda \in \pi_1(\cl,L)$ based at an $H$-monotone Lagrangian $L$. For product tori in $\C^n$ and the Clifford torus in $\CP^n$, we have seen that any element in $\Sh(L)$ can be realized by the concatenation of a loop with a straight path. 

\begin{question}
	\label{q:concatenation}
	Let $L \subset (X,\omega)$ be an $H$-monotone Lagrangian. Can any $\eta \in \Sh(L)$ be written as $\Flux(\Lambda * \Pi) = \Flux \Pi \circ \cm(\Lambda)$, where $\Lambda \in \pi_1(\cl,L)$ and $\Pi$ is a Lagrangian isotopy whose flux is contained in a line segment in $H^1(L,\R)$ for all times?
\end{question}

Recall that the set $\Sh^*(L) \subset H^1(L;\R)$ of all fluxes that can be reached by a straight line is called the \emph{star shape}~\cite{SheTonVia19}. It was also studied in \cite{EntGanMem18}. This question asks whether every flux is a star flux up to base change by a Lagrangian monodromy element. 

Finally, we have seen that $\LHam(L)$ is $C^1$-locally path connected when it is a Lagrangian torus in $\C^2$ (\cref{thm:B}) or a toric fibre in $\CP^2$ (\cref{rem:flux-set_loc-connected_CP^2})~---~despite the fact that the flux set is often dense in the latter case. Moreover, the examples constructed in \cite{BreKim23} where $\LHam(L)$ is not $C^1$-locally path connected are in \emph{non-monotone} $S^2\times S^2$. This thus begs the question:
\begin{question}
	\label{q:4-H-monotone}
	Let $L$ be a Lagrangian torus in a 4-dimensional $H$-monotone symplectic manifold $X$. Must $\LHam(L)$ be $C^1$-locally path connected?
\end{question}

Recall that we know this is false when $\dim X\geqslant 6$~\cite{ChaLec24}.

\pagebreak
\appendix
\section{Paths of Hamiltonian isotopic Lagrangians}
\label{sec:paths}
\commentintoc{We prove some basic results characterizing when Lagrangian paths are generated by Hamiltonian isotopies. These results are known to experts, but are scattered throughout the litterature or folkloric~---~we simply gather them into one easy-to-cite place.}

The first appendix concerns a result that is known to experts but hard to find in the literature.

\begin{proposition} \label{prop:path-lag}
    Let $\Pi:[0,1]\to \cl$ be a smooth path from $L_0$ to $L_1$. The following statements are equivalent:
    \begin{enumerate}[label=(\alph*)]
    	\item $\Flux(\{\Pi(\tau)\}_{0\leqslant \tau\leqslant t})=0$, $\forall t$;
    	\item $\Pi(t)\in\LHam(L)$, $\forall t$;
    	\item there is some $H:[0,1]\times X\to\R$ such that $\Pi(t)=\phi^t_H(L_0)$, $\forall t$. 
    \end{enumerate}
    Moreover, if we only have that $\omega(H_2(X,\Pi(t)))$ is constant for all $t$, then we still have that $\Pi(t)=\psi_t(L_0)$ for some $\{\psi_t\}\in\Symp_{c,0}(X)$.  
\end{proposition}

By a \emph{smooth path} of Lagrangians, we mean that there is a smooth map $F:[0,1]\times L\to M$ such that $f_t=F(t,\cdot)$ parametrizes $\Pi(t)$. By the Whitney approximation theorem, every path in $\cl$ is homotopic relative to its endpoints to such a path.

To prove the statement, we need the following technical lemma.

\begin{lemma} \label{lem:local-to-global_exact}
    Suppose that $F:[0,1]\times L\to X$ has the property that, for every $t_0\in [0,1]$, there is some $\epsilon>0$ and some open $U$ of $X$ containing $f_t(L)$ with $\omega|_U=d\lambda$ and such that $f_t^*\lambda=dh_t$ for all $t\in (t_0-\epsilon,t_0+\epsilon)$. Then, there exists a Hamiltonian $H$ on $X$ such that $f_t(L)=\phi^t_H(L)$ for all $t\in [0,1]$.
\end{lemma}

\begin{proof}
  Consider the closed time-dependent 1-form $\alpha_t:=f_t^*(\iota_{v_t}\omega)$ on $L$, where $v_t:=\frac{df_t}{dt}$. By~\cite[Lemma~2.3]{AkvSal01}, it suffices to prove that $\alpha_t$ is exact for all $t\in [0,1]$. For some $t_0\in [0,1]$, take $\epsilon$, $U$, and $\lambda$ as above. Then, for every $t\in (t_0-\epsilon,t_0+\epsilon)$, a direct computation gives that
  \begin{align} \label{eq:alpha-primitive}
    \alpha_t=d\left(\frac{d}{dt}h_t-\lambda(v_t)\right),
  \end{align}
  where $f_t^*\lambda=dh_t$ as above. This proves the statement.
\end{proof}

\begin{proof}[Proof of Proposition~\ref{prop:path-lag}]
  It is trivial that \textit{(c)} implies \textit{(a)} and \textit{(b)}. Before proving the other implications, we make the following general observation. Let $F:[0,1]\times L\to X$ parametrize the path $\Pi$. For $t_0\in [0,1]$, we fix a Weinstein neighbourhood $U$ of $f_t(L)$. Then, by the smoothness of $F$, there is some $\epsilon>0$ such that $f_t(L)$ can be seen as the graph of a closed 1-form $\sigma_t$ of $L$ through that Weinstein neighborhood for all $t\in (t_0-\epsilon,t_0+\epsilon)$. \par

  Take a loop $\ell:S^1\to L$. Note that $\sigma_t(\ell)$ is precisely the area of the cylinder $C_t=\cup_{s\in [t_0,t]} \sigma_s(\ell)=\cup_{s\in [t_0,t]} f_s(\ell)$. Furthermore, the area of $C_t$ is continuous in $t$:
  \begin{align*}
    |\omega(C_t)-\omega(C_s)|=|\sigma_t(\ell)-\sigma_s(\ell)|\leqslant ||\sigma_t-\sigma_s||\ \mathrm{length}(\ell),
  \end{align*}
  where the norm and the length are computed using the same Riemannian metric on $L$. 
  
  Suppose now that we are in \textit{(a)}. Then, for $t\geqslant t_0$, we have that
  \begin{align*}
    \sigma_t(\ell)=\omega(C_t)=\Flux(\{\gamma(\tau)\}_{0\leqslant \tau\leqslant t})(\ell)-\Flux(\{\gamma(\tau)\}_{0\leqslant \tau\leqslant t_0})(\ell)=0.
  \end{align*}
  Likewise, if $t\leqslant t_0$, we have that
  \begin{align*}
    \sigma_t(\ell)=\omega(C_t)=-\left(\Flux(\{\gamma(\tau)\}_{0\leqslant \tau\leqslant t})(\ell)-\Flux(\{\gamma(\tau)\}_{0\leqslant \tau\leqslant t_0})(\ell)\right)=0.
  \end{align*}
  Going through all possible loops $\ell$, we conclude that $\sigma_t$ is exact for all $t\in (t_0-\epsilon,t_0+\epsilon)$. We then get \textit{(c)} from Lemma~\ref{lem:local-to-global_exact}.
  
  Suppose that we are in \textit{(b)}, i.e.\ $f_t(L)\in\LHam$. Then, there is some Hamiltonian isotopy $\{\phi^{t,t_0}_s\}_{s\in [0,1]}$ such that $\phi^{t,t_0}_1(f_t(L))=f_{t_0}(L)$. Therefore, $C'_t:=\cup_{s\in[0,1]}\phi^{t,t_0}_s(f_t(\ell))$ is a cylinder with vanishing area and one boundary agreeing with $C_t$, so that
  \begin{align*}
    \omega(C_t)=\omega(C_t\sharp C'_t)\in\omega(H_2(M,f_{t_0}(L))).
  \end{align*}
  But $\omega(H_2(M,L))$ is totally disconnected since $H_2(X,f_{t_0}(L))$ is finitely generated. Therefore, $\omega(C_t)$ must be constant in $t$, so that $\sigma_t(\ell)= 0$. We again get \textit{(c)} from Lemma~\ref{lem:local-to-global_exact}. \par

  Finally, we deal with the case when we only know that $\omega(H_2(X,f_t(L)))$ is constant. Suppose there is some $A\in H_2(X,f_{t_0}(L))$ such that $\del A=\ell$. Then, we again have that
  \begin{align*}
    \omega(C_t)=\omega(A\sharp C_t)-\omega(A)\in\omega(H_2(M,f_{t_0}(L))),
  \end{align*}
  so that $\sigma_t(\ell)= 0$ as in the above paragraph. But (\ref{eq:alpha-primitive}) becomes $\alpha_t=\frac{d}{dt}\sigma_t-d(\lambda(v_t))$ in the non-exact case. Therefore, $\alpha_t$ must vanish on the image of the boundary map $H_2(X,f_t(L))\to H_1(f_t(L))=H_1(L)$. By the long exact sequence in cohomology (with compact coefficients), this is equivalent to the existence of a compactly supported time-dependent closed 1-form $\theta_t$ on $X$ such that $[\alpha_t]=f_t^*[\theta_t]$. But then, there is a function $K:[0,1]\times L\to \R$ such that $\alpha_t=f_t^*\theta_t+dk_t$. It then suffices to pick a compactly supported function $\widetilde{K}:[0,1]\times X\to\R$ such that $f_t^*\widetilde{k}_t=k_t$ and to let $\{\psi_t\}$ be the symplectic isotopy generated by $\theta_t+d\widetilde{k}_t$.
\end{proof}

\begin{remark} \label{rem:surj_boundary-map}
    The proof of Proposition~\ref{prop:path-lag} implies that, whenever $H_1(L_0;\R)\to H_1(X;\R)$ vanishes, any Lagrangian isotopy in $\LSymp$ is a Hamiltonian isotopy. This had already appeared under the additional hypothesis that $\omega(H_2(X,L))$ be discrete in~\cite{ChaLec24}. \\
    At the other end of the spectrum, it also implies that every Lagrangian isotopy is generated by a symplectic one whenever $H_1(L_0;\R)\to H_1(X;\R)$ is injective.
  \end{remark}

\section{Lagrangian homotopies with vanishing flux}
\label{app:Lag_homotopies}
\commentintoc{We study when Lagrangian isotopies with vanishing flux are homotopic to exact \emph{homotopies}, an immersed variant of Hamiltonian isotopies. We give a complete answer by making use of $h$-principles.}

A symplectic isotopy with vanishing \emph{symplectic} flux is homotopic (relative to its endpoints) to a Hamiltonian isotopy. This classical fact is due to Banyaga \cite{Ban97}; see, for example, \cite[Theorem 10.2.5]{McDSal17}. The corresponding statement is false for Lagrangian isotopies: there is a Lagrangian isotopy between the Clifford and the Chekanov torus in $\C^2$ with vanishing flux. In this appendix, we discuss a Lagrangian analogue of this fact in the setting of Lagrangian \emph{immersions}. More precisely, we prove the following.

\begin{theorem} \label{thm:flux-0-to-exact}
	Let $\{f_t:Q\looparrowright X\}_{t\in [0,1]}$ be a Lagrangian homotopy. The total flux $\Flux(\{f_t\})$ vanishes if and only if $\{f_t\}$ is Lagrangian homotopic with fixed endpoints to an exact homotopy.
\end{theorem}

We recall that a \emph{Lagrangian homotopy} is a (smooth) 1-parameter family of Lagrangian immersions. A Lagrangian homotopy $\{f_t\}$ is \emph{exact} if the associated map $F:[0,1]\times Q\to X; (t,x)\mapsto f_t(x)$ has the property that $F^*\omega=dt\wedge dh_t$ for some smooth family $\{h_t:Q\to\R\}$. When all $f_t$'s are embeddings, exactness is equivalent to being Hamiltonian; see Lemma~\ref{lem:local-to-global_exact}. \par

The proof of Theorem~\ref{thm:flux-0-to-exact} relies on a clever use of the fibred version of the Gromov--Duchamp $h$-principle for Legendrian immersions~\cite{Gromov1970,Duc1984}. We thank Dominique Rathel-Fournier for pointing out that ingenious trick to us. \par

\begin{proof}
	Note that, for a Lagrangian homotopy $F=\{f_t\}$, we have that $F^*\omega=dt\wedge\beta_t$ for some time-dependent closed 1-form $\beta_t$ on $Q$. Then, a direct computation gives that (\ref{eq:fluxmap}) is equivalent to
	\begin{align} \label{eq:flux_1-form}
		\Flux(\{f_t\})=\int_0^1[\beta_t]dt.
	\end{align}
	In particular, if we set $\sigma_t:=\int_0^t\beta_sds$, then $\Flux(\{f_t\})=[\sigma_1]$. \par
	
	From this formalism, we directly see that exact homotopies have vanishing flux. Since the flux is invariant under homotopies, this shows one side of the implication. \par
	
	Suppose therefore that $F=\{f_t\}$ has vanishing flux, i.e.\ that $\sigma_1$ is exact. There exist $r>0$ and $\Psi=\{\psi_t\}:[0,1]\times D_r^*Q\to X$ such that $\psi_t^*\omega=\omega_0$ and $\psi_t\circ\iota=f_t$, where $\omega_0=d\lambda_0$ is the standard symplectic form on $T^*Q$ and $\iota:Q\hookrightarrow T^*Q$ is the inclusion of the 0-section. That $r>0$ may be chosen uniformly follows directly from the parametric version of Moser isotopy, which follows from the usual proof; see \cite[Lemma~3.2.1]{McDSal17}, for example. See also \cite{Yam25} for more explicit bounds on $r$ in terms of curvature in the K\"ahler setting~---~their approach should also give explicit, though uglier, bounds in the non-K\"ahler setting. \par
	
	Consider the 1-form $\lambda_t:=\lambda_0+\pi^*\sigma_t$ on $T^*Q$, where $\pi:T^*Q\to Q$ is the canonical projection. As noted above, the fact that $f_t$ is Lagrangian ensures that $\sigma_t$ is closed. Therefore, $\lambda_t$ is a primitive of $\omega_0$ for all $t$. Moreover,
	\begin{align*}
		\iota^*\lambda_1=\iota^*\lambda_0+\iota^*\pi^*\sigma_1=\sigma_1,
	\end{align*}
	so that $\iota$ is exact for both $\lambda_0$ and $\lambda_1$, since $[\sigma_1]=\Flux({f_t})=0$ by hypothesis. \par
	
	But an immersion $g:Q\looparrowright T^*Q$ is Lagrangian and $\lambda_t$-exact if and only if it admits a lift $\tilde{g}:Q\looparrowright T^*Q\times\R$ that is Legendrian for the contact form $\alpha_t:=e^x\lambda_t$, where $x$ denotes the coordinate in the $\R$ factor. Therefore, $\iota$ admits an $\alpha_i$-Legendrian lift $\tilde{g}_i$ to $T^*Q\times\R$ for $i=0,1$. Moreover, these lifts extend to a 1-parameter family of isotropic monomorphisms $\{\tilde{G}_t:TQ\to T(T^*Q\times \R)\}$, where $\tilde{G}_t$ is isotropic for the contact structure $\ker\alpha_t$. This is because the obvious projection $\ker\alpha_t\to TT^*Q$ is an isomorphism of vector bundles. By the fibred version of the $h$-principle for Legendrian immersions~---~see Theorem~24.1.3 in \cite{EliMis02} and the remark thereafter~---~this 1-parameter family is homotopic relative its endpoints to one coming from a homotopy $\{\tilde{g}_t:Q\looparrowright T^*Q\times\R\}$ where $\tilde{g}_t$ is $\alpha_t$-Legendrian. Projecting back onto $T^*Q$, this gives a homotopy relative to endpoints from the constant Lagrangian homotopy $\{\iota_t=\iota\}$ to a Lagrangian homotopy $\{g_t:Q\looparrowright T^*Q\}$ such that $g_t$ is $\lambda_t$-exact. Note that the $h$-principle for Legendrian immersions also holds in its $C^0$-dense version, so that we may assume that $g_t$ is $C^0$-close to $\iota$. In particular, we may assume it has image in $D^*_rL$. \par
	
	But then, $F'=\{f'_t:=\psi_t\circ g_t\}$ is a Lagrangian homotopy that is homotopic relative to its endpoints to $F$. Furthermore, $\{f'_s\}_{s\leq t}$ is homotopic to the concatenation $\{\psi_s\circ g_0\}_{s\leq t}\#\{\psi_t\circ g_s\}_{s\leq t}$; this is precisely the core of the proof of Claim~\ref{claim:path_homotopy_pcl}. Therefore, we have that
	\begin{align*}
		\Flux(\{f'_s\}_{s\leq t})
		&= \Flux(\{\psi_s\circ g_0\}_{s\leq t})+\Flux(\{\psi_t\circ g_s\}_{s\leq t}) \\
		&= \Flux(\{f_s\}_{s\leq t})+\Flux(\{g_s\}_{s\leq t}) \\
		&= [\sigma_t]+[g_t^*\lambda_0]-[g_0^*\lambda_0] \\
		&= [\sigma_t]+[g_t^*\lambda_t]-[g_t^*\pi^*\sigma_t] \\
		&= [\sigma_t]-[\sigma_t] \\
		&= 0.
	\end{align*}
	Here, the third line follows from Stokes' identity, while the fifth uses the fact that $\pi\circ g_t$ is homotopic to $\pi\circ g_0=id_L$. Therefore, $\{f'_t\}$ is an exact homotopy, which concludes the proof.
\end{proof}

\begin{remark} \label{rem:previous_h-principle}
	In a previous version of this appendix, we proved Theorem~\ref{thm:flux-0-to-exact} for \emph{$H$-rational} Lagrangian immersions, that is, those $f:Q\looparrowright X$ whose cone $\mathrm{Cone}(f)$ respects $\omega(H_\bullet(\mathrm{Cone}(f)))=\tau\Z$ for some $\tau\geq 0$. Our proof relied on an extension of Evans and K{\c{e}}dra's $h$-principle for monotone immersions in $\C^n$~\cite[Theorem~D]{EvaKed15}, which is covered by the above proof.
	
	Nonetheless, the central equation of that proof gives a nice interpretation of the Lagrangian flux, which could be of independent interest. Essentially, if $\{f_t\}$ is a Lagrangian homotopy and $E\to X$ is the prequantization bundle of $X$ with curvature form $-2\pi\omega$, $f_t^*E$ is flat for all $t$. In particular, the isomorphism class of $f_t^*E$ as a flat bundle is completely determined by its holonomy representation $\rho_t:H_1(Q;\Z)\to\mathrm{U}(1)$. The Lagrangian flux measures then precisely the change in holonomy:
	\begin{align} \label{eq:holomomy-rep_flux}
		\rho_t=\rho_0\cdot\left(\exp\circ \Flux(\{g_\tau\}_{0\leqslant \tau\leqslant t})\right).
	\end{align}
	This is a direct application of the Gauss--Bonnet theorem (see \cite[Lemma~7]{Ver20}), applied to the cylinders swept by $\{g_t\}$, whose area is precisely what the flux measures.
\end{remark}

\newpage
\bibliographystyle{alpha}
\bibliography{mybibfile}

\end{document}